\documentclass{article}
\usepackage{amsmath,mathrsfs,amssymb,amsthm,amscd}
\usepackage{color,epsfig,latexsym,graphicx,eucal,layout,fancyhdr}
\usepackage{enumerate}
\usepackage[]{fontenc}
\usepackage[all]{xy}
\usepackage[bookmarks=false]{hyperref}
\usepackage{graphics}
\usepackage{a4wide}
\usepackage{psfrag}

\newtheorem{thm}{Theorem}[section]
\newtheorem{lemma}[thm]{Lemma}
\newtheorem{prop}[thm]{Proposition}
\newtheorem{cor}[thm]{Corollary}
\newtheorem{defn}[thm]{Definition}

\newtheorem{rmk}[thm]{Remark}

\newcounter{EQNR}
\addtocounter{section}{-1} 
\numberwithin{equation}{section} 

\def\Empty{}
\newcommand\oplabel[1]{
  \def\OpArg{#1} \ifx \OpArg\Empty {} \else
    \label{#1}
  \fi}

\long\def\realfig#1#2#3#4{
\begin{figure}[tp]
\centerline{\psfig{figure=#2,width=#4}} \caption[#1]{#3}
\oplabel{#1}
\end{figure}}

\input{psbox}

\psfordvips

\begin{document}

\title{Heat kernel asymptotics on sequences of elliptically degenerating Riemann surfaces}
\author{D. Garbin and J. Jorgenson
}
\date{22 Feb 2016}
\maketitle

\begin{flushright}
\emph{Dedicated to Dennis Hejhal,\\ on the occasion of his 67th belated happy birthday}\\
\end{flushright}

\begin{abstract}\noindent
This is the first of two articles in which we define an elliptically degenerating 
family of hyperbolic Riemann surfaces and study the asymptotic behavior of the
associated spectral theory.   Our study is motivated by a result from \cite{He 83}, 
which Hejhal attributes to Selberg, proving spectral accumulation for the family of 
Hecke triangle groups.  In this article, we prove various results regarding
the asymptotic behavior of heat kernels and traces of heat kernels for both real and complex
time.  In \cite{GJ 16}, we will use the results from this article and 
study the asymptotic behavior of numerous spectral functions through elliptic degeneration, 
including spectral counting functions, Selberg's zeta function, Hurwitz-type zeta functions,
determinants of the Laplacian, wave kernels, spectral projections, small eigenfunctions,
and small eigenvalues.  The method of proof we employ follows the template set in previous
articles which study spectral theory on degenerating families of finite volume Riemann 
surfaces (\cite{HJL 95}, \cite{HJL 97}, \cite{JoLu 97a}, and \cite{JoLu 97b}) and on degenerating
families of finite volume hyperbolic three manifolds (\cite{DJ 98}).  Although the types of results
developed here and in \cite{GJ 16} are similar to those in existing articles, it is necessary to
thoroughly present all details in the setting of elliptic degeneration in order to uncover all
nuances in this setting.  
\end{abstract}

\tableofcontents

\newpage

\section{Introduction}
\label{Introduction}
\hskip 0.2in
The last result stated by Hejhal in his monumental work \cite{He 83} is a qualitative
theorem, which he attributes to Selberg, concerning spectral accumulation at the bottom
of the range of the continuous spectrum for the family of Hecke triangle groups.  In brief, the present
article is the first in a series of two papers in which we further generalize and quantify the
Hejhal-Selberg theorem.  Before describing our work, let us discuss the theorem
from page 579 of \cite{He 83} which serves as motivation for our study.

\vskip .10in
\hskip 0.2in
For any integer $N \geq 3$, consider the two matrices
\begin{align*}
\left(\begin{matrix}0 & -1 \\ 1 & 0 \end{matrix}\right)
\,\,\,\,\,\textrm{and}\,\,\,\,\,
\left(\begin{matrix} 1 & 2 \cos(\pi/N) \\ 0 & 1 \end{matrix}\right).
\end{align*}
These matrices generate a discrete subgroup of $\text{\rm PSL}(2,\mathbb R)$, denoted
by $G_{N}$, and called the Hecke triangle group.  The group $G_{N}$ acts on
the upper half plane $\mathbb H$, and through elementary considerations one can show that
the set
\begin{align*}
\{z \in \mathbb H \,\,:\,\, \vert z \vert > 1 \,\,\,\,\,\textrm{and}\,\,\,\,\,|\textrm{Re}(z)| < \cos(\pi/N)\}
\end{align*}
is a fundamental domain for $G_{N} \backslash \mathbb H$. 

\vskip .10in
\hskip 0.2in
Let $\Delta$ denote the Laplacian associated to the hyperbolic metric on $\mathbb H$.  We will assume
knowledge of the spectral theory of the Laplacian acting on smooth functions on $G_{N} \backslash \mathbb H$,
referring the reader to \cite{He 83} and references therein for background material.  In particular, we
use that for each $N$, Weyl's law associated to the Laplacian consists of counting discrete eigenvalues 
of the Laplacian as well as poles of the determinant $\varphi_{N}(s)$ of the scattering matrix which is computed 
from the constant terms in the Fourier expansion of a non-holomorphic (parabolic) Eisenstein series. With this 
albeit brief description of background material, we can now state the Hejhal-Selberg result.

\vskip .10in
\hskip 0.2in
\emph{For any $t_{0} \in \mathbb R$ and $\delta > 0$, there is an $N_{0}$ such that if $N > N_{0}$, the
set $[1/2-\delta, 1/2) \times [t_{0}-\delta, t_{0} + \delta]$ will contain a pole of $\varphi_{N}$.  
}

\vskip .10in
\noindent
In words, the poles of $\varphi_{N}$ are densely accumulating along the line $\textrm{Re}(s) = 1/2$.  

\vskip .10in
\hskip 0.2in
At this time, there is a number of natural questions which one can pose.  
Firstly, can one further quantify the Hejhal-Selberg result beyond the assertion
regarding accumulation of poles of $\varphi_{N}$?  Observe that the above
stated fundamental domains converge, as $N$ tends to infinity, to a region which is a fundamental 
domain for the discrete group generated by
\begin{align*}
\left(\begin{matrix}0 & -1 \\ 1 & 0 \end{matrix}\right)
\,\,\,\,\,\textrm{and}\,\,\,\,\,
\left(\begin{matrix} 1 & 2  \\ 0 & 1 \end{matrix}\right).
\end{align*}
Does the asymptotic behavior of the spectral theory of $G_{N}\backslash \mathbb H$, 
in whatever form, converge to the spectral theory of the limiting fundamental domain?  Finally,
is the Hejhal-Selberg theorem indicative of a general phenomena which exists for other sequences
of geometric objects?  

\vskip .10in
\hskip 0.2in
The purpose of this article and the subsequent paper \cite{GJ 16} is to explore in detail 
these questions.  We begin, in this article, by defining an elliptically degenerating family
of hyperbolic Riemann surfaces, one example of which is the sequence of Hecke triangle groups.
In effect, an elliptically degenerating sequence is obtained by choosing a fixed hyperbolic Riemann
surface $M$ with a finite set of points $P_{1}, \cdots, P_{q}$, and then we change the local
coordinate at each point $P_{j}$ from a complex variable $z$ to $z^{1/n_{j}}$, and then we let each
$n_{j}$ approach infinity.  

\vskip .10in
\hskip 0.2in
Having established the general setting which we will study, we then focus our attention to 
the heat kernel associated to the hyperbolic Laplacian $\Delta$ which acts on smooth functions on
the underlying surface.  In the present article, we prove a number of results corresponding to the
asymptotic behavior of the heat kernel, and traces of the heat kernel, on an elliptically degenerating family
of Riemann surfaces.  In addition to pointwise convergence results, we study the asymptotic
behavior of the trace of the heat kernel for small time, large time, and complex time, and in
each case we establish asymptotic expansions in the degenerating parameters with attention paid to uniformity
issues.  In the present article, we slightly deviate from the heat kernel itself by proving convergence
of the small eigenvalues and small eigenvalues, and we do so in order to strengthen the convergence
results we derive.

\vskip .10in
\hskip 0.2in
In the subsequent article \cite{GJ 16}, we use the heat kernel convergence results proved here to
study the asymptotic behavior of numerous spectral functions through elliptic degeneration,
including spectral counting functions, Selberg's zeta function, Hurwitz-type zeta functions,
determinants of the Laplacian, wave kernels, spectral projections, small eigenfunctions,
and small eigenvalues.  At that time, a corollary of our more general theorem will be a quantitative
version, with error term, of the Hejhal-Selberg result, thus answering one of the questions posed above.

\vskip .10in
\hskip 0.2in
The method of proof we employ follows the template set in previous
articles which study spectral theory on degenerating families of finite volume Riemann
surfaces (\cite{HJL 95}, \cite{HJL 97}, \cite{JoLu 97a}, and \cite{JoLu 97b}) and on degenerating
families of finite volume hyperbolic three manifolds (\cite{DJ 98}).  Although the types of results
developed here and in \cite{GJ 16} are similar to those in existing articles, it is necessary to
thoroughly present all details in the setting of elliptic degeneration in order to uncover all
nuances in this setting.  Additionally, we did not believe it would be ``mathematically honest'' to
simply assert that the methodology of these articles applies in the setting of elliptically degenerating
surfaces, so in addition to proving a translation of the method to the setting of elliptically degenerating
surfaces, we felt it necessary to provide verification of details.

\vskip .10in
\hskip 0.2in
Aside from its interest within the field of spectral analysis, the problem of studying elliptically 
degenerating Riemann surfaces has manifested itself elsewhere.  In \cite{vP 10}, von Pippich defined and
studied elliptic Eisenstein series, ultimately proving analogues of the classical theorems, namely differential
equation, meromorphic continuation, and Kronecker limit formula.  In \cite{GvP 09}, the authors studied
elliptic Eisenstein series through elliptic degeneration.  One of their main result was to prove that 
certain elliptic Eisenstein series, when rescaled, converge to parabolic Eisenstein series through elliptic
degeneration.  We refer to \cite{vP 10} and \cite{GvP 09} for more precise statements and proofs.  

\vskip .10in
\hskip 0.2in
Recently, Freixas i Montplet and von Pippich have undertaken a fascinating project involving elliptic degeneration.  Let $\mathcal{M}_{g,n}$ denote the moduli space of genus $g$ hyperbolic Riemann surfaces with $n$ marked points. 
The Takhtajan--Zograf form $\omega_{\mathrm{TZ}}$ on $\mathcal{M}_{g,n}$ is constructed using (parabolic)
Eisenstein series associated to the marked points. The relation of $\omega_{\mathrm{TZ}}$ to the Weil--Petersson 
form $\omega_{\mathrm{WP}}$ is given by a local index theorem which can be derived from an arithmetic
Riemann--Roch isometry.  In \cite{FvP 11}, the authors are studying  an analogue of $\omega_{\mathrm{TZ}}$ on the moduli space of Riemann surfaces with $n$ marked weighted points using  elliptic Eisenstein series. Initially, 
they are striving towards an arithmetic Riemann--Roch isometry, the computation of the relative determinant of the
Laplacian on an hyperbolic cone.  Having established the isometry, in the sense of Arakelov theory, the authors
plan to work toward showing a ``consistency'' with their work and previously established isometries by studying
their identity through elliptic degeneration.  Consequently, the significance of the analysis in the present
paper goes beyond the applications we develop in \cite{GJ 16}.  

\vskip .10in
\hskip 0.2in
The project of defining and studying hyperbolic spectral theory through elliptic degeneration was initiated in the early 1990's by
the second named author (JJ) when he was collaborating with R. Lundelius.  Indeed, the project was referred
to as ``in preparation'' in several previous publications.  Unfortunately, the distinction of ``in preparation''
was very premature.  By 1995, Lundelius left academic mathematics, and by 1997, Jorgenson had completed all aspects of papers which were published as jointly written with Lundelius.  Approximately ten years later, 
the first named author (DG) developed the question of elliptic degeneration as part of his own investigations.
The present article as well as \cite{GJ 16} are the product of the consequent collaboration of the two authors
of the articles.

\section{Geometry of elliptic degeneration}
\label{Geometry of elliptic degeneration}

\hskip 0.2in In this section we will present the notion of elliptic degeneration of hyperbolic Riemann surfaces. We will consider finite volume surfaces (compact or non-compact), having elliptic fixed points of finite order, i.e. surfaces with conical ends. Elliptic degeneration occurs when the orders of such elliptic fixed points are increasing without a bound: As these orders are running off to infinity, their corresponding cone angles approach zero. Alternatively, as elliptic elements of the fundamental group become parabolic, their corresponding conical ends turn into cusps.  

\hskip 0.2in Let $M$ be a connected hyperbolic Riemann surface of finite volume, either compact or non-compact. For simplicity, let us assume that $M$ is connected, so then $M$ can be realized as the quotient manifold $\Gamma \backslash \mathbb H$, where $\mathbb H$ is the hyperbolic upper half space and $\Gamma$ is a discrete subgroup of $\textrm{SL}(2,\mathbb R) \slash \{ \pm1 \}$.

\hskip 0.2in Aside from the identity, the elements of $\Gamma$ can be classified into three classes, according to the type of their fixed points when viewed as fractional linear transformations or equivalently to the value of their absolute trace if viewed as matrices.  An element $\gamma \in \Gamma$ is called hyperbolic, parabolic, or elliptic,  if $\gamma$ is conjugated in $\textrm{SL}(2,\mathbb R)$ to a dilation, horizontal translation,  or rotation respectively. This is analogous to $|\textrm{Tr}(\gamma)|$ being greater than, equal, or less than 2, respectively.  Furthermore, an element $\gamma$ is called primitive, if it is not a power other than $\pm1$ of any other element of the group. With this in mind, a primitive hyperbolic element $\gamma$ is conjugated to $\left(\begin{array}{cc}
e^{\ell_{\gamma}/2} & 0 \\ 0&e^{-\ell_{\gamma}/2} \end{array}\right)$, where $\ell_{\gamma}$ is the length of the simple closed geodesic on the surface $M$ in the homotopy class of $\gamma$. 
A primitive parabolic element $\gamma$ is conjugated to  $\left(\begin{array}{cc} 1 & w_{\gamma} \\ 0& 1 \end{array}\right)$, where $w_{\gamma}$ denotes the width of the cusp fixed by $\gamma$, while a primitive elliptic element  $\gamma$ is conjugated to 
$\left(\begin{array}{cc}
\cos(\pi/q_{\gamma}) & \sin(\pi/q_{\gamma}) \\ -\sin(\pi/q_{\gamma})&\cos(\pi/q_{\gamma}) \end{array}\right)$, where $2\pi/q_{\gamma}$ is the angle of the conical point fixed by $\gamma$. The positive integer $q_{\gamma}$ is the order of the centralizer subgroup of the elliptic element $\gamma$. We will say that the corresponding elliptic fixed point has order $q_{\gamma}.$

\realfig{singlecone}{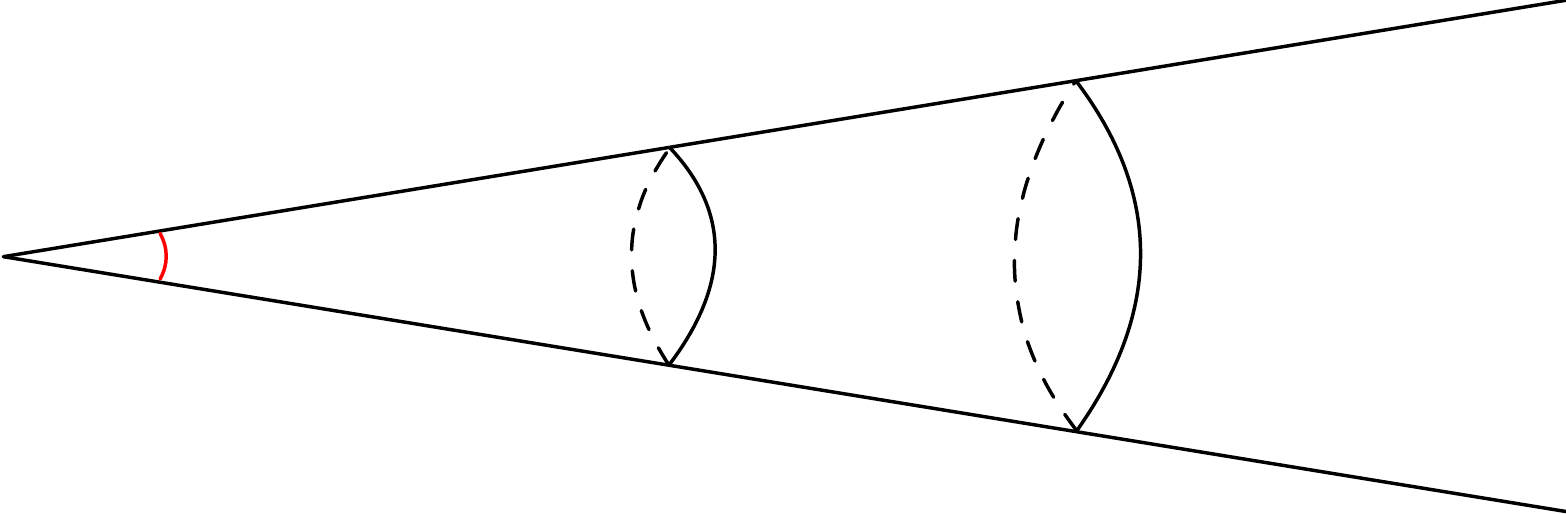}{The geometry of the hyperbolic cone $C_q$ 
$$ {
 \at{-8\pscm}{2.8\pscm}{$\alpha = \frac{2\pi}{q}$}
 \at{-4.8\pscm}{3.3\pscm}{$\rho_1$} 
 \at{-2.5\pscm}{3.7\pscm}{$\rho_2$} 
}
$$}{9cm}

\hskip .20in For a given positive integer $q$, let $C_q$ denote
the infinite hyperbolic cone of angle $2\pi/q$. One can realize $C_q$ 
as a half-infinite cylinder
\begin{equation}\label{hyperbolicpolarcoordinates}
C_{q} = \{ (\rho, \theta): \rho > 0, \theta \in [0,2\pi)\}\,\,.
\end{equation}
equipped with a smooth metric such that the length of meridians of constant $\rho$, which we denote here by $\ell(\rho)$, goes to zero as $\rho$ approaches zero (see Fig. \ref{singlecone}). Consider the Riemannian metric given by
\begin{equation}
ds^2 = d\rho^2 + q^{-2} \sinh^2 (\rho) d\theta^2\,\,.
\end{equation}
It easily follows that $\ell(\rho) = \cfrac{2\pi}{q}\sinh(\rho)\,\,.$ The notion of angle is defined as the rate of change in the length of the circle with respect to the change in distance near $\rho=0$, namely 
\begin{align*}
\lim_{\rho_1 \to 0} \left[\lim_{\rho_2 \to \rho_1} \frac{\ell(\rho_2)-\ell(\rho_1)}{d(\rho_2,\rho_1)}\right]\,\,,
\end{align*}
provided that the limit exists. That $d(\rho_2,\rho_1) = |\rho_2-\rho_1|$, implies that $C_{q}$ is an infinite hyperbolic cone of angle $2\pi/q$ with apex at $\rho=0.$
It is elementary to show that the volume form of the manifold $C_{q}$ is given by 
\begin{equation}
d\mu = q^{-1} \sinh (\rho) d\rho d\theta.
\end{equation}

\hskip 0.2in A fundamental domain for $C_{q}$ in the hyperbolic unit disc model is
provided by a sector with vertex at the origin and with angle
$2\pi/q$. In coordinates, we write $\{ \alpha \exp(i\phi): 0 \le
\alpha < 1, 0 \le \phi < 2\pi/q\}$. The hyperbolic metric on $C_{q}$
is the metric induced onto the fundamental domain viewed as a subset of
the unit disc endowed with its complete hyperbolic metric.  The
isotropy subgroup which corresponds to this fundamental domain
consists of the numbers $\exp(2\pi i k/q) $ for $k = 1,2, \cdots ,q$
acting by multiplication. A direct computation shows if $z = (\rho,
\theta)$ and $\gamma_{q}$ denotes a generator of the fundamental group
of $C_{q}$, then

\begin{equation}\label{beardon}
\cosh d(z,\gamma_{q}^n z) = 1+ 2 \sin^2(\pi n /q)\sinh^2 (\rho).
\end{equation}

\hskip .20in Let $C_{q,\varepsilon}$ denote the submanifold of
$C_{q}$ obtained by restricting the first coordinate of $(\rho,
\theta)$ to $0 \le  \rho < \cosh^{-1}(1+\varepsilon q/2\pi)$. A
fundamental domain for $C_{q,\varepsilon}$ in the unit disc model is
obtained by adding the restriction that
$\alpha < (\varepsilon q/(4\pi+\varepsilon q))^{1/2}.$ An elementary
calculation shows that the volume of this manifold
$\textrm{vol}(C_{q,\varepsilon}) = \varepsilon$, and the length of
the boundary of $C_{q,\varepsilon}$ is $(4\pi \varepsilon/q +
\varepsilon^2)^{1/2}$. For $\varepsilon_1 < \varepsilon_2$ one can
show that the length between the boundaries of the two nested cones
$C_{q,\varepsilon_1}$ and $C_{q,\varepsilon_2}$ is
\begin{equation*}
d_{\mathbb{H}}(\partial C_{q,\varepsilon_1}, \partial
C_{q,\varepsilon_2}) = \log \Bigg(\frac{\varepsilon_2 q + 2\pi +
\sqrt{\varepsilon_2 q(4\pi+\varepsilon_2q)}} {\varepsilon_1 q + 2\pi
+ \sqrt{\varepsilon_1 q(4\pi+\varepsilon_1q)}}\Bigg).
\end{equation*}

\hskip .20in Let $C_{\infty}$ denote an infinite cusp. A
fundamental domain for $C_{\infty}$ in the upper half plane is given
by the set $\{x+iy : y > 0, 0 < x < 1\}$. A fundamental domain for $C_{\infty}$
in the upper half plane is obtained
by identifying the boundary points $iy$
with $1+iy$. The isotropy subgroup that corresponds to the above
fundamental domain consists of $\mathbb{Z}$ acting by addition.
As before, let $C_{\infty,\varepsilon}$ denote the submanifold of
$C_{\infty}$ obtained by restricting the $y$ coordinate of the
fundamental domain given above to $y > 2\varepsilon$. Easy
computations show that $\textrm{vol} (C_{\infty,\varepsilon}) =
\varepsilon/2$, and the length of the boundary of
$C_{\infty,\varepsilon}$ is also $\varepsilon/2$.

\realfig{doublecones}{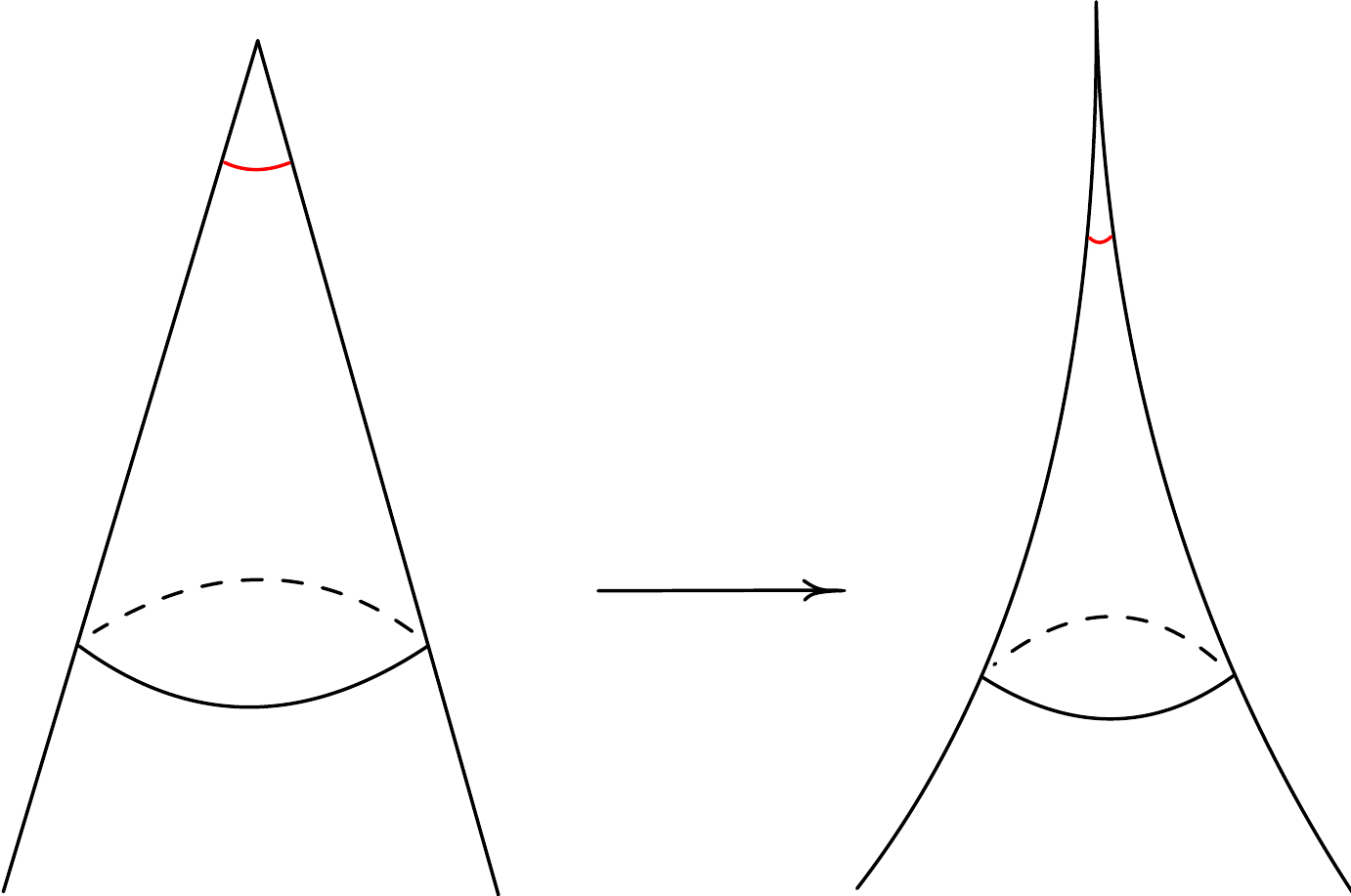}{The elliptic degeneration of a hyperbolic cone into a cusp 
$$ {
 \at{-7.3\pscm}{5.6\pscm}{$\alpha = \frac{2\pi}{q}$}
 \at{-3.5\pscm}{5.6\pscm}{$\alpha = 0$} 
 \at{-5.6\pscm}{3.3\pscm}{$q \to \infty$} 
 \at{-9.4\pscm}{4.2\pscm}{$C_q$} 
 \at{-2.2\pscm}{4.2\pscm}{$C_{\infty}$} 
}
$$}{9cm}

\hskip 0.2in Following Section 2 of \cite{Ju 98}, we are now ready to define the concept of elliptic degeneration. In its quintessential form, elliptic degeneration turns a cone of finite order $q$ into a cone of infinite order, i.e. a cusp (see Fig. \ref{doublecones}). To view this, we realize the positive angle cone $C_q$ as the half-infinite cylinder $\{ (x,y) : x\in [0,1), y\in (0,\infty)\}$, by changing the $(\rho,\theta)$ coordinates in (\ref{hyperbolicpolarcoordinates}) as $\theta = 2\pi x$ and $\rho = 2 \tanh^{-1}(e^{-\alpha y})$, where $\alpha = 2\pi /q.$ In $(x,y)$ coordinates, $C_q$ is a cone of angle $\alpha = 2\pi/q$ with apex at $y=\infty$, equipped with the Riemannian metric 
\begin{align*}
ds_{q}^{2} = \frac{dx^2+dy^2}{\alpha^{-2}\sinh^2(\alpha y)}\,\,.
\end{align*}
As the order $q$ goes to infinity, or equivalently as the angle $\alpha$ goes to zero, the cone $C_q$ turns into the cusp $C_{\infty}$ with metric given by 
\begin{align*}
ds_{\infty}^{2} = \frac{dx^2+dy^2}{y^2}\,\,.
\end{align*}

\hskip .20in To turn several cones into cusps, we proceed as follows. Let $q = (q_1,q_2, \cdots, q_m)$ be a
vector of the orders of elliptic fixed points. In this case we
define $C_{q} = \cup_{k=1}^{m} C_{q_k}$. We similarly define $C_{q,
\varepsilon}$ as a union over the components of
$q$. We say that the vector $q$ approaches infinity if and
only if each of its components approach infinity. 
Consequently, the Riemannian manifold $C_{q}$ (with $m$ connected
components) converges to $m$ copies of the limit Riemannian manifold
$C_{\infty}$ as $q \to \infty$. Similarly, $C_{q,
\varepsilon}$ converges to $m$ copies of $C_{\infty,\varepsilon}$.
We shall write these limits as $m \times C_{\infty}$ and $m \times
C_{\infty,\varepsilon}$.

\hskip .20in With these in mind, let us make the following definition.
\begin{defn}
A family of finite volume hyperbolic surfaces
$M_{q}$ parametrized by the $m$-vector $q$ will be called
an elliptically degenerating surface if it has the following
properties (see Fig. \ref{ed4}):
\begin{enumerate}[a)]
\item For any $\varepsilon < 1/2$, the surface $C_{q, \varepsilon}$ (with $m$ components) embeds
isometrically into $M_{q}$.
\item As $q \to \infty$, $M_{q}$ converges to a complete, hyperbolic surface
$M_{\infty}$ in the following sense. The surface $M_{\infty}$ contains $m$
embedded copies of $C_{\infty,\varepsilon}$ which is the limit of
$C_{q, \varepsilon} \subset M_{q}$. The geometry of
$M_{q} \backslash C_{q, \varepsilon}$ converges to the
geometry of $M_{\infty} \backslash (m \times
C_{\infty,\varepsilon})$.
\end{enumerate}
\end{defn}

\begin{rmk} 
\emph{In the above definition, $m \times
C_{\infty,\varepsilon}$ refers to the ``new'' cusps of $M_{\infty}$,
that is, the cusps which developed from degeneration. In particular,
for every $q$, it is possible to identify points $x(q)$ and $y(q)$ on
$M_{q} \backslash C_{q, \varepsilon}$ such that
$\lim_{q \to \infty} d_{q}(x(q), y(q)) =
d_{\infty}(x(\infty), y(\infty))$. Henceforth, we shall suppress the
$q$ dependence of points which are identified during
degeneration and simply write $x$ and $y$. The volume forms induced
by the converging metrics also converge uniformly on $M_{q}
\backslash C_{q, \varepsilon}$, and all such measures are
absolutely continuous with respect to each other. In general, the
hyperbolic volume form occurring in an integral will be denoted by
$d\mu$ with an appropriate subscript when needed (for example,
$d\mu_{q}$). Length measure will be denoted by $d\rho$.
}
\end{rmk}

\hskip .20in The description of the degeneration of $M_{q}$ to
the limit surface $M_{\infty}$ also applies to the degeneration of
$C_{q}$ and $C_{q, \delta}$ (with $\varepsilon < \delta$)
to their limit surfaces, $m \times C_{\infty}$ and $m \times
C_{\infty, \delta}$ respectively.


\realfig{ed4}{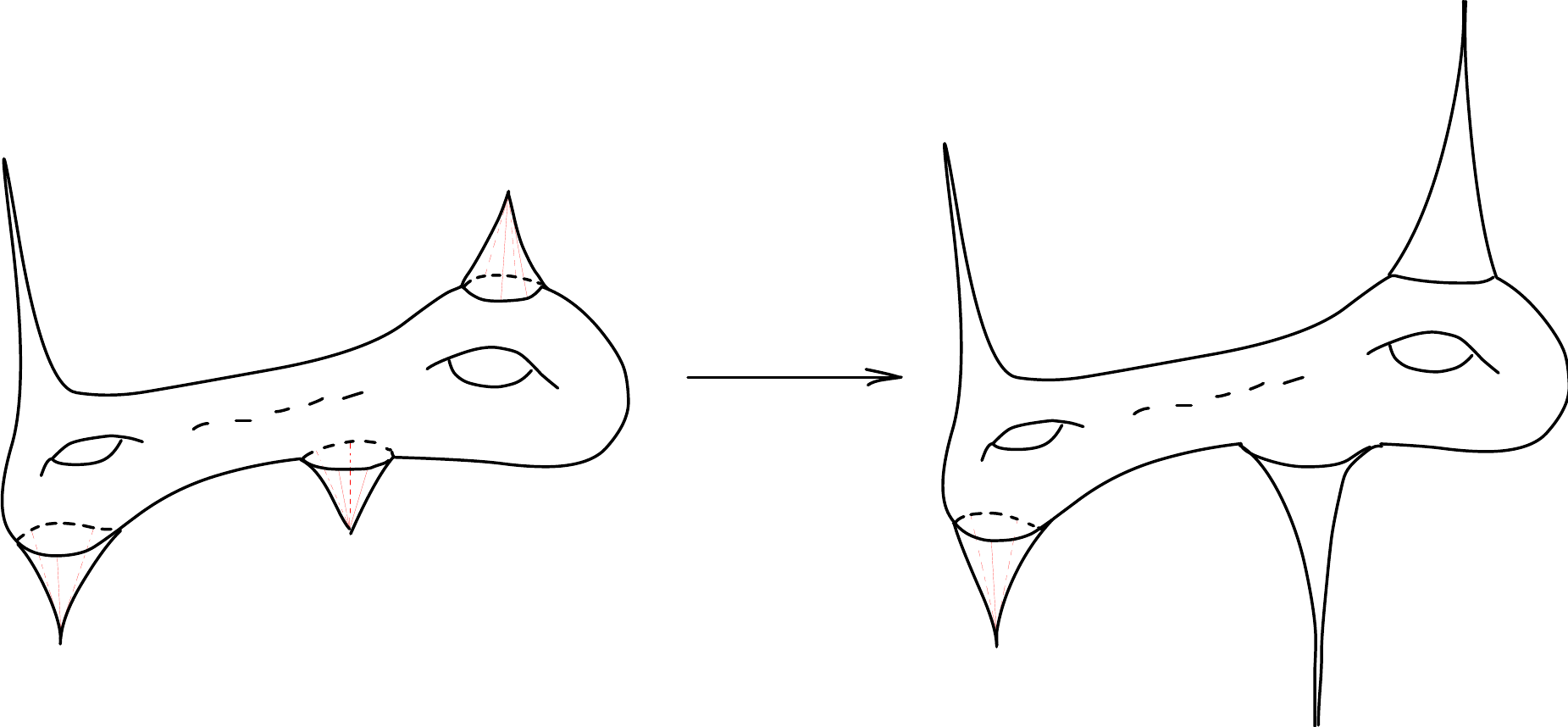}{Elliptic degeneration of
$q_1$ and $q_2$ $$ {
 \at{-5.83\pscm}{1.5\pscm}{$\frac{2\pi}{q_1}$}
 \at{-5.1\pscm}{4.2\pscm}{$\frac{2\pi}{q_2}$}
 \at{-7.73\pscm}{0.8\pscm}{$\frac{2\pi}{q_3}$}
 \at{-4.15\pscm}{3\pscm}{$q \to \infty$}
 \at{-8.8\pscm}{3\pscm}{$M_{q}$}
 \at{1\pscm}{3\pscm}{$M_{\infty}$}
}
$$}{9cm}

\hskip 0.2in The next result is the main theorem of \cite{Ju 98}. 

\begin{prop}\label{Judge}
Let $M$ be a surface with Euler characteristic $\chi(M)$, having one conical neighborhood $E$. Let $[g]$ be a pointwise conformal class of hyperbolic metrics on $M$. For each number $\alpha \in \big[0,2\pi(1-\chi(M))\big)$, let $g_{\alpha}$ be the unique hyperbolic metric on $M$ such that $(E,g_{\alpha})$ is a hyperbolic cone of angle $\alpha.$ Define the conformal factor $w_{\alpha}:M \mapsto \mathbb{R}$ where $g_{\alpha}=e^{w_{\alpha}}g_0.$ Then the map $\alpha \mapsto w_{\alpha} \in C^{\infty}(M)$ is real analytic.
\end{prop}

\emph{Idea of proof.} 
Consider hyperbolic metrics $g$ on $M$ for which $(E,g)$ embeds into $(\mathbb{S}^1 \times (0,\infty),m_{\alpha})$, where  the Riemannian metric is given by
\begin{align*}
m_{\alpha} = \frac{d\theta^2+dy^2}{\alpha^{-2}\sinh^2(\alpha y)}\,\,,
\end{align*}
that is an infinite volume hyperbolic cone of angle $\alpha.$ 

\hskip 0.2in By the Gauss-Bonnet formula and the negativity of the curvature, it follows that $\alpha -2\pi +2\pi\chi(M) <0\,\,.$ In \cite{McO 88}, the author shows that this inequality implies the existence of such a metric $g$ which is determined by its conformal structure and the angle $\alpha$. This gives a bijection between the interval $\big[0,2\pi(1-\chi(M))\big)$ and the set of conical hyperbolic metrics in a given conformal class, by sending $\alpha$ to $g_{\alpha}\,\,.$

\hskip 0.2in Fix an angle $\alpha_0$ in the above interval. Consider the conformal factor $w_{\alpha}:M \mapsto \mathbb{R}\,\,,$ satisfying $g_{\alpha} = e^{2w_{\alpha}} g_{\alpha_0}\,\,.$ The map $w_{\alpha}$ is then part of the zero set of a carefully chosen elliptic second order differential operator between Banach spaces whose Frechet derivative has bounded inverse. The implicit function theorem between Banach spaces is the used to show that the map $\alpha \mapsto w_{\alpha}$ is real analytic.  
\mbox{}\qed

\begin{rmk}
\emph{
The above proposition states that given a finite volume hyperbolic surface $M_{\infty}$ with $p$ cusps, there exists a family of hyperbolic surfaces $\{M_{q}\}$, with $p - m$ cusps indexed by the $m$-tuple
$q$ such that
\begin{equation*}
\lim_{q \to \infty } M_{q} = M_{\infty}
\end{equation*}}
\end{rmk}

\section{Regularized heat traces}
\label{Regularized heat traces}

\hskip 0.20in In this section we establish integral representations
for the hyperbolic and elliptic heat traces for complex valued time.
We then define what we call a regularized trace of the heat kernel.
If the hyperbolic Riemann surface is compact, then the regularized
trace must agree with the trace of the heat kernel. The non-compact
case comes with parabolic elements whose contribution to the trace
is unbounded. In such case, we need to subtract the contribution of
the parabolic elements, i.e. to regularize the trace of the heat
kernel. The section ends with two remarks. The first remark relates
our expression for the elliptic heat trace with the expression that
is most common in the literature (\cite{Ku 73}, \cite{He 76}.) The second remark presents in brief how the results in this
section lead to the Selberg trace formula. The staples of this
section are the periodization of the heat kernel on a hyperbolic
Riemann surface of finite volume together with the integral
expression for the heat kernel in the upper half plane.

\hskip 0.2in
Let $\Delta_M$ denote the Laplace operator on the surface $M$. Consider the heat operator $\Delta_M + \partial_t$ acting on functions $f : M\times {\mathbb R}^{+} \mapsto \mathbb R$ which are $C^2(M)$ and $C^1({\mathbb R}^{+})$. Then the heat kernel associated to $M$ is the minimal integral kernel which inverts the heat operator. Namely, the heat kernel is a function $K_M: \mathbb{R} \times M \times M \mapsto \mathbb R$ satisfying the following conditions. For any bounded function $f\in C^2(M)$ consider the integral transform
\begin{align*}
u(t,x) = \int_M K_M(t,x,y)f(y)d\mu_M(y)\,\,.
\end{align*}
Then the following differential and initial time conditions are met
\begin{align*}
\Delta_x u+ \partial_t u = 0
\,\,\,\,\,\text{\rm and} \,\,\,\,\,
f(x) = \lim_{t\to 0^+} u(t,x)\,\,.
\end{align*}

\hskip 0.2in
In this paper we consider hyperbolic surfaces having conical singularities (see \cite{Ju 95}), surfaces realized as the action of discrete groups $\Gamma$ of $\mathrm{PSL}(2,\mathbb{R})$ acting on $\mathbb{H}$. The conical singularities are present once the group $\Gamma$ contains elements (other than the identity) having fixed points. Such is the case with the full modular group $\mathrm{PSL}(2,\mathbb{Z})\,\,.$
In particular, let $M$ be a compact hyperbolic surface, having  $n$ marked points $\{c_i\}_{i=1}^{n}.$ The Riemannian metric $g$ on $M$ is called conically singular hyperbolic metric if and only if for every $i = 1,..,n$ there exists a chart $(U_i,\mu_i)$ about the point $c_i$ isometrically mapping $U_i$ to a hyperbolic cone model with associated angle $\alpha_i$. The metric $g$ induces a compatible complex structure on $M\backslash \{c_i\}_{i=1}^n$ and together with the charts  $\{(U_i,\mu_i)\}_{i=1}^{n}$, provide a complex structure on $M$. Given such structure, there exist a unique complete hyperbolic metric on $M\backslash \{c_i\}_{i=1}^n$ for which each $c_i$ is a cusp. 

\hskip 0.2in
As the surfaces in consideration have conical points, there is a way to extend the domain on which the Laplace operator acts so that it is self-adjoint. The Friedrichs procedure is one such possible extension and we use it here for our spectral purposes. Namely, the domain of the extension is the closure in $L^2(M)$ of the space
\begin{align*}
D = \left\{ f \in L^2(M) \,\,: \,\, \int_M (<\textrm{grad} f, \textrm{grad} f > + f^2)d\mu < \infty\ \textrm{ and } \int_{\partial \textrm{cusp}} f d\mu= 0 \right\}
\end{align*} 
where $d\mu$ denotes the hypebolic volume form and the domain of integration for the second integral is a horocycle.
For details of the above construction we refer the reader to \cite{LP 76}, \cite{CdV 83}, \cite{Ju 95}, and \cite{Ji 94}. Throughout this paper, we will refer to the pseudo-Laplacian above as simply the Laplace operator.

\hskip 0.2in If $M$ is compact, then the spectrum of the (non-negative) Laplace operator is discrete, consisting of eigenvalues $0=\lambda_0 < \lambda_1 \le \lambda_2 \le \,\,\, \to \infty$ counted with multiplicity. Associated to these eigenvalues there is complete system $\{\phi_n(x)\}_{n=0}^{\infty}$
of orthonormal eigenfunction of the Laplace operator on $M.$  For $t>0$ and $x,y \in M$, the heat kernel has the following realization \begin{align}\label{heatkernelexpansioncompact}
K_M(t,x,y) = \sum_{n=0}^{\infty} e^{-\lambda_n t} \phi_n(x) \phi_n(y)\,\,,
\end{align}
and the sum converges uniformly on $[t_0,\infty)\times M \times M$ for fixed $t_0>0$ (see for instance \cite{Ch 84}.)

\hskip 0.2in If $M$ is not compact, the spectrum has a discrete part as well as a continuous part in the real interval $[1/4,\infty]$. The continuos spectrum comes from the parabolic Eisenstein series $E_{\text{par};M,P}(s,z)$ associated to the each cusp $P$ of $M$. In such case, the spectral expansion has the following form (coming from \cite{He 83})
\begin{align}\label{heatkernelexpansionnoncompact}
\notag K_M(t,x,y) =& \sum_{\text{discrete}} e^{-\lambda_n t} \phi_n(x) \phi_n(y)\\
&+ \frac{1}{2\pi} \sum_{\text{cusps }P} \int_{0}^{\infty} e^{-(1/4+r^2)t} E_{\text{par};M,P}(1/2+ir,x) \overline{E_{\text{par};M,P}(1/2+ir,y)} dr\,\,.
\end{align}

\hskip .20in  Let $K_{\mathbb{H}}(t, \tilde{x}, \tilde{y})$
denote the heat kernel on the upper half plane. Recall that
$K_{\mathbb{H}}(t, \tilde{x}, \tilde{y})$ is a function of $t$ and
the hyperbolic distance $d = d_{\mathbb{H}}(\tilde{x}, \tilde{y})$
between $\tilde{x}$ and $\tilde{y}$, so
\begin{equation*}
K_{\mathbb{H}}(t, \tilde{x}, \tilde{y}) = K_{\mathbb{H}}(t,
d_{\mathbb{H}}(\tilde{x}, \tilde{y})).
\end{equation*}
Quoting from page 246 of \cite{Ch 84}, we have for $d>0$
\begin{equation}\label{heatkernelonH1}
K_{\mathbb{H}}(t, \rho) = \frac{\sqrt{2}e^{-t/4}}{(4\pi t)^{3/2}}
\int_{\rho}^{\infty} \frac{u e^{-u^2/4t} du}{\sqrt{\cosh u - \cosh \rho}}
\end{equation}
with
\begin{equation}\label{heatkernelonH0}
K_{\mathbb{H}}(t, 0) = \frac{1}{2\pi } \int_{0}^{\infty}
e^{-(1/4+r^2)t} \tanh(\pi r)rdr.
\end{equation}

\begin{rmk}\label{complexrealtime}
\emph{It is possible to extend the heat kernel from real valued time
to complex valued time. For the complex valued time $z=t+is$ with
$t>0$ we have
\begin{equation*}
K_{\mathbb{H}}(z, d) = \frac{\sqrt{2}e^{-z/4}}{(4\pi z)^{3/2}}
\int_{d}^{\infty} \frac{u e^{-u^2/4z}du}{\sqrt{\cosh u - \cosh d}}
\end{equation*}
To see that this makes sense set $\tau=|z|^2/t$ which is clearly
positive. Then we have the bound
\begin{align*}
|K_{\mathbb{H}}(z,d)| &\le
\frac{\sqrt{2}e^{-t/4}}{(4\pi)^{3/2}(t^2+s^2)^{3/4}}
\int_{d}^{\infty} \frac{ue^{-u^2/4\tau}du}{\sqrt{\cosh u - \cosh
d}}\\
&\le e^{s^2/4t}t^{-3/2}(t^2+s^2)^{3/4} K_{\mathbb{H}}(\tau,d).
\end{align*}
}
\end{rmk}

\hskip 0.2in For any hyperbolic Riemann surface $M \simeq \Gamma \backslash \mathbb H$, one can express
the heat kernel as a periodization of the heat kernel of the
hyperbolic plane. Let $x$ and $y$ denote points on $M$ with lifts
$\tilde{x}$ and $\tilde{y}$ to $\mathbb{H}$. Then we can write the
heat kernel on $M$ as
\begin{equation}\label{periodization}
K_M(t,x,y) = \sum_{\gamma \in \Gamma} K_{\mathbb{H}}(t,
d_{\mathbb{H}}(\tilde{x}, \gamma \tilde{y})).
\end{equation}

Denote by $H(\Gamma), P(\Gamma)$, and $E(\Gamma)$ complete sets of
$\Gamma$-inconjugate primitive hyperbolic, parabolic, and elliptic elements, respectively, of the group  $\Gamma$. If $M$ is compact, then $P(\Gamma)$ is
empty. Let $\Gamma_{\gamma}$ denote the centralizer of $\gamma \in
\Gamma$. If $\gamma$ is a hyperbolic or a parabolic element then
$\Gamma_{\gamma}$ is isomorphic to the infinite cyclic group. If
$\gamma$ is elliptic then its centralizer is isomorphic
to the finite cyclic group of order $q_{\gamma}$. In each instance, the centralizer is generated by a primitive element. We can use
elementary theory of Fuchsian groups (see for instance \cite{McK 72}) and decompose the group $\Gamma$ into conjugacy classes as follows. Given any hyperbolic or parabolic element $\eta$, there exist a primitive element $\gamma$ in $H(\Gamma)$ or $P(\Gamma)$ respectively, and a unique positive integer $n$, such that  $\eta$ belongs to the conjugacy class  $\{\kappa^{-1}\gamma^n\kappa : \kappa \in \Gamma_{\gamma} \backslash \Gamma\}$. For a given elliptic element $\eta$, there exist a primitive element $\gamma \in E(\Gamma)$  and a unique integer $1\le n < q_{\gamma}$, such that  $\eta$ belongs to the conjugacy class  $\{\kappa^{-1}\gamma^n \kappa : \kappa \in \Gamma_{\gamma} \backslash \Gamma\}$. With these in mind, we can write the periodization (\ref{periodization}) as 
\begin{align*}
K_M(t,x,y) = K_{\mathbb{H}}(t,\tilde{x},\tilde{y}) &+
 \sum_{\gamma \in P(\Gamma)} \sum_{n=1}^{\infty}\sum_{\kappa \in
\Gamma_{\gamma} \backslash \Gamma}
K_{\mathbb{H}}(t,\tilde{x},\kappa^{-1} \gamma^n \kappa
\tilde{y})\\
&+  \sum_{\gamma \in H(\Gamma)} \sum_{n=1}^{\infty}\sum_{\kappa \in
\Gamma_{\gamma} \backslash \Gamma}
K_{\mathbb{H}}(t,\tilde{x},\kappa^{-1} \gamma^n \kappa
\tilde{y})\\
&+  \sum_{\gamma \in E(\Gamma)} \sum_{n=1}^{q_{\gamma}-1}
\sum_{\kappa \in \Gamma_{\gamma} \backslash \Gamma}
K_{\mathbb{H}}(t,\tilde{x},\kappa^{-1} \gamma^n \kappa \tilde{y}).
\end{align*}
Using the above decomposition we define the parabolic contribution
(i.e. the contribution coming from the parabolic elements) to the
trace of the heat kernel by
\begin{align*}
PK_M(t,x) = \sum_{\gamma \in P(\Gamma)} \sum_{n=1}^{\infty} 
\sum_{\kappa \in \Gamma_{\gamma} \backslash \Gamma}
K_{\mathbb{H}}(t,\tilde{x},\kappa^{-1} \gamma^n \kappa \tilde{x})
\end{align*}
and in a similar manner we define the hyperbolic contribution and
elliptic contribution which we denote by $HK_M(t,x)$ and $EK_M(t,x)$
respectively.

\begin{thm}\label{integralrepresentation}
Let $M$ be a connected, hyperbolic Riemann surfaces of finite volume
with $p$ cusps and $m$ elliptic fixed points. \mbox{}
\begin{enumerate}[(a)]
\item
For each $t>0$, the sum
\begin{align*}
HK_M(t,x) = \sum_{\gamma \in H(\Gamma)} \sum_{n=1}^{\infty} 
\sum_{\kappa \in \Gamma_{\gamma} \backslash \Gamma}
K_{\mathbb{H}}(t,\tilde{x},\kappa^{-1} \gamma^n \kappa \tilde{x})
\end{align*}
is a well-defined function of $x \in M$.
\item
For $\gamma \in H(\Gamma)$
denote by $C_{\gamma}$ the infinite volume hyperbolic cylinder
which is realized as $\Gamma_{\gamma}\backslash \mathbb{H}$. Then we
have the equality
\begin{align*}
\text{\rm HTr}K_M(t) &= \int_M HK_M(t,x) d\mu(x)\\
&= \frac{1}{2} \sum_{\gamma \in H(\Gamma)}
\int_{C_{\gamma}}(K_{C_{\gamma}}-K_{\mathbb{H}})(t,x,x)d\mu(x).
\end{align*}
\item
For each $t>0$, the sum
\begin{align*}
EK_M(t,x) = \sum_{\gamma \in E(\Gamma)} \sum_{n=1}^{q_{\gamma}-1}
\sum_{\kappa \in \Gamma_{\gamma} \backslash \Gamma}
K_{\mathbb{H}}(t,\tilde{x},\kappa^{-1} \gamma^n \kappa \tilde{x})
\end{align*}
is a well-defined function of $x \in M$.
\item
For $\gamma \in E(\Gamma)$
denote by $C_{\gamma}$ the infinite volume hyperbolic cone which
is realized as $\Gamma_{\gamma}\backslash \mathbb{H}$. Then we have
the equality
\begin{align*}
\text{\rm ETr}K_M(t) &= \int_M EK_M(t,x) d\mu(x)\\
&= \sum_{\gamma \in E(\Gamma)}
\int_{C_{\gamma}}(K_{C_{\gamma}}-K_{\mathbb{H}})(t,x,x)d\mu(x).
\end{align*}
\end{enumerate}
\end{thm}
\begin{proof}
Note that parts (c) and (d) are the ``elliptic version'' of the
first two parts and follow the same proof pattern. For details for
parts (a) and (b) see Theorem 1.1 of \cite{JoLu 97b}.
\end{proof}

\begin{defn}\label{standardheattrace}
Let us define the regularized or standard heat trace for the
connected hyperbolic surface $M$ as
\begin{align*}
\text{\rm STr}K_M(t) = \text{\rm HTr}K_M(t) + \text{\rm ETr}K_M(t) +
\mathrm{vol}(M)K_{\mathbb{H}}(t,0).
\end{align*}
If $M$ is a hyperbolic Riemann surface of finite volume, but not
connected, let $M_1, \cdots, M_n$ denote the connected components,
and define
\begin{align*}
\text{\rm HTr}K_M(t) &= \sum_{j=1}^n \text{\rm HTr}K_{M_j}(t) \quad \text{\rm ETr}K_M(t) =
\sum_{j=1}^n \text{\rm ETr}K_{M_j}(t)\\
\text{\rm STr}K_M(t) &= \sum_{j=1}^n \text{\rm STr}K_{M_j}(t)
\end{align*}
\end{defn}

\hskip 0.2in The following result due to Selberg \cite{Se 56}
evaluates the integral representation stated in Theorem
\ref{integralrepresentation} part (b).
\begin{thm}\label{hyperbolicheattrace}
Let $M$ be a connected, hyperbolic Riemann surface of finite volume
with $p$ cusps and $m$ elliptic fixed points. Let $\ell_{\gamma}$ be
the length of the geodesic in the homotopy class determined by
$\gamma \in H(\Gamma)$. Then the hyperbolic trace of the heat kernel
is given by
\begin{equation*}
\text{\rm HTr}K_M(t) = \frac{e^{-t/4}}{\sqrt{16\pi t}}
\sum_{\gamma \in H(\Gamma)}\sum_{n=1}^{\infty}
\frac{\ell_{\gamma}}{\sinh(n
\ell_{\gamma}/2)}e^{-(n\ell_{\gamma})^2/4t}.
\end{equation*}
\end{thm}
\begin{proof}
For details see Theorem 1.3 of \cite{JoLu 97b}.
\end{proof}

\hskip 0.2in The following result evaluates the integral
representation stated in Theorem \ref{integralrepresentation} part
(d).
\begin{thm}\label{ellipticheattrace}
Let $M$ be a connected, hyperbolic Riemann surface of finite volume
with $p$ cusps and $m$ elliptic fixed points. For each $\gamma \in
E(\Gamma)$ denote by $q_{\gamma}$ the order of the finite cyclic
group $\Gamma_{\gamma}$ generated by $\gamma$. Denote by
$C_{q_{\gamma}}$ the infinite hyperbolic cone associated to $\gamma$
which is realized as $\Gamma_{\gamma}\backslash \mathbb{H}$. Then
the elliptic trace of the heat kernel is given by
\begin{equation*}
\text{\rm ETr}K_M(t) = \frac{e^{-t/4}}{\sqrt{16\pi t}}\sum_{\gamma\in
E(\Gamma)}\sum_{n=1}^{q_{\gamma}-1} \frac{1}{q_{\gamma}} 
\int_{0}^{\infty}
\frac{e^{-u^2/4t}\cosh(u/2)}{\sinh^2(u/2)+\sin^2(n\pi/q_{\gamma})}du.
\end{equation*}
\end{thm}
\begin{proof}
We start by unfolding the integral representing the elliptic heat
trace, so that the integration will take place over each cone
$C_{q_{\gamma}}$.  That is, we write
\begin{align*}
\text{\rm ETr}K_M(t) =& \int_{M} \sum_{\gamma \in E(\Gamma)} \sum_{n=1}^{q_{\gamma}-1} \sum_{\kappa \in
\Gamma_{\gamma} \backslash \Gamma}
 K_{\mathbb{H}}(t,z,\kappa^{-1} \gamma^n \kappa z) d\mu(z)\\
=&\sum_{\gamma \in E(\Gamma)} \sum_{n=1}^{q_{\gamma}-1} \sum_{\kappa
\in \Gamma_{\gamma} \backslash \Gamma}
\int_{\Gamma \backslash \mathbb{H}} K_{\mathbb{H}}(t,\kappa z,\gamma^n \kappa z) d\mu(z)\\
=&\sum_{\gamma \in E(\Gamma)}\sum_{n=1}^{q_{\gamma}-1}
\int_{C_{q_{\gamma}}}K_{\mathbb{H}}(t,z,\gamma^n z) d\mu(z).
\end{align*}
We know proceed by computing the inner integral. In doing so, it
will be convenient to write $q$ instead of $q_{\gamma}$ and call the
inner integral $I$. Using the hyperbolic polar coordinates as
described by equation (\ref{hyperbolicpolarcoordinates}), we can
write
\begin{align*}
I &= \int_{C_q}K_{\mathbb{H}}(t,d(z,\gamma^n z)) d\mu (z)\\
&= \int_{0}^{2\pi}\int_{0}^{\infty} K_{\mathbb{H}}(t,d(z,\gamma^n
z))q^{-1}\sinh(\rho)d\rho d\theta.
\end{align*}
Using the representation for the heat kernel on the upper half plane
as in equation (\ref{heatkernelonH1}), we can further write
\begin{align*}
I &= \int_{0}^{2\pi}\int_{0}^{\infty} \frac{\sqrt{2}e^{-t/4}}{(4\pi
t)^{3/2}} \int_{d(z,\gamma^n z)}^{\infty}
\frac{ue^{-u^2/4t}du}{\sqrt{\cosh u - \cosh d(z,\gamma^n z)}}
q^{-1}\sinh(\rho)d\rho d\theta\\
&=\frac{\sqrt{2}e^{-t/4}}{(4\pi t)^{3/2}} \frac{2\pi}{q}
\int_{0}^{\infty} \int_{d(z,\gamma^n z)}^{\infty}
\frac{ue^{-u^2/4t}du}{\sqrt{\cosh u - \cosh d(z,\gamma^n z)}}
\sinh(\rho)d\rho
\end{align*}
where the last equality follows from integrating with respect to the
$\theta$ variable.

Referring to equation (\ref{beardon}), let $a(n,q,\rho)=d(z,\gamma^n
z)=\cosh^{-1}(1+2\sin^2(n\pi/q)\sinh^2(\rho))$. Then we have
\begin{align*}
I =\frac{\sqrt{2}e^{-t/4}}{(4\pi t)^{3/2}} \frac{2\pi}{q}
\int_{0}^{\infty} \int_{a(n,q,\rho)}^{\infty}
\frac{ue^{-u^2/4t}du}{\sqrt{\cosh u -
1-2\sin^2(n\pi/q)\sinh^2(\rho)}} \sinh(\rho)d\rho.
\end{align*}
Now make the following change of variables $x=\cosh(\rho)$, so then
$dx=\sinh(\rho)d\rho$ and the limits of integration change to 1 and
$\infty$ respectively. Using that $\sinh^2 (\rho) = x^2-1$, we can
further write
\begin{align*}
I =\frac{\sqrt{2}e^{-t/4}}{(4\pi t)^{3/2}} \frac{2\pi}{q}
\int_{1}^{\infty} \int_{b(n,q,x)}^{\infty}
\frac{ue^{-u^2/4t}dudx}{\sqrt{\cosh u - 1-2\sin^2(n\pi/q)(x^2-1)}}
\end{align*}
where $b(n,q,x)=\cosh^{-1}(1+2\sin^2(n\pi/q)(x^2-1))$.

We proceed by interchanging the limits of integration. As such, note
that the variable $u$ will range from 0 to $\infty$ and the variable
$x$ from 1 to $c(n,q,u)$, the latter being defined by
\begin{align*}
c(n,q,u) = \sqrt{1+\frac{\cosh u - 1}{2\sin^2(n\pi/q)}}
\end{align*}
which comes from solving $\cosh u - 1-2\sin^2(n\pi/q)(x^2-1)=0$ in
terms of $x$. Thus, we have that
\begin{align*}
I =\frac{\sqrt{2}e^{-t/4}}{(4\pi t)^{3/2}} \frac{2\pi}{q}
\int_{0}^{\infty} \int_{1}^{c(n,q,u)}
\frac{ue^{-u^2/4t}dxdu}{\sqrt{\cosh u - 1-2\sin^2(n\pi/q)(x^2-1)}}.
\end{align*}
We can now compute the integral with respect to $x$. In doing so we
will use the integration formula
\begin{align*}
\int_{a}^{b} \frac{dx}{\sqrt{\alpha -\beta x^2}} =
-\frac{1}{\sqrt{\beta}}
\cos^{-1}\Bigg(x\sqrt{\frac{\beta}{\alpha}}\Bigg)\Bigg|_{x=a}^{x=b}.
\end{align*}
With $\alpha = \cosh u - 1+2\sin^2(n\pi/q)$ and
$\beta=2\sin^2(n\pi/q)$ we get that the inner integral equals
\begin{align*}
-\frac{1}{\sqrt{2}\sin(n\pi/q)}\Bigg[ &\cos^{-1}\Bigg(
\sqrt{1+\frac{\cosh u -
1}{2\sin^2(n\pi/q)}}\sqrt{\frac{2\sin^2(n\pi/q)}{{\cosh u -
1+2\sin^2(n\pi/q)}}}\Bigg)\\
&- \cos^{-1}\Bigg( \sqrt{\frac{2\sin^2(n\pi/q)}{{\cosh u -
1+2\sin^2(n\pi/q)}}}\Bigg)\Bigg]\\
&= \frac{1}{\sqrt{2}\sin(n\pi/q)}\cos^{-1}\Bigg(
\sqrt{\frac{2\sin^2(n\pi/q)}{{\cosh u - 1+2\sin^2(n\pi/q)}}}\Bigg)
\end{align*}
since the first term in the brackets equals $\cos^{-1}(1)=0$. Then
we can write
\begin{align*}
I =\frac{e^{-t/4}}{(4\pi t)^{3/2}} \frac{2\pi}{q\sin(n\pi/q)}
\int_{0}^{\infty} ue^{-u^2/4t} \cos^{-1}\Bigg(
\sqrt{\frac{2\sin^2(n\pi/q)}{{\cosh u - 1+2\sin^2(n\pi/q)}}}\Bigg)
du.
\end{align*}

Noting that $-2t\partial_u (e^{-u^2/4t})=ue^{-u^2/4t}$ we proceed by
integrating by parts
\begin{align*}
I =&\frac{e^{-t/4}}{(4\pi t)^{3/2}} \frac{2\pi (-2t)}{q\sin(n\pi/q)}
\int_{0}^{\infty} d(e^{-u^2/4t}) \cos^{-1}\Bigg(
\sqrt{\frac{2\sin^2(n\pi/q)}{{\cosh u - 1+2\sin^2(n\pi/q)}}}\Bigg)
du\\
=&-\frac{e^{-t/4}}{(4\pi t)^{1/2}} \frac{1}{q\sin(n\pi/q)} \Bigg[
e^{-u^2/4t} \cos^{-1}\Bigg( \sqrt{\frac{2\sin^2(n\pi/q)}{{\cosh u -
1+2\sin^2(n\pi/q)}}}\Bigg)\Bigg|_{u=0}^{u=\infty}\\
&- \int_{0}^{\infty} e^{-u^2/4t} \frac{d}{du}\cos^{-1}\Bigg(
\sqrt{\frac{2\sin^2(n\pi/q)}{{\cosh u - 1+2\sin^2(n\pi/q)}}}\Bigg)
du\Bigg].
\end{align*}
Note that the first term inside the brackets evaluates to 0 at both
limits of integration. This leaves us with
\begin{align*}
I =\frac{e^{-t/4}}{(4\pi t)^{1/2}} \frac{1}{q\sin(n\pi/q)}
\int_{0}^{\infty} e^{-u^2/4t} \frac{d}{du}\cos^{-1}\Bigg(
\sqrt{\frac{2\sin^2(n\pi/q)}{{\cosh u - 1+2\sin^2(n\pi/q)}}}\Bigg)
du.
\end{align*}
Using the formulas $\cosh u - 1=2\sinh^2(u/2)$ and $\sinh u =
2\sinh(u/2)\cosh(u/2)$ we get that
\begin{align*}
\frac{d}{du}\cos^{-1}\Bigg( \sqrt{\frac{2\sin^2(n\pi/q)}{{\cosh u -
1+2\sin^2(n\pi/q)}}}\Bigg) = \frac{1}{2}\cdot
\frac{\sin(n\pi/q)\cosh(u/2)}{\sinh^{2}(u/2)+\sin^2(n\pi/q)}
\end{align*}
which we then plug into the above integral to get
\begin{align*}
I &=\frac{e^{-t/4}}{2q\sqrt{4\pi t}} \int_{0}^{\infty}
\frac{e^{-u^2/4t}\cosh(u/2)}{\sinh^{2}(u/2)+\sin^2(n\pi/q)}du\\
&=\frac{e^{-t/4}}{q\sqrt{16\pi t}} \int_{0}^{\infty}
\frac{e^{-u^2/4t}\cosh(u/2)}{\sinh^{2}(u/2)+\sin^2(n\pi/q)}du.
\end{align*}
which completes the proof.
\end{proof}

\begin{rmk}\label{parseval}
\emph{The literature on the Selberg trace formula (see for example \cite{He 76} on page 351
or \cite{Ku 73} on pages 100-102), often gives the following expression for the elliptic heat trace
\begin{align}\label{Hejhalelliptictrace}
\text{\rm ETr}K_M(t) = \sum_{\gamma\in E(\Gamma)} 
\sum_{n=1}^{q_{\gamma}-1}\frac{e^{-t/4}}{2q_{\gamma}\sin(n\pi/q_{\gamma})}
            \int_{-\infty}^{\infty} \frac{e^{-2\pi nr / q_{\gamma} - tr^2}}{1+e^{-2\pi r}}dr
\end{align}
whereas our computations in Theorem \ref{ellipticheattrace} show
that
\begin{align*}
\text{\rm ETr}K_M(t) = \frac{e^{-t/4}}{\sqrt{16\pi t}}\sum_{\gamma\in
E(\Gamma)} \sum_{n=1}^{q_{\gamma}-1} \frac{1}{q_{\gamma}}
\int_{0}^{\infty}
\frac{e^{-u^2/4t}\cosh(u/2)}{\sinh^2(u/2)+\sin^2(n\pi/q_{\gamma})}du.
\end{align*}
\text{\hskip 0.2in}To see that the two different representations are equal we use
Parseval formula. Recall that the Fourier transform $\hat{f}(u)$  of a
function $f(r) \in L^1(\mathbb R)$ is given by
\begin{align*}
\hat{f}(u) = \frac{1}{2\pi}\int_{-\infty}^{\infty}
f(r)e^{-iur}dr\,\,,
\end{align*}
where the choice of the normalizing constant is consistent with the work of \cite{He 76} and \cite{He 83}.
Consider the integral in equation (\ref{Hejhalelliptictrace}) above
which we write as
\begin{align*}
I =\int_{-\infty}^{\infty} e^{-tr^2}\frac{e^{-2\pi n r / q
}}{1+e^{-2\pi r}}dr
\end{align*}
where for simplicity of notation we have placed $q$ in lieu of $q_{\gamma}$.
Set
\begin{align*}
f(r) = e^{-tr^2} \textrm{ and } g(r)=\frac{e^{-2\pi n r / q
}}{1+e^{-2\pi r}}\,\,.
\end{align*}
It immediately follows that
\begin{align*}
\hat{f}(u) = \frac{1}{\sqrt{4\pi t}} e^{-u^2/4t}\,\,.
\end{align*}
To compute the Fourier transform of $g(r)$, namely
\begin{align*}
\hat{g}(u) = \frac{1}{2\pi} \int_{-\infty}^{\infty} \frac{e^{-(2\pi n/q + iu)r}}{1+e^{-2\pi r}}dr\,\,,
\end{align*}
we set $h(z)=\cfrac{e^{-(2\pi n/q + iu)z}}{1+e^{-2\pi z}}$ and note that $h(z)$ has simple poles at $z=\cfrac{i}{2}(2m+1) \,\,, (m\in \mathbb{Z}),$ with residue $(2\pi)^{-1}\left(e^{-i(\pi n/q + iu/2)}\right)^{2m+1}$. By integrating $h(z)$ over the standard upper semicircle contour and using the residue theorem, we then obtain
\begin{align*}
\hat{g}(u) = \frac{i}{2\pi} \sum_{m=0}^{\infty} \left(e^{-i(\pi n/q + iu/2)}\right)^{2m+1}\,\,.
\end{align*}
Recall that for $|z|<1$ we have $\sum_{m=0}^{\infty} z^{2m+1} = \cfrac{z}{1-z^2} = \cfrac{1}{2}\cdot \cfrac{2}{z^{-1}-z}\,\,.$ Assuming initially that $u<0$, with $z=e^{-i(\pi n/q + iu/2)}$, we then arrive at
\begin{align*}
\hat{g}(u) = \frac{1}{4\pi} \frac{2i}{e^{i(\pi n/q + iu/2)}- e^{-i(\pi n/q + iu/2)}} = \frac{\csc(n\pi/q+iu/2)}{4\pi}\,\,.
\end{align*}
\hskip 0.2in Using Parseval formula (see formula (13) on page 202 of \cite{Ru 74}), we have that 
\begin{align*}
I =& \int_{-\infty}^{\infty} f(r)\overline{g(r)}dr
= 2\pi \int_{-\infty}^{\infty} \hat{f}(u)\overline{\hat{g}(u)}du\\
=&\frac{1}{\sqrt{16\pi t}}\int_{-\infty}^{\infty} e^{-u^2/4t}
\overline{\csc(n\pi/q +iu/2)}du\,\,,
\end{align*}
where the $2\pi$ factor on the Fourier transform side is due to the choice of normalization in the definition of the Fourier transform.
Using that $\sin z =\sin (x+iy) = \sin x \cosh y + i \cos x \sinh y$
we can write
\begin{align*}
\csc(x+iy)&=\frac{1}{\sin x \cosh y + i \cos x \sinh y} \\
&=\frac{\sin x \cosh y}{\sinh^2 y +  \sin^2 x}-i\cdot\frac{\cos x
\sinh y}{\sinh^2 y + \sin^2 x}.
\end{align*}
With $x = n\pi/q$ and $y=u/2$ we can write
\begin{align*}
I =& \frac{1}{\sqrt{16\pi t}}\int_{-\infty}^{\infty} e^{-u^2/4t}
\frac{\sin(n\pi/q) \cosh (u/2)}{\sinh^2 (u/2) +  \sin^2 (n\pi/q)}
du\\
&+\frac{i}{\sqrt{16\pi t}}\int_{-\infty}^{\infty} e^{-u^2/4t}
\frac{\cos(n\pi/q) \sinh (u/2)}{\sinh^2 (u/2) +  \sin^2 (n\pi/q)}
du.
\end{align*}
Note that the first integrand in the right hand side above is an
even function in the variable of integration whereas the second
integrand is odd. As such, the second integral equals to zero and we
can write
\begin{align*}
I= \frac{2\sin(n\pi/q)}{\sqrt{16\pi t}}\int_{0}^{\infty}
\frac{e^{-u^2/4t} \cosh (u/2)}{\sinh^2 (u/2) +  \sin^2 (n\pi/q)} du.
\end{align*}
Finally, substituting the above in equation (\ref{Hejhalelliptictrace}) yields the integral representation of $\text{\rm ETr}K_M(t)$ as it appears in Theorem \ref{ellipticheattrace}.
}
\end{rmk}

\begin{rmk}\label{Selbergtraceformula}
\emph{In the case $M$ is compact, the standard trace $\text{\rm STr}K_M(t)$ is simply
the trace of the heat kernel. One immediately obtains from (\ref{heatkernelexpansioncompact}) the spectral aspect of the standard trace,
\begin{align}\label{regularizedtracecompactspectral}
\text{\rm STr}K_M(t) = \int_{M} K_M(t,x,x)d\mu(x) = \sum_{n=0}^{\infty} e^{-\lambda_nt}\,\,.
\end{align}
On the other hand, the results in this section, namely Definition \ref{standardheattrace}, integral representations (\ref{heatkernelonH0}) and (\ref{Hejhalelliptictrace}), and Theorems \ref{hyperbolicheattrace} and \ref{ellipticheattrace}, give the geometric interpretation of the standard trace, namely
\begin{align}\label{regularizedtracegeometric}
\notag \text{\rm STr}K_M(t) =&
\frac{\textrm{vol}(M)}{4\pi}
\int_{-\infty}^{\infty}e^{-(r^2+1/4)t}\tanh(\pi r)rdr\\
\notag &+ \sum_{\gamma
\in H(\Gamma)} \sum_{n=1}^{\infty} \frac{\ell_{\gamma}}{\sinh(n
\ell_{\gamma}/2)} \frac{e^{-t/4}}{\sqrt{16\pi t}} e^{-(n\ell_{\gamma})^2/4t}\\
&+\sum_{\gamma\in E(\Gamma)} 
\sum_{n=1}^{q_{\gamma}-1}\frac{e^{-t/4}}{2q_{\gamma}\sin(n\pi/q_{\gamma})}
            \int_{-\infty}^{\infty} \frac{e^{-2\pi nr / q_{\gamma} - tr^2}}{1+e^{-2\pi r}}dr\,\, .
\end{align}
The combination of (\ref{regularizedtracecompactspectral}) and (\ref{regularizedtracegeometric}) represent an instance of the Selberg trace formula as applied to the function $f(r) = e^{-tr^2}$ and its Fourier transform  $\hat{f}(u) =
(4\pi t)^{-1/2}e^{-u^2/4t}$. \\
\text{\hskip 0.2in} One can use this special case to generalize the trace formula to a larger class of functions. First, denote by $r_n$ the solutions to $\lambda_n = 1/4+ r_n^2$. The non-negativity of the eigenvalues imply that for each $n$ there are two solutions $r_n$ which are either opposite real numbers or complex conjugate numbers in the segment $[-i/2,i/2]\,\,.$ \\
\text{\hskip 0.2in}To continue, let $h(t)$ be any measurable function for which $h(t)e^{(1/4+\varepsilon)t}$ is in $L^1(\mathbb R)\,$ for some $\varepsilon>0\,.$ Multiply the right-hand side of (\ref{regularizedtracecompactspectral}) and (\ref{regularizedtracegeometric}) by $h(t)e^{t/4}$ and integrate from 0 to $\infty$ with respect to $t.$ Set
\begin{align*}
H(r)=\int_{0}^{\infty} h(t) e^{-r^2t} dt\,\,.
\end{align*}
By rewriting the absolute integrand of $H(r)$ as $|h(t)e^{(1/4+\varepsilon)t})|\cdot |e^{-(r^2+1/4+\varepsilon)t})|$ and recalling the imposed conditions on $h(t)$, it easily follows that $H(r)$ is analytic inside the horizontal strip $|\text{Im}(r) | \le 1/2 + \varepsilon'$ for some $\varepsilon'>0$ depending on $\varepsilon\,\,.$ The Fourier transform of $H(r)$ has the form
\begin{align*}
\hat{H}(u)=\int_{0}^{\infty}  h(t) \frac{1}{\sqrt{4\pi t}}e^{-u^2/4t} dt\,\,.
\end{align*}
Putting these facts together yields the Selberg trace formula in the compact case, namely
\begin{align}\label{STFcompact}
\notag \sum_{r_n} H(r_n) =&
\frac{\textrm{vol}(M)}{4\pi}
\int_{-\infty}^{\infty}H(r) \tanh(\pi r)rdr\\
\notag &+ \sum_{\gamma
\in H(\Gamma)} \sum_{n=1}^{\infty} \frac{\ell_{\gamma}}{2\sinh(n
\ell_{\gamma}/2)}\hat{H}(n\ell_{\gamma})\\
&+\sum_{\gamma\in
E(\Gamma)}\sum_{n=1}^{q_{\gamma}-1} \frac{1}{2q_{\gamma}\sin(n\pi/q_{\gamma})}
            \int_{-\infty}^{\infty} H(r) \frac{e^{-2\pi nr / q_{\gamma}}}{1+e^{-2\pi r}}dr\,\, ,
\end{align}
where the sum on the left is taken over $r_n \in (0,\infty) \cup [0,i/2]\,\,.$ We note that (\ref{STFcompact}) above agrees with the formula in Theorem 5.1 of \cite{He 76}, with $\chi$ being the trivial character of the group $\Gamma$. 
\\
\\
\text{\hskip 0.2in} In the case $M$ is non-compact, the regularized trace equals the trace of the heat kernel minus the contribution of the parabolic conjugacy classes. While the geometric side of the regularized trace is precisely as in (\ref{regularizedtracegeometric}), the spectral side has the following presentation
\begin{align}\label{regularizedtracenoncompactspectral}
\notag \text{\rm STr}K_M(t) =& \sum_{C(M)} e^{-\lambda_n t}  - \frac{1}{4\pi} \int_{-\infty}^{\infty} e^{-(r^2+1/4)t} \frac{\phi'}{\phi}(1/2+ir)dr\\
\notag &+ \frac{p}{2\pi} \int_{-\infty}^{\infty} e^{-(r^2+1/4)t} \frac{\Gamma'}{\Gamma}(1+ir)dr\\
&- \frac{1}{4}\Big(p - \text{Tr }\Phi(1/2) \Big)e^{-t/4} + \frac{p \log 2}{\sqrt{4\pi t}} e^{-t/4}\,\,,
\end{align}
where   $C(M)$ denotes a set of eigenvalues associated to $L^2$ eigenfunctions on $M$, $\phi(s)$ the determinant of the scattering matrix $\Phi(s)$, $\Gamma(s)$ the Euler gamma function, while $p$ the number of cusps of $M$ (see page 313 of \cite{He 83}).\\
\text{\hskip 0.2in} One can use the same argument as in the compact case to obtain the formal Selberg trace formula in the non-compact case. While the geometric side doesn't change (see the right-hand side of (\ref{STFcompact}), the spectral side is as follows:
\begin{align}\label{STFnoncompact}
\notag \text{\rm spectral side} =& \sum_{r_n} H(r_n)  - \frac{1}{4\pi} \int_{-\infty}^{\infty} H(r) \frac{\phi'}{\phi}(1/2+ir)dr\\
\notag &+ \frac{p}{2\pi} \int_{-\infty}^{\infty} H(r)\frac{\Gamma'}{\Gamma}(1+ir)dr\\
&- \frac{1}{4}\Big(p - \text{Tr }\Phi(1/2) \Big)H(0) + p \log(2) \hat{H}(0)\,\,.
\end{align}
}
\end{rmk}

\section{Results from spectral theory and the heat equation}
\label{Results from spectral theory and the heat equation}

\hskip 0.2in
In this section we present  various bounds on the heat kernel
associated to the degenerating family. As we will see, these bounds turn out to be independent of the family parameter $q$ as well as the imaginary part of the time variable $z$.  In the process,
we will also define and analyze the Poisson kernel subject to Dirichlet condition on
the hyperbolic cone of finite volume $C_{q,\delta}$,  as this Poisson kernel can be realized as a normal derivative of the heat kernel on the finite volume cone.

\hskip 0.2in The first two propositions of this section yield bounds
which are independent of $q$.

\begin{prop}
Let $R_q$ denote either a finite volume degenerating hyperbolic surface $M_q$ or an infinite volume degenerating hyperbolic cone $C_q$. 
Then the spectral expansion for the heat kernel on $R_q$ converges in the topology of smooth functions on compact sets. That is, every derivative in the space variables converges uniformly on compact sets.
\end{prop}

\begin{proof}
If the surface $R_q$ is compact, then the spectral expansion of the heat kernel has the following expression 
\begin{align*}
K_{R_q} (t,x,y) = \sum_{n=0}^{\infty} e^{-\lambda_{n,q}t} \phi_{n,q}(x)\phi_{n,q}(y)
\end{align*}
for an orthonormal basis of eigenfunctions $\{\phi_{n,q}\}$ with corresponding eigenvalues $\{\lambda_{n,q}\}$. Proposition 20.1 of \cite{Sh 87} (see also \cite{Ch 84} pp. 139 - 140) shows that the series above converges smoothly on compact subsets of $\mathbb{R}_{+} \times R_q\times R_q$. 

\hskip 0.2in In the non-compact case, let $\{\Omega_{m,q}\}$ be a compact exhaustion of $R_q$. The heat kernel subject to Dirichlet boundary conditions on $\Omega_{m,q}$ converges smoothly to the heat kernel on $M_q$ (see \cite{Ch 84} starting on p. 187).
Proposition 20.1 of \cite{Sh 87} makes use of Theorem 7.6 in \cite{Sh 87} (Sobolev imbedding theorem compact case). Replacing the latter result by Corollary 7.11 in \cite{GT 83}, allows to extend Proposition 20.1 of \cite{Sh 87} to the Dirichlet heat kernel on $\Omega_{m,q}$. This in turn implies that spectral expansion of the heat kernel subject to Dirichlet boundary conditions on $\Omega_{m,q}$ converges smoothly.        
\end{proof}

\begin{prop}\label{kerneluniformbounds}
Let $M_{q}$ be a compact degenerating surface. Let $\varepsilon
< \delta, x\in \partial C_{q,\varepsilon},$ and $\zeta \in
\partial C_{q,\delta}$. $\partial_{n,x}$ will denote the normal
derivative on $\partial C_{q,\varepsilon}$ with respect to the
variable $x$. For $t$ in a compact interval not containing 0, there
exists a number $C$ independent of $q$ and $t$ such that for
all $s \in \mathbb{R}$, we have the following bounds.
\begin{align}
2|K_{M_{q}}(t+is,x,\zeta)| &\le K_{M_{q}}(t,x,x) +
K_{M_{q}}(t,\zeta,\zeta) \le C\\
2|\partial_{n,x}K_{M_{q}}(t+is,x,\zeta)| &\le
\partial_{n,y}\partial_{n,z}
K_{M_{q}}(t,y,z)\Big|_{\substack{y=x \\z=x}}
+ K_{M_{q}}(t,\zeta,\zeta)\le C\\
2|K_{C_{q}}(t+is,x,\zeta)| &\le K_{C_{q}}(t,x,x) +
K_{C_{q}}(t,\zeta,\zeta) \le C\\
2|\partial_{n,x}K_{C_{q}}(t+is,x,\zeta)| &\le
\partial_{n,y}\partial_{n,z}
K_{C_{q}}(t,y,z)\Big|_{\substack{y=x \\z=x}} +
K_{C_{q}}(t,\zeta,\zeta)\le C
\end{align}
\end{prop}

\begin{proof}
See the proof of Proposition 2.1 of \cite{JoLu 97a}.
\end{proof}

\hskip 0.2in Our next task is to define the Poisson kernel on the
hyperbolic cone $C_{q,\delta}$ associated to an elliptic
representative of order $q$.

\begin{defn}
\emph{Let $K_{C_{q,\delta}}^{D}$ be the Dirichlet heat kernel
on the hyperbolic domain $C_{q,\delta}$. For any point $\zeta
\in
\partial C_{q,\delta}$, denote by $\partial_{n,\zeta}$ the inward normal derivative.
The Poisson kernel $P_{q,\delta}(t,x,\zeta)$ of the hyperbolic
domain $C_{q,\delta}$ is defined by
\begin{align*}
P_{q,\delta}(t,x,\zeta)=\partial_{n,\zeta}K_{C_{q,\delta}}^{D}(t,x,\zeta).
\end{align*}
}
\end{defn}

\begin{rmk}\label{Poissonkernelremark}
\emph{From the Theorem 5 on page 168 of \cite{Ch 84}, we have the
following description of the Poisson kernel. The function
$P_{q,\delta}(t,x,\zeta)$ is an integral kernel for $t>0$ with
$x\in C_{q,\delta}$ and $\zeta \in \partial C_{q,\delta}$,
which solves the following boundary value problem. Let $u=u(t,x)$
satisfy
\begin{align*}
(\Delta - \partial_t)u&=0\\
u(0,x)&=0 \\
u(t,\zeta)&=f(t,\zeta)
\end{align*}
for $\zeta\in \partial C_{q,\delta}$. Then
\begin{align*}
u(t,x) = \int_{0}^{t} \int_{\partial C_{q,\delta}}
P_{q,\delta}(t-\sigma,x,\zeta)f(\sigma,\zeta)d\rho(\zeta)d\sigma
\end{align*}
where $d\rho(\zeta)$ represents the line element on the boundary of
$C_{q,\delta}$.}
\end{rmk}

\hskip 0.2in The following lemma establishes various estimates for
the Poisson kernel which are independent of $q$.

\begin{prop}\label{poissonkernelproperties}
Let $C_{q}$ be a family of infinite volume hyperbolic cones.
For any $\delta > 0$, any $0<\varepsilon<\delta$, and any real
values $t_0,t_1>0$, the following results hold.
\begin{enumerate}[(a)]
\item For all $0<t \le t_1$, $x\in C_{q,\varepsilon}$, and $\zeta \in \partial
C_{q,\delta}$, there is a constant $M$ independent of $q$
such that
\begin{align*}
0\le P_{q,\delta}(t,x,\zeta)\le M.
\end{align*}
\item For all $t_0 \le t \le t_1$, $x\in C_{q,\delta}$, and $\zeta \in \partial
C_{q,\delta}$, there is a constant $M$ independent of $q$
such that
\begin{align*}
0\le P_{q,\delta}(t,x,\zeta)\le M.
\end{align*}
\item For fixed $s$, the $L^2$-norm $\parallel P_{q,\delta}(t+is, \cdot,
\zeta)\parallel_{2;C_{q,\delta}}$ is decreasing in $t$.
\end{enumerate}
\end{prop}

\begin{proof}
The proof follows the template given in Proposition 2.3
of \cite{JoLu 97a} with some minor changes which we now present. 

\hskip 0.2in The Dirichlet heat
kernel is non-negative for all values of the parameters and equal to
zero when either $x$ or $\zeta$ lie on the boundary of the cone
$C_{q, \delta}$. In turn, the difference quotient whose limit
is the Poisson kernel, is a limit of non-negative functions, so that
the Poisson kernel is non-negative. This proves the lower bounds in parts a) and b).

\hskip 0.2in To prove part (c), we break up the Poisson kernel as
$u+iv$ for real valued functions $u$ and $v$. With $s \in
\mathbb{R}$ fixed and $\zeta \in \partial C_{q,\delta}$ we can
write
\begin{align*}
\partial_t \parallel P_{q,\delta}(t+is, \cdot,
\zeta)\parallel_{2;C_{q,\delta}}^{2} &= \partial_t
\int_{C_{q,\delta}} (u^2+v^2)d\mu \\
&= 2\int_{C_{q,\delta}} (uu_t+vv_t)d\mu \\
&= 2\int_{C_{q,\delta}} (u \Delta u+v \Delta v)d\mu \\
&= -2 \int_{C_{q,\delta}} (\langle\text{grad } u, \text{grad }u\rangle +
\langle\text{grad }v, \text{grad }v\rangle)d\mu \\
&\le 0
\end{align*}
where in the last equality we use Green's theorem as applied to functions that vanish on the boundary of $C_{q,\delta}.$

\hskip 0.2in It remains to consider the upper bounds in part (a) and
(b). In the process, we will show that the Poisson kernel is bounded above by the normal derivative of the Dirichlet heat kernel.
Recall from Section \ref{Geometry of elliptic degeneration}, that a fundamental domain for the finite
volume hyperbolic cone $C_{q,\delta}$ is given by the following
subset of the hyperbolic unit disc,
\begin{align*}
\mathcal{F} = \{\rho e^{i\theta}\in \mathbb{D} : 0\le \rho <
\cosh^{-1}(1+\delta q /2\pi),0\le \theta < 2\pi/q\}.
\end{align*}
Let $K_\mathcal{F}^D$ denote the Dirichlet heat kernel of the domain
$\mathcal{F}$. Given a point $\tilde{\zeta}$ on the boundary of
$\mathcal{F}$, there is a unique geodesic $g$ (which depends on the
choice of point and fundamental domain) tangent to the boundary of
$\mathcal{F}$ at $\tilde{\zeta}$. The geodesic $g$ separates the
hyperbolic unit disc into two components. Denote by $\mathcal{V}$
the component containing the fundamental domain $\mathcal{F}$ (see
Fig. \ref{geodesic3}). Associated to this component we have the Poisson kernel
which we denote by $K_{\mathcal{V}}^D$. Using the isotropy subgroup
for $C_{q}$ and the definition of the Poisson kernel, we can
write
\begin{align}\label{poissonkernelperiodization}
P_{q,\delta}(t,x,\zeta) = \partial_{n,\zeta}
K_{C_{q,\delta}}^D(t,x,\zeta) = \partial_{n,\tilde{\zeta}}
\sum_{k=1}^{q} K_{\mathcal{F}}^D (t, \exp(2\pi i
k/q)\tilde{x},\tilde{\zeta}).
\end{align}
Let $p:\mathbb{R}\to \mathbb{D}$ be the unique geodesic parametrized
by arclength defined by the following conditions:
$p(0)=\tilde{\zeta}$, $p$ is perpendicular to $g$ at
$\tilde{\zeta}$, and $p(\mathbb{R}_{+})\subset \mathcal{V}$. Since
$\mathcal{F} \subset \mathcal{V}$, we have that $K_{\mathcal{F}}^D
\le K_{\mathcal{V}}^D$, and since $p(0)$ lies on the boundary of
$C_{q,\delta}$, we have that
\begin{align}\label{dirichletheatkernelbound}
\partial_{n,\tilde{\zeta}}
K_{\mathcal{F}}^D (t, \tilde{x},\tilde{\zeta}) = \lim_{h\to 0}
\frac{K_{\mathcal{F}}^D (t, \tilde{x},p(h))}{h} \le \lim_{h\to 0}
\frac{K_{\mathcal{V}}^D (t, \tilde{x},p(h))}{h}.
\end{align}
Because the boundary of $\mathcal{V}$ is a geodesic (namely $g$),
the method of images implies the equality
\begin{align*}
K_{\mathcal{V}}^D (t, \tilde{x},p(h)) = K_{\mathbb{D}} (t,
\tilde{x},p(h)) - K_{\mathbb{D}} (t, \tilde{x},p(-h)).
\end{align*}
From this and equation (\ref{dirichletheatkernelbound}), we can
write
\begin{align*}
\partial_{n,\tilde{\zeta}}
K_{\mathcal{F}}^D (t, \tilde{x},\tilde{\zeta}) \le 2
\partial_{n,\tilde{\zeta}} K_{\mathbb{D}} (t,
\tilde{x},\tilde{\zeta}).
\end{align*}
Differentiating the sum in (\ref{poissonkernelperiodization}) we
obtain the inequality
\begin{align}\label{poissonkernelinequality}
P_{q,\delta}(t,x,\zeta) \le 2 \partial_{n,\tilde{\zeta}}
\sum_{k=1}^{q} K_{\mathbb{D}} (t, \exp(2\pi i
k/q)\tilde{x},\tilde{\zeta}) \le 2 \partial_{n,\tilde{\zeta}}
K_{C_{q}} (t, x,\zeta)\,\,.
\end{align}
If $x$ and $\zeta$ are bounded away form each other, inequality
(\ref{poissonkernelinequality}) and the limit
(\ref{convergenceheatkernels2}) of Theorem
\ref{convergenceheatkernels} imply that
(\ref{poissonkernelinequality}) is uniformly convergent as $q
\to  \infty$ for $0 < t \le t_1$, proving the upper bound in part (a)
of the proposition. On the other hand, if $x\in C_{q,\delta}$
and $\zeta\in \partial C_{q,\delta}$, the limit
(\ref{convergenceheatkernels2}) again implies that
(\ref{poissonkernelinequality}) converges uniformly for $0 < t_1 \le
t \le t_2$, proving the upper bound in part (b) of the proposition.
\end{proof}

\realfig{geodesic3}{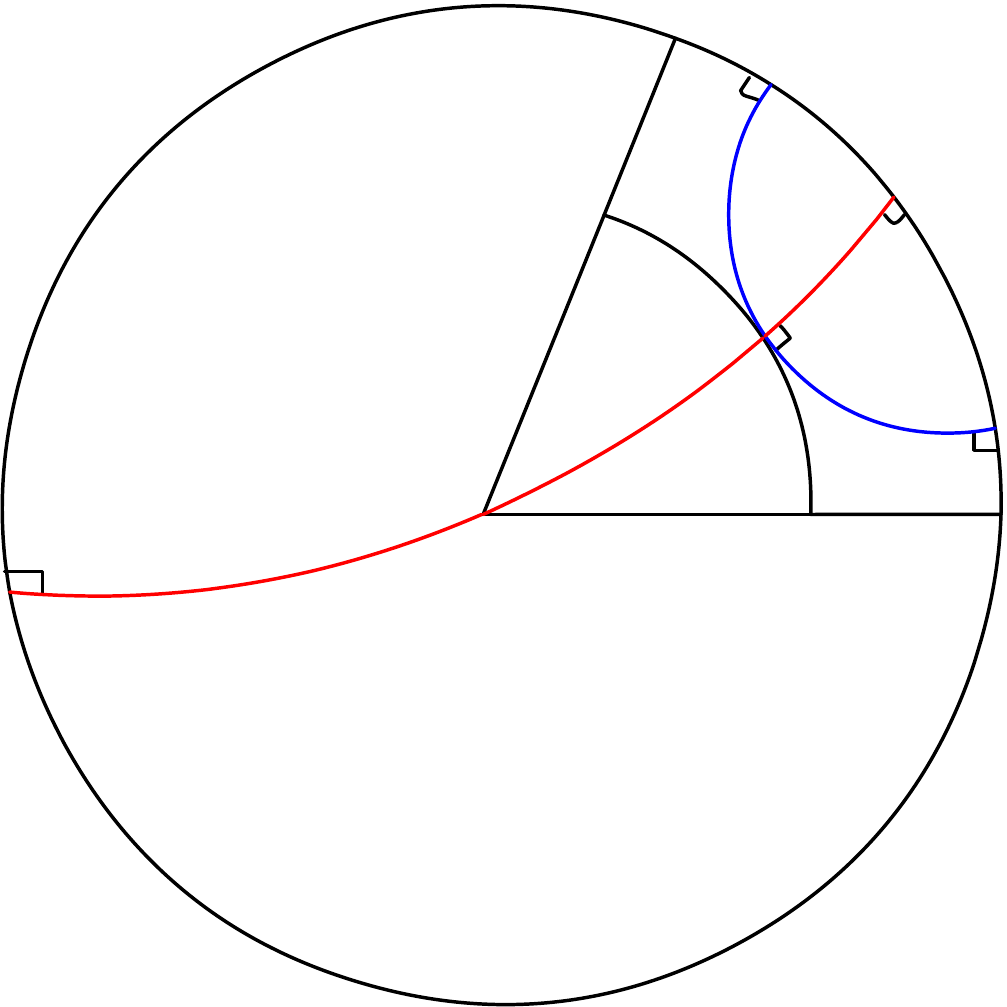}{A geometric construction in the
hyperbolic disc model
$$ {
 \at{-4\pscm}{4.2\pscm}{$\mathcal{F}$}
 \at{-5.1\pscm}{2\pscm}{$\mathcal{V}$}
 \at{-5\pscm}{3\pscm}{$0$}
 \at{-3.3\pscm}{4.3\pscm}{$\tilde{\zeta}$}
 \at{-6.3\pscm}{3.3\pscm}{\color{red}$p$\color{black}}
 \at{-4.3\pscm}{5.3\pscm}{\color{blue}$g$\color{black}}
}
$$}{5.5cm}

\begin{lemma} \label{greentheoremlemma}
Let $z=t+is$ with fixed $t$. Let $x\in C_{q,\delta}$ and
suppose that $f(z,x)$ is a $C^2$ function for which
$(\partial_z-\Delta_x)f(z,x)=0$. Write $f=u+iv$ for some real-valued
functions $u$ and $v$. Then we have
\begin{align*}
\partial_s \parallel f(t+is,\cdot)\parallel_{2;C_{q,\delta}}^2 = 4\int_{\partial
C_{q,\delta}} (v \partial_n u - u \partial_n v)d \rho.
\end{align*}
\end{lemma}
\begin{proof}
Since $t$ is fixed, we have that $\partial_z = (-i/2) \partial_s$. This
together with the heat equation imply
\begin{align*}
\Delta_x f &= \Delta_x u + i \Delta_x v = \partial_z u + i \partial_z v \\
&= -\frac{i}{2} \partial_s u + \frac{1}{2} \partial_s v.
\end{align*}
Then we can write
\begin{align*}
\partial_s \parallel f(t+is,\cdot)\parallel_{2,C_{q,\delta}}^2 &=
\partial_s \int_{C_{q,\delta}} (u^2 + v^2)d \mu\\
&= 2\int_{C_{q,\delta}} (u\partial_s u + v\partial_s v)d \mu\\
&=4\int_{C_{q,\delta}} (-u\Delta_x v + v\Delta_x u)d \mu\\
&= 4\int_{\partial C_{q,\delta}} (-u\partial_n v + v\partial_n
u)d \rho
\end{align*}
where the last equality follows from Green's theorem.
\end{proof}

\begin{cor}\label{constantinsnorm}
For fixed $t>0$ and $\zeta \in \partial C_{q,\delta}$, the
$L^2$-norm
\begin{align*}
\parallel P_{q,\delta}(t+is, \cdot, \zeta)\parallel_{2;C_{q,\delta}}
\end{align*}
is constant in $s$.
\end{cor}
\begin{proof}
For fixed $\zeta \in \partial C_{q,\delta}$, define
\begin{align*}
f(t+is,x) = P_{q,\delta}(t+is, x , \zeta)
\end{align*}
and apply the previous lemma. Note that the Poisson kernel vanishes
at boundary values, so that the integrand in the previous lemma is
identically zero.
\end{proof}

\begin{lemma}\label{supremumboundlemma}
\mbox{}
\begin{enumerate}[(a)]
\item
Let $0< \varepsilon < \delta < 1/2$, and $t_0 \ge 0$. There exists a
number $C$ independent of $q$ such that for $0 \le t \le t_0$
\begin{align*}
\sup_{\substack{x \in C_{q,\varepsilon} \\ \zeta \in \partial
C_{q,\delta}}} |K_{M_{q}}-K_{C_{q}}|(t,x,\zeta) \le
C
\end{align*}
\item
Let $0 < \varepsilon < \delta$. For fixed $t > 0$ and $\zeta \in
\partial C_{q,\delta}$,
\begin{align*}
\parallel(K_{M_{q}}-K_{C_{q}})(t+is, \cdot, \zeta)\parallel_{2;C_{q,\varepsilon}}
\le C \sqrt{1+|s|}
\end{align*}
\end{enumerate}
\end{lemma}

\begin{proof}
The inclusion of the fundamental groups gives the lower bound
\begin{align*}
(K_{M_{q}}-K_{C_{q}})(t,x,\zeta)\ge 0.
\end{align*}
The function $(K_{M_{q}}-K_{C_{q}})(t,x,\zeta)$ is a
solution to the heat equation in $t$ and $x$ on the domain
$C_{q,\varepsilon}$ with zero initial data. By applying the
maximum principle and the positivity of the heat kernel
$K_{C_{q}}$ as in Lemma \ref{threeintegralslemma}, we can write
\begin{align*}
(K_{M_{q}}-K_{C_{q}})(t,x,\zeta) \le \sup_{\substack{z \in
\partial C_{q,\varepsilon} \\ w \in \partial
C_{q,\varepsilon}\\ 0\le \tau \le t}} K_{M_{q}}(\tau,z,w).
\end{align*}
Using the above inequality together with the first bound in
Proposition \ref{kerneluniformbounds} we complete part (a) of this
lemma.

\hskip 0.2in To prove part (b), first note that from part (a) we have that the sup norm is bounded. Consequently, the $L^2$ norm is bounded and we can
write
\begin{align*}
\parallel (K_{M_{q}}-K_{C_{q}})(t+is,\cdot,\zeta)\parallel_{
2;C_{q,\varepsilon}}^2 \le C_1
\end{align*}
for some $C_1$ independent of $q$. We then substitute
$(K_{M_{q}}-K_{C_{q}})(t,x,\zeta)$ for $f$ in the
statement of Lemma \ref{greentheoremlemma} and use the bounds in
Proposition \ref{kerneluniformbounds} to write
\begin{align*}
\Big| \partial_s \parallel
(K_{M_{q}}-K_{C_{q}})(t+is,\cdot,\zeta)\parallel_{
2;C_{q,\varepsilon}}^2 \Big|  &\le  4\int_{\partial C_{q,\delta}} \left|-u\partial_n v + v\partial_n
u\right| d \rho \\
&\le C \cdot \text{len}(\partial C_{q,\delta}) = C \sqrt{4\pi\varepsilon / q + \varepsilon^2}\\
&\le C_2\,\,
\end{align*}
for some $C_2$ independent of $q$ (since $q >2$ and $\varepsilon < 1/2$). Integrate the last
inequality and get
\begin{align*}
\parallel (K_{M_{q}}-K_{C_{q}})(t+is,\cdot,\zeta)\parallel_{
2;C_{q,\varepsilon}}^2 \le C_1 + C_2|s|.
\end{align*}
We complete the proof by taking the square root.
\end{proof}

\section{Analysis of elliptic heat trace}
\label{Analysis of elliptic heat trace}

\hskip .2in In this section, we estimate the trace of the heat
kernel on the hyperbolic cone $C_q$ both over all of $C_q$ as well
as over the truncated conical region $C_q\backslash C_{q,\delta}$ .
The main technical computation is Proposition
\ref{thmintegraloverdifferencecusps}, which is analogous to Theorem
3.1 of \cite{JoLu 97a}. The remainder of the section states corollaries which we derive from 
Proposition
\ref{thmintegraloverdifferencecusps}.
\begin{prop}\label{thmintegraloverdifferencecusps}
For any $\delta > 0$ and $z = t + is$ with $t > 0$, set
\begin{equation*}
\eta = \frac{t}{4(t^2+s^2)} \textrm{\quad and \quad} \gamma =
\log\left(1+\left(\frac{\delta}{2\pi}\right)^2\right).
\end{equation*}
Then, if we let $\zeta_{\mathbb{Q}}$ denote the Riemann zeta
function, we have the bound
\begin{equation*}
\Bigg|\int_{C_q\backslash C_{q, \delta}}
(K_{C_q}-K_{\mathbb{H}})(z,x,x)d\mu(x)\Bigg| \le
\frac{e^{-t/4}}{\sqrt{\pi|z|}} \left(\frac{\delta}{2\pi}\right)^{-2\eta\gamma} \bigg[\zeta_{\mathbb{Q}} (1+2\eta\gamma)+\pi \bigg]\,\,.
\end{equation*}
\end{prop}

\begin{proof}
This first part of the proof (up to the integral representation
given in the equation (\ref{Iqdelta})) follows the very same ideas
as in the computation of the elliptic heat trace presented in
Theorem \ref{ellipticheattrace}. However, since this is a rather
technical proposition, we will present it in detail. 

\hskip 0.2in The domain of
integration can be modeled by the region in $\mathbb{H}$ described
by
\begin{equation*}
\{(\rho, \theta): r(\delta,q)\le \rho < \infty, 0\le\theta<2\pi \}
\textrm{\quad where } r(\delta,q)=\cosh^{-1}(1+q\delta/2\pi).
\end{equation*}
In these coordinates, write any $x \in \mathbb{H}$ as $x =
(\rho,\theta)$ and, referring to (\ref{beardon}), let
\begin{equation*}
a(n,q,\rho)=d(x,\gamma_q^n x) = \cosh^{-1}(1+ 2 \sin^2(\pi n
/q)\sinh^2 (\rho)).
\end{equation*}
In these coordinates, the integrand under consideration depends
solely on $\rho$. Therefore, we are studying the explicit expression
\begin{align*}
I_{q,\delta}(z)&=\int_{C_q\backslash C_{q, \delta}} (K_{C_q}-K_{\mathbb{H}})(z,x,x)d\mu(x)\\
&=\int_{r(\delta,q)}^{\infty} \int_{0}^{2\pi} \sum_{n=1}^{q-1} K_{\mathbb{H}}(z,d(x,\gamma_q^n x))q^{-1} \sinh(\rho)d\theta d\rho\\
&=\frac{2\pi}{q} \int_{r(\delta,q)}^{\infty} \sum_{n=1}^{q-1} K_{\mathbb{H}}(z,a(n,q,\rho))) \sinh(\rho) d\rho\\
&=\frac{ e^{-z/4}}{q \sqrt{\pi (2z)^3}} \sum_{n=1}^{q-1}
\int_{r(\delta,q)}^{\infty} \int_{a(n,q,\rho)}^{\infty} \frac{u
e^{-u^2/4z} \sinh(\rho) du d\rho}{\sqrt{\cosh u -1- 2\sin^2(\pi n
/q)\sinh^2 (\rho)}}.
\end{align*}
Now let
\begin{equation*}
x = \cosh \rho \textrm{\quad and \quad} b(n,q,x) =
\cosh^{-1}(1+2\sin^2(n\pi/q)(x^2-1))
\end{equation*}
so then
\begin{equation*}
I_{q,\delta}(z) = \frac{ e^{-z/4}}{q \sqrt{\pi (2z)^3}}
\sum_{n=1}^{q-1} \int_{1+q\delta/2\pi}^{\infty}
\int_{b(n,q,x)}^{\infty} \frac{u e^{-u^2/4z} du dx} {\sqrt{\cosh u
-1 + 2\sin^2(\pi n /q)(1-x^2)}}.
\end{equation*}
To continue, let us interchange the order of integration in this
iterated integral. To do so, we set
\begin{equation*}
c(n,q,u) = \sqrt{1+ \frac{\cosh u -1}{2\sin^2(n\pi/q)}}
\textrm{\quad and \quad} d(n,q,\delta)=b(n,q,1+q\delta/2\pi)
\end{equation*}
so then we get
\begin{align*}
I_{q,\delta}(z) =& \frac{ e^{-z/4}}{q \sqrt{\pi (2z)^3}}
\sum_{n=1}^{q-1} \int_{d(n,q,\delta)}^{\infty}
\int_{1+q\delta/2\pi}^{c(n,q,u)} \frac{u e^{-u^2/4z} dx du}
{\sqrt{\cosh u -1 + 2\sin^2(\pi n /q)(1-x^2)}}\\
=&\frac{ e^{-z/4} }{4q \sqrt{\pi z^3}} \sum_{n=1}^{q-1}
\frac{1}{\sin(n\pi/q)}
\int_{d(n,q,\delta)}^{\infty} u e^{-u^2/4z}
\cos^{-1}\left(\frac{\sqrt{2}\sin(n\pi/q)(1+q\delta/2\pi)}{\sqrt{\cosh
u -1 + 2\sin^2(\pi n /q)}}\right)du.
\end{align*}
In the above computations, we used the elementary formula
\begin{equation*}
\int_a^b \frac{dx}{\sqrt{c-dx^2}} = -\frac{1}{\sqrt{d}} \Big(
\cos^{-1}(b\sqrt{d/c})-\cos^{-1}(a\sqrt{d/c})\Big)
\end{equation*}
in order to compute the integral with respect to $x$. Continuing,
let us define
\begin{equation*}
f(u,q,n)=\cos^{-1}\Bigg(\frac{\sqrt{2}\sin(n\pi/q)(1+q\delta/2\pi)}{\sqrt{\cosh
u -1 + 2\sin^2(n\pi /q)}}\Bigg)
\end{equation*}
and let $f'$ denote the derivative with respect to $u$. Therefore
after integrating by parts, we are left with the expression
\begin{equation}\label{Iqdelta}
I_{q,\delta}(z) = \frac{ e^{-z/4} }{2q\sqrt{\pi z}} \sum_{n=1}^{q-1} \frac{1}{\sin(n\pi/q)}\int_{d(n,q,\delta)}^{\infty} e^{-u^2/4z} f'(u,q,n)du.
\end{equation}
\hskip .20in  To finish, we shall now estimate the integral in
(\ref{Iqdelta}). First, observe that
\begin{equation*}
f'(u,q,n) = \frac{(1+q\delta/2\pi)}{\sqrt{2}}\cdot
\frac{\sin(n\pi/q)\sinh u \left(\cosh u - 1 + 2\sin^2(n\pi/q)\right)^{-1}}{\sqrt{\cosh u -1 +2\sin^2(n\pi/q) - 2\sin^2(n\pi/q) (1+q\delta/2\pi)^2 }}.
\end{equation*}
Clearly, it follows that $f'> 0$. Observing that
\begin{equation*}
\int_{d(n,q,\delta)}^{\infty} f'(u,q,n)d u= \frac{\pi}{2},
\end{equation*}
we proceed in bounding the integral $I_{q,\delta}(z)$ as follows:
\begin{align}\label{Iqdeltaprebound}
\notag |I_{q,\delta}(z)| &= \Bigg|\frac{ e^{-z/4} }{2q\sqrt{\pi z}} \sum_{n=1}^{q-1} \frac{1}{\sin(n\pi/q)}\int_{d(n,q,\delta)}^{\infty} e^{-u^2/4z} f'(u,q,n)du\Bigg|\\
\notag &\le \frac{ e^{-t/4} }{2q\sqrt{\pi |z|}} \sum_{n=1}^{q-1} \frac{1}{\sin(n\pi/q)}\int_{d(n,q,\delta)}^{\infty} e^{-\eta u^2} f'(u,q,n)du\\
\notag &\le \frac{ e^{-t/4} }{2q\sqrt{\pi |z|}} \sum_{n=1}^{q-1} \frac{1}{\sin(n\pi/q)} \Bigg(\sup_{u \ge d(n,q,\delta)} e^{-\eta u^2} \Bigg)\int_{d(n,q,\delta)}^{\infty}  f'(u,q,n)du\\
&\le \frac{ \sqrt{\pi}e^{-t/4} }{4q \sqrt{|z|}} \sum_{n=1}^{q-1}
\frac{1}{\sin(n\pi/q)} e^{-\eta d(n,q,\delta)^2} .
\end{align}
Recalling the definition of $d(n,q,\delta)$ and using that $\log (x)
\le \log(x+\sqrt{x^2-1})=\cosh^{-1}(x)$ for $x\ge 1$, we note that
\begin{align*}
\exp(-\eta d(n,q,\delta)^2) &= \exp\Bigg(-\eta
\big(\cosh^{-1}\big)^2\Big(1+\frac{q\delta\sin^2(n\pi/q)(4\pi+q\delta)}{2\pi^2}\Big)\Bigg)\\
&\le \exp\Bigg(-\eta
\log^{2}\Big(1+\frac{q\delta\sin^2(n\pi/q)(4\pi+q\delta)}{2\pi^2}\Big)\Bigg).
\end{align*}
This allows us to further bound the last inequality of
(\ref{Iqdeltaprebound}) thus giving
\begin{align}\label{Iqdeltabound}
\notag |I_{q,\delta}(z)|&\le \frac{ \sqrt{\pi}e^{-t/4} }{4q \sqrt{|z|}} \sum_{n=1}^{q-1} \frac{1}{\sin(n\pi/q)} \exp\Bigg(-\eta \log^2\bigg(1+\frac{q\delta\sin^2(n\pi/q)(4\pi+q\delta)}{2\pi^2}\bigg)\Bigg) \\
&\le \frac{ \sqrt{\pi}e^{-t/4} }{4q \sqrt{|z|}} \sum_{n=1}^{q-1}
\frac{1}{\sin(n\pi/q)} \Bigg( 1 +
\frac{q\delta\sin^2(n\pi/q)(4\pi+q\delta)}{2\pi^2}\Bigg)^{-\eta
h(n,q,\delta)}
\end{align}
where
\begin{equation*}
h(n,q,\delta) =
\log\Bigg(1+\frac{q\delta\sin^2(n\pi/q)(4\pi+q\delta)}{2\pi^2}\Bigg).
\end{equation*}
We continue by using the set decomposition
\begin{equation*}
\{1,2, \cdots ,q-1\} = \mathcal{I}_1 \cup \mathcal{I}_2
\end{equation*}
where
\begin{equation*}
\mathcal{I}_1 = \{n: 1 \le n \le q-1 \textrm{ and } \sin(n\pi/q) \le
\sqrt{2}/2 \}
\end{equation*}
so then
\begin{equation*}
\mathcal{I}_2 = \{n: 1 \le n \le q-1 \textrm{ and } \sin(n\pi/q) >
\sqrt{2}/2 \}.
\end{equation*}
We then write the sum in (\ref{Iqdeltabound}) as $F_1(z) + F_2(z)$,
where $F_1(z)$ is the sum over the integers in $\mathcal{I}_1$, and
$F_2(z)$ is the sum over the integers in $\mathcal{I}_2$. Let us
first consider the sum yielding $F_1(z)$. Note that one has $2[q/4]$
such terms indexed by $n = 1, 2, \cdots , [q/4]$ and $n = q - 1, q -
2, \cdots , q - [q/4]$. Since $\sin x = \sin(\pi  - x)$, we only sum
$n = 1, 2, \cdots , [q/4]$ and multiply by 2. It is elementary that
if $0 \le \sin x \le \sqrt{2}/2$ then $\sin x \ge x/2$. With all
this, the contribution to (\ref{Iqdeltabound}) of these terms is
bounded from above by
\begin{align}\label{F1bound}
\notag F_1(z) &= \frac{ \sqrt{\pi}e^{-t/4} }{4q \sqrt{|z|}} \sum_{n \in \mathcal{I}_1} \frac{1}{\sin(n\pi/q)} \Bigg( 1 + \frac{q\delta\sin^2(n\pi/q)(4\pi+q\delta)}{2\pi^2}\Bigg)^{-\eta h(n,q,\delta)}\\
\notag &\le \frac{ \sqrt{\pi}e^{-t/4} }{4q \sqrt{|z|}}  \cdot 2
\sum_{n=1}^{[q/4]} \frac{2q}{n\pi}\Bigg( 1 +
\frac{q\delta(n\pi/2q)^2(4\pi+q\delta)}{2\pi^2}\Bigg)^{-\eta
h(n,q,\delta)}\\
\notag &\le \frac{ e^{-t/4} }{\sqrt{\pi|z|}} \sum_{n=1}^{[q/4]} n^{-1} \left( n^2 \frac{\delta^2}{4\pi^2}\right)^{-\eta
\log(1 + \delta^2/(2\pi)^2)}\\
\notag &= \frac{ e^{-t/4} }{\sqrt{\pi|z|}} \left( \frac{\delta^2}{4\pi^2} \right)^{-\eta \gamma} \sum_{n=1}^{[q/4]} n^{-1-2\eta\gamma} \\
&\le \frac{ e^{-t/4} }{\sqrt{\pi|z|}} \left( \frac{\delta}{2\pi} \right)^{-2\eta \gamma} \zeta_{\mathbb{Q}}(1+2\eta \gamma).
\end{align}
The last expression occurs in the upper bound asserted in the
statement of the proposition.

\hskip .2in We now study the terms in (\ref{Iqdeltabound}) which
yields $F_2(z)$, namely those for which $\sin(n\pi/q) > \sqrt{2}/2$.
There are at most $q - 1 - 2[q/4]$ of these terms, which we bound
using the coarse but adequate estimate $q - 1 - 2[q/4] \le q$. With
this, we have the bounds
\begin{align}\label{F2bound}
\notag F_2(z) &= \frac{ \sqrt{\pi}e^{-t/4} }{4q \sqrt{|z|}} \sum_{n \in \mathcal{I}_2} \frac{1}{\sin(n\pi/q)} \Bigg( 1 + \frac{q\delta\sin^2(n\pi/q)(4\pi+q\delta)}{2\pi^2}\Bigg)^{-\eta h(n,q,\delta)}\\
\notag &\le \frac{ \sqrt{\pi}e^{-t/4} }{4q \sqrt{|z|}} \sum_{n=1}^{q} \frac{2}{\sqrt{2}} \Bigg( \frac{ (q\delta \sqrt{2}/2)^2}{2\pi^2}\Bigg)^{-\eta \log(1 + \delta^2/(2\pi)^2)}\\
\notag &\le \frac{ \sqrt{\pi}e^{-t/4} }{4q \sqrt{|z|}} \sum_{n=1}^{q} \frac{2}{\sqrt{2}} \left(\frac{\delta^2}{4\pi^2}\right)^{-\eta \gamma}\\
&\le \frac{ \pi e^{-t/4} }{ \sqrt{\pi|z|}} \left(\frac{\delta}{2\pi}\right)^{-2\eta \gamma}.
\end{align}
Combining (\ref{F1bound}) and (\ref{F2bound}) completes the proof of the proposition.
\end{proof}

\begin{rmk}
\emph{It is interesting to compare the bound obtained in Proposition
\ref{thmintegraloverdifferencecusps} (elliptic degeneration) with
the bound given in Theorem 3.1 of \cite{JoLu 97a} (hyperbolic
degeneration). When doing so, one sees a remarkable similarity
between the upper bounds, which one would expect in some
philosophical level. However, structurally, one can make a more
precise comparison between the two results. As in \cite{JoLu 97a},
our results rely on the coefficient of the special value of the
Riemann zeta function. When comparing Theorem 3.1 of \cite{JoLu 97a}
and Proposition \ref{thmintegraloverdifferencecusps} above, we need
to keep in mind that the domain of integration in the setting of
hyperbolic degeneration of \cite{JoLu 97a} consists of two
integrals, whereas in elliptic degeneration there is one component.
Furthermore, one could easily modify the expressions involving $\delta$ in Theorem 3.1 of
\cite{JoLu 97a} so that they match their counterparts here.
As we will see, the additive factor next to the zeta value present here (but not in \cite{JoLu 97a}),
will play a significant role in the behavior of the trace near $t = 0$.
} 
\end{rmk}

\begin{rmk}\label{Iqzero}
\emph{For any $q$, we have that
\begin{equation*}
I_{q,0}(z) = \frac{ e^{-z/4} }{q\sqrt{16\pi z}} \sum_{n=1}^{q-1}
\int_{0}^{\infty}\frac{e^{-u^2/4z} \cosh(u/2)}{\sinh^2 (u/2) +
\sin^2(n \pi  /q)}du.
\end{equation*}
The result follows by taking $\delta = 0$ in the integral
representation given by the equation (\ref{Iqdelta}). The integral
above is one of the terms (corresponding to one elliptic fixed
point) in the elliptic heat trace formula as detailed in Theorem
\ref{ellipticheattrace}}.
\end{rmk}

\begin{cor}\label{uniformboundscorollary}
Let $z=t+is$ with $t>0$. Then the following bounds and limits hold
uniformly in $q.$
\begin{enumerate}[(a)]
\item
There is a constant $C$ (depending on $\delta, t,$ and the number of degenerating elliptic elements) such that
\begin{equation*}
\Bigg|\int_{C_{q}\backslash C_{q, \delta}}
(K_{C_{q}}-K_{\mathbb{H}})(t+is,x,x)d\mu(x)\Bigg| \le C
(1+|s|)^{3/2}.
\end{equation*}
\item
The following limit holds
\begin{equation*}
\lim_{\delta \to \infty}\Bigg|\int_{C_{q}\backslash C_{q,
\delta}} (K_{C_{q}}-K_{\mathbb{H}})(t+is,x,x)d\mu(x)\Bigg|=0.
\end{equation*}
\item
For any $0<\varepsilon < 1/2$, there exists a constant $C$ (depending on $t$) such that
\begin{equation*}
\Bigg|\int_{C_{q, \varepsilon}}
(K_{M_{q}}-K_{C_{q}})(t+is,x,x)d\mu(x)\Bigg| \le C
(1+|s|)^{3/2}.
\end{equation*}
\item
For fixed $z=t+is$ there exists a constant $C$ such that
\begin{equation*}
\Bigg|\int_{C_{q, \varepsilon}}
(K_{M_{q}}-K_{C_{q}})(t+is,x,x)d\mu(x)\Bigg| \le C
\sqrt{\varepsilon}.
\end{equation*}
\end{enumerate}
\end{cor}

\begin{proof}
For part (a), the quantity above is a sum of integrals over
$C_{q_k}\backslash C_{q_k, \delta}$ for each $q_k\in
q$. Recall that $\zeta_{\mathbb{Q}}(1+\varepsilon)\sim
\varepsilon^{-1}$ for $\varepsilon$ near 0, so then the result
follows immediately from Proposition
\ref{thmintegraloverdifferencecusps}. Part (b) follows directly by
inspection from Proposition \ref{thmintegraloverdifferencecusps}.

\hskip 0.2in Part (c) estimates the integral (II) in Theorem \ref{pointwiseanduniformconvergence}. For simplicity, we can work with one
degenerating cone whose order is $q$. Choose $\delta$ such that
$\varepsilon <\delta <1/2$. For any $x,y \in
C_{q,\varepsilon}$ with $y$ fixed, we can use the Poisson
kernel to write
\begin{align*}
(K_{M_q} - K_{C_{q}})(t,x,y) = \int_{0}^{t}\int_{\partial
C_{q,\delta}}
P_{q,\delta}(t-\sigma,x,\zeta)(K_{M_q} -
K_{C_{q}})(\sigma,\zeta,y)d\rho(\zeta)d\sigma
\end{align*}
which naturally extends to complex values of $t$. Set $y=x$ and
integrate to get
\begin{align*}
\int_{C_{q,\varepsilon}}&(K_{M_q} -
K_{C_{q}})(t+is,x,x)d\mu(x) \\
&= \int_{C_{q,\varepsilon}} \int_{0}^{t+is}\int_{\partial
C_{q,\delta}}
P_{q,\delta}(t+is-\sigma,x,\zeta)(K_{M_q} -
K_{C_{q}})(\sigma,\zeta,x)d\rho(\zeta)d\sigma d\mu(x)\\
&= \int_{\partial C_{q,\delta}}
\int_{0}^{t+is}\int_{C_{q,\varepsilon}}
P_{q,\delta}(t+is-\sigma,x,\zeta)(K_{M_q} -
K_{C_{q}})(\sigma,\zeta,x)d\mu(x)d\sigma d\rho(\zeta)
\end{align*}
where the interchange of the order of integration is justified by
the fact that the integrand is continuous. Break up the above
integral as
\begin{align*}
\int_{\partial C_{q,\delta}}
\int_{0}^{t+is}\int_{C_{q,\varepsilon}} =& \int_{\partial
C_{q,\delta}}
\int_{0}^{t/2}\int_{C_{q,\varepsilon}}\\
&+\int_{\partial C_{q,\delta}}
\int_{t/2}^{t/2+is}\int_{C_{q,\varepsilon}}+\int_{\partial
C_{q,\delta}}
\int_{t/2+is}^{t+is}\int_{C_{q,\varepsilon}}.
\end{align*}
Label the above three integrals A,B, and C. The path of
integration in the variable $\sigma$ consists of three linear segments.

\hskip 0.2in \emph{Estimating integral A.} For $0 \le \sigma \le
t/2$, the supremum bound from Lemma \ref{supremumboundlemma} part (a) gives
the bound
\begin{align*}
\parallel (K_{M_{q}}-K_{C_{q}})(\sigma,\zeta,\cdot)
\parallel_{2;C_{q,\varepsilon}} \le M
\sqrt{\textrm{vol}(C_{q,\varepsilon})} = M \sqrt{\varepsilon},
\end{align*}
where $M$ is independent of $q$. For $0\le \sigma \le t/2$,
inclusions of domains gives the inequality
\begin{align*}
\parallel P_{q,\delta}(t+is-\sigma,\cdot,\zeta)
\parallel_{2;C_{q,\varepsilon}} \le \parallel P_{q,\delta}(t+is-\sigma,\cdot,\zeta)
\parallel_{2;C_{q,\delta}}.
\end{align*}
The $L^2$-norm above is constant in $s$, by Corollary
\ref{constantinsnorm}, and decreasing in $t$, by Proposition
\ref{poissonkernelproperties} part (c). Then we can write
\begin{align*}
\parallel P_{q,\delta}(t+is-\sigma,\cdot,\zeta)
\parallel_{2;C_{q,\delta}} &\le \parallel P_{q,\delta}(t-\sigma,\cdot,\zeta)
\parallel_{2;C_{q,\delta}}\\
&\le \parallel P_{q,\delta}(t/2,\cdot,\zeta)
\parallel_{2;C_{q,\delta}}\\
&\le \sup_{x\in C_{q,\delta}}
P_{q,\delta}(t/2,x,\zeta)
\sqrt{\textrm{vol}(C_{q,\delta})}\\
&\le M \sqrt{\delta},
\end{align*}
where the last inequality follows from Proposition
\ref{poissonkernelproperties} part (a). Recall that the length of the
boundary of $C_{q,\delta}$ is $\sqrt{4\pi \delta/q +
\delta^2}$. Using these facts and the Cauchy-Schwarz
inequality, we get the bound for integral A from above by
\begin{align}\label{boundintegralA}
\notag\int_{\partial C_{q,\delta}} \int_{0}^{t/2} \parallel
P_{q,\delta}(t+is-\sigma,\cdot,\zeta) &
\parallel_{2;C_{q,\varepsilon}}  \parallel (K_{M_{q}}-K_{C_{q}})(\sigma,\zeta,\cdot)
\parallel_{2;C_{q,\varepsilon}} d\sigma d\rho(\zeta)\\
\notag &\le M^2 \sqrt{\varepsilon\delta}\int_{\partial
C_{q,\delta}} \int_{0}^{t/2}d\sigma d\rho(\zeta)\\
\notag&\le M^2 \sqrt{\varepsilon\delta}  \sqrt{4\pi \delta/q +
\delta^2} \cdot t\\
&\le M^2\delta \sqrt{\varepsilon(4\pi  +
\delta)} \cdot t
\end{align}
since $q$ is an positive integer grater than 2.

\hskip 0.2in \emph{Estimating integral B.} For $\sigma$ on the line
segment from $t/2$ to $t/2+is$, we can use Lemma
\ref{supremumboundlemma} part (b) to obtain the inequality
\begin{align*}
\parallel (K_{M_{q}}-K_{C_{q}})(\sigma,\zeta,\cdot)
\parallel_{2;C_{q,\varepsilon}} \le
M \sqrt{1+|\textrm{Im}(\sigma)|}.
\end{align*}
As before, using the Cauchy-Schwarz inequality and Proposition
\ref{poissonkernelproperties} part (c), we can bound integral B from
above by
\begin{align}\label{boundintegralB}
\notag\Bigg| \int_{\partial C_{q,\delta}} \int_{t/2}^{t/2+is}
\parallel P_{q,\delta}(t+is-\sigma,\cdot,\zeta) &
\parallel_{2;C_{q,\varepsilon}}  \parallel (K_{M_{q}}-K_{C_{q}})(\sigma,\zeta,\cdot)
\parallel_{2;C_{q,\varepsilon}} d\sigma d\rho(\zeta)\Bigg|\\
\notag &\le M^2 \sqrt{\delta} \sqrt{4\pi \delta/q + \delta^2} \cdot
\frac{2}{3}(1+|s|)^{3/2}\\
&\le M^2\delta \sqrt{4\pi  + \delta} \cdot
(1+|s|)^{3/2}.
\end{align}
\hskip 0.2in \emph{Estimating integral C.} Using similar arguments
as before, we can bound integral C from above by
\begin{align}\label{boundintegralC}
\notag \Bigg| \int_{\partial C_{q,\delta}}
\int_{t/2+is}^{t+is} & \sup_{x\in C_{q,\delta}}
P_{q,\delta}(t+is-\sigma,\cdot,\zeta)  \cdot
\sqrt{\textrm{vol}(C_{q,\delta})} M \sqrt{1+|s|}d\sigma d\rho(\zeta)\Bigg|\\
\notag &\le M t \sqrt{4\pi \delta/q + \delta^2)}
\sqrt{\delta}  \sqrt{1+|s|} \sup_{\substack{ x\in
C_{q,\varepsilon} \\ \tau \in (0,t/2)}}
P_{q,\delta}(\tau,x,\zeta)\\
&\le M^2 \delta \sqrt{4\pi + \delta} \cdot t \cdot \sqrt{1+|s|}.
\end{align}
By combining (\ref{boundintegralA}),(\ref{boundintegralB}), and
(\ref{boundintegralC}) we complete the proof of part (c) of this
corollary.

\hskip 0.2in It remains to prove part (d). In this direction, note
that integrals A and C from part (c) lines (\ref{boundintegralA}) and
(\ref{boundintegralC}) depend explicitly on $\sqrt{\varepsilon}$.
Since $z=t+is$ is fixed, we can improve the bound of integral B,
namely we can show that integral B depends on $\sqrt{\varepsilon}$.
Let $\sigma = a+ib$ with $a>0$. Then we can write
\begin{align*}
|(K_{M_{q}}-K_{C_{q}})(\sigma,x,y)| \le
\exp(b^2/4a)a^{3/2}(a^2+b^2)^{3/4}(K_{M_{q}}-K_{C_{q}})(\tau,x,y)
\end{align*}
where $\tau = (a^2+b^2)/a$. For $|b|\le |s|$, the right hand side is
bounded independent of $q$ according to Lemma
\ref{supremumboundlemma} part (a). The supremum bound leads to the
$L^2$-bound
\begin{align*}
\parallel (K_{M_{q}}-K_{C_{q}})(\sigma,\zeta,\cdot)
\parallel_{2;C_{q,\varepsilon}} \le M
\sqrt{\textrm{vol}(C_{q_k,\varepsilon})} = M \sqrt{\varepsilon},
\end{align*}
which holds for the relevant range of $b$. With this bound, we carry
on as in (\ref{boundintegralB}) and obtain that integral B is also
dependent on $\sqrt{\varepsilon}$. This completes the proof.
\end{proof}

\begin{lemma}\label{constantboundlemma}
For fixed $t>0$, the integral
\begin{align*}
\int_{M_{q}\backslash C_{q,\varepsilon}}|(K_{M_{q}} -
K_{\mathbb{H}})(t+is,x,x)|d\mu(x)
\end{align*}
is bounded as a function of $s$, independently of $q.$
\end{lemma}

\begin{proof}
From Proposition \ref{kerneluniformbounds} part (a), the above
integral can be bounded from above by
\begin{align}\label{constantboundlemma1}
\int_{M_{q}\backslash C_{q,\varepsilon}}
K_{M_{q}}(t,x,x)d\mu(x)+ \textrm{vol}(M_{q}\backslash
C_{q,\varepsilon})K_{\mathbb{H}}(z,0).
\end{align}

Using the periodization of the heat kernel, we can rewrite the integrand in (\ref{constantboundlemma1}) as a Stieltjes integral
\begin{align*}
K_{M_{q}}(t,x,x) = \sum_{\gamma \in \Gamma_q} K_{\mathbb{H}}(t, d_{\mathbb{H}}(\tilde{x},\gamma\tilde{x})
= \int_{0}^{\infty} K_{\mathbb{H}}(t,\rho) \ dN_{\Gamma_q}(x;\rho)
\end{align*}
where $N_{\Gamma_q}(x;\rho)$ counts the number of geodesics about $x$ whose length is bounded above by $\rho$.
Directly from (\ref{heatkernelonH1}), we obtain the following bound  
\begin{align*} 
K_{M_{q}}(t,x,x) \le 
\frac{100e^{-t/4}}{(4\pi t)^{3/2}}\int_{0}^{\infty} e^{-\rho^2/4t} \ dN_{\Gamma_q}(x;\rho).
\end{align*}
Using Lemma 4 of \cite{JoLu 95} which we apply to the function $f(\rho) = e^{-\rho^2/4t}$, we arrive at the following bound for the heat kernel
\begin{align}\label{constantboundlemma2}
K_{M_{q}}(t,x,x) \le 
\frac{Ce^{-t/4}}{(4\pi t)^{3/2}}.
\end{align}

As far as the volume term, denote by $g$ the genus of the family (which is independent of the degeneration parameter $q$) and by $\kappa$ the number of \emph{all} cusps and cones of the family (which also stays constant). Then we can write
\begin{align}\label{constantboundlemma3}
\textrm{vol}(M_{q}) \le 2\pi(2g-2+\kappa). 
\end{align}

Additionally, from  (\ref{heatkernelonH0}) which we extend here to complex time, we obtain the following bound 
\begin{align}\label{constantboundlemma4}
|K_{\mathbb{H}}(z,0)| \le \frac{e^{-t/4}}{4\pi t}.
\end{align}

Finally, the combination of (\ref{constantboundlemma1}) thorough (\ref{constantboundlemma4}), yields the bound
\begin{align}\label{constantboundlemma5}
\int_{M_{q}\backslash C_{q,\varepsilon}}|(K_{M_{q}} - K_{\mathbb{H}})(t+is,x,x)|d\mu(x) 
\le  \frac{2\pi(2g-2+\kappa)e^{-t/4}}{4\pi t} \left( \frac{C}{\sqrt{4\pi t}} +  1 \right)
\end{align}
which completes the proof.
\end{proof}

\section{Convergence of regularized heat traces}
\label{Convergence of regularized heat traces}

\hskip 0.2in In this section we will make use of the estimates from
previous sections to prove the convergence through elliptic degeneration of the regularized trace of the heat kernel on the $M_{q}$ to the regularized trace on the limiting surface $M_{\infty}$. The trace of the heat kernel alone diverges through degeneration since the degenerating elliptic elements converge to cusps in the limiting surface.
The result applies to elliptically degenerating families $M_{q}$ of finite volume, both in the compact and non-compact case. The convergence has a dual aspect, on the one hand pointwise and on the other hand uniform with respect to the time variable.

\hskip 0.2in We start by defining the degenerating trace of the heat kernel.
Let ${M_{q}}$ be a degenerating sequence of
connected, hyperbolic Riemann surface of finite volume having
limiting surface $M_{\infty}$. This means that each surface in the
family is realized as $\Gamma_{q}\backslash\mathbb{H}$. 
Let $DE(\Gamma_q)$ denote the subset of $E(\Gamma_q)$ consisting of the inconjugate primitive elliptic elements associated to the conical points we wish to degenerate into cusps.
We define the degenerating heat trace via the
integral
\begin{align*}
\text{\rm DTr}K_{M_{q}}(z) &=
\int_{C_{q}}(K_{C_{q}}-K_{\mathbb{H}})(z,x,x)d\mu(x)\\
&=\sum_{\gamma \in DE(\Gamma_{q})}
\int_{C_{\gamma}}(K_{C_{\gamma}}-K_{\mathbb{H}})(z,x,x)d\mu(x).
\end{align*}
for all complex values $z=t+is$ with $t>0$.
\begin{prop}
In the above setting, for any $t>0$ we have the equality
\begin{equation*}
\text{\rm DTr}K_{M_{q}}(z) = \frac{e^{-z/4}}{\sqrt{16\pi
z}}\sum_{\gamma\in DE(\Gamma_{q})}
\sum_{n=1}^{q_{\gamma}-1} \frac{1}{q_{\gamma}} \int_{0}^{\infty}
\frac{e^{-u^2/4z}\cosh(u/2)}{\sinh^2(u/2)+\sin^2(n\pi/q_{\gamma})}du.
\end{equation*}
\end{prop}
\begin{proof}
One can use the same arguments as in the proof of Theorem \ref{ellipticheattrace}.
\end{proof}

\hskip 0.2in The next result presents the behavior through degeneration of the heat kernel and its derivatives. Namely, we have the following theorem. For brevity, we only state the result. For details, we refer the reader to \cite{JoLu 95} and Theorem 1.3 of \cite{JoLu 97a} which one can easily adapt to the elliptic degeneration setting.

\begin{thm}\label{convergenceheatkernels}
Let $R_{q}$ denote either $M_{q}$ or $C_{q}$. For $i
= 1, 2$, let $\nu_i = \nu_i(q)$ be a tangent vector of unit
length based at $x_i \in R_{q}$ which converges as $q \to
\infty$. Denote by $\partial_{\nu_i,x_i}$ the directional
derivative with respect to the variable $x_i$ in the direction
$\nu_i$. Assume that either $x_1$ or $x_2$ is not a degenerating
conical point. Then
\begin{align}\label{convergenceheatkernels1}
\lim_{q \to \infty} K_{R_{q}}(z,x_1,x_2) &=
K_{R_{\infty}}(z,x_1,x_2) \\
\label{convergenceheatkernels2}\lim_{q \to \infty}
\partial_{\nu_i,x_i} K_{R_{q}}(z,x_1,x_2) &=
\partial_{\nu_i,x_i} K_{R_{\infty}}(z,x_1,x_2) \textrm{ for }
i=1,2\\
\label{convergenceheatkernels3}\lim_{q \to \infty}
\partial_{\nu_1,x_1} \partial_{\nu_2,x_2} K_{R_{q}}(z,x_1,x_2)
&= \partial_{\nu_1,x_1}
\partial_{\nu_2,x_2} K_{R_{\infty}}(z,x_1,x_2)
\end{align}

\hskip .2in
\begin{enumerate}[(a)]
\item Let A be a bounded set in the complex plane with $\inf_{z\in A}
\rm{Re}(z) > 0.$ For any $\varepsilon > 0$, the convergence is
uniform on $A \times R_{q} \backslash C_{q,\varepsilon}
\times R_{q} \backslash C_{q,\varepsilon}$.

\item We define $D_{\varepsilon,\varepsilon'}$ to be an $\varepsilon'$
neighborhood of the diagonal of $R_{q} \backslash
C_{q,\varepsilon} \times R_{q} \backslash
C_{q,\varepsilon}$. That is,
\begin{equation*}
D_{\varepsilon,\varepsilon'} = \{ (x_1,x_2) \in R_{q}
\backslash C_{q,\varepsilon} \times R_{q} \backslash
C_{q,\varepsilon} : d(x_1,x_2) < \varepsilon' \}
\end{equation*}
Let $B$ be a bounded set in the complex plane with $\inf_{z\in B}
\rm{Re}(z) \ge 0$. For any $\varepsilon > 0$ and $\varepsilon' > 0$,
the convergence is uniform on $B \times ((R_{q} \backslash
C_{q,\varepsilon} \times R_{q} \backslash
C_{q,\varepsilon}) \backslash D_{\varepsilon,\varepsilon'})$.
\end{enumerate}
\end{thm}

\hskip 0.2in In the course of the proof of the main result of this section, we analyze the behavior of three integrals. These integrals represent the regularized trace of the heat kernel from which we exclude the identity term. In other words, these integrals amount the contribution to the regularized trace of the hyperbolic and elliptic heat traces minus the contribution of the degenerating elliptic trace. We will only look at the compact case. The adaptation of the Lemma \ref{threeintegralslemma} below to the non-compact case should immediately follow.

\begin{lemma}\label{threeintegralslemma}
Let $M =\Gamma \backslash \mathbb{H}$ be a compact connected hyperbolic
surface, having $m$ degenerating elliptic elements.
Then for every sufficiently small $\varepsilon$, we have
\begin{align*}
(\text{\rm HTr}K_M+\text{\rm ETr}K_M - \text{\rm DTr}K_M)(t+is)= &  \int_{M\backslash(m \times C_{q,\varepsilon})} (K_M-K_{\mathbb{H}})(t+is,x,x)d\mu(x)\\
                      &+ \int_{m \times C_{q,\varepsilon}}(K_M-K_{C_{q}})(t+is,x,x)d\mu(x)\\
                      &- \int_{m \times (C_{q}\backslash
                      C_{q,\varepsilon})}(K_{C_{q}}-K_{\mathbb{H}})(t+is,x,x)d\mu(x),
\end{align*}
where $\text{\rm ETr}K_M(t+is)$ and $\text{\rm DTr}K_M(t+is)$ denote the contribution to the trace of the non-degenerating and degenerating elliptic elements respectively.
\end{lemma}

\begin{proof}
The formal aspect of the above equality follows from the derivation of the group sum side
of the Selberg trace formula (in particular, see \cite{McK 72}). For
clarity, we give a proof below (see also \cite{DJ 98}).

\hskip 0.2in For simplicity, we will assume that $m = 1$, i.e. only one conical end degenerates. We will drop $m$ from notations and to further simplify matters, we will 
work with real time traces.   
Using the periodization of the heat kernel on $M$ allows us to write
\begin{align*}
HK_M(t,x)+EK_M(t,x) - DK_M(t,x)= K_M(t,x,x) -DK_M(t,x) - K_{\mathbb{H}}(t,0).
\end{align*}
It follows that
\begin{align*}
\text{\rm HTr}K_M(t)+\text{\rm ETr}K_M(t) -\text{\rm DTr}K_M(t) =& \int_{M\backslash 
C_{q,\varepsilon}} [K_M(t,x,x)-DK_M(t,x) -
K_{\mathbb{H}}(t,0)]d\mu (x)\\
&+\int_{C_{q,\varepsilon}} [K_M(t,x,x)-DK_M(t,x) -
K_{\mathbb{H}}(t,0)]d\mu (x).
\end{align*}
After separating the terms in the first integral of the right hand
side above, we can further write
\begin{align}\label{threeintegrals}
\text{\rm HTr}K_M(t)+\text{\rm ETr}K_M(t) - \text{\rm DTr}K_M(t)=&\int_{M\backslash 
C_{q,\varepsilon}} [K_M(t,x,x)-
K_{\mathbb{H}}(t,0)]d\mu (x) \notag\\
&-\int_{M\backslash C_{q,\varepsilon}} DK_M(t,x)d\mu (x) \\
&+\int_{C_{q,\varepsilon}} [K_M(t,x,x)-DK_M(t,x) -
K_{\mathbb{H}}(t,0)]d\mu (x).\notag
\end{align}
Denoting with $\gamma$ the only degenerating element in $DE(\Gamma)$ and using the decomposition into conjugacy classes, we can
write the degenerating contribution to the trace as follows
\begin{align*}
DK_M(t,x) &=  \sum_{n=1}^{q_{\gamma} -1}
\sum_{\kappa \in \Gamma_{\gamma} \backslash \Gamma}
K_{\mathbb{H}}(t,\tilde{x},\kappa^{-1} \gamma^n \kappa \tilde{x})\\
&=   \sum_{\kappa \in
\Gamma_{\gamma} \backslash \Gamma}  \sum_{n=1}^{q_{\gamma} -1} K_{\mathbb{H}}(t,\kappa
\tilde{x}, \gamma^n \kappa \tilde{x})\\
&= \sum_{\kappa \in \Gamma_{\gamma}
\backslash \Gamma} (K_{C_{q}}-K_{\mathbb{H}})(t,\kappa
\tilde{x}, \kappa \tilde{x})\\
&= 
(K_{C_{q}}-K_{\mathbb{H}})(t, \tilde{x},\tilde{x})+
\sum_{\substack{\kappa \in \Gamma_{\gamma} \backslash \Gamma \\
\kappa \ne id}} (K_{C_{q}}-K_{\mathbb{H}})(t,\kappa
\tilde{x}, \kappa \tilde{x}).
\end{align*}
To simplify notation we drop the tilde signs. We can now write the
third integral in the right hand side of equation
(\ref{threeintegrals}) as
\begin{align*}
\int_{ C_{q,\varepsilon}} [K_M(t,x,x)-&DK_M(t,x) -
K_{\mathbb{H}}(t,0)]d\mu (x)\\
=&\int_{ C_{q,\varepsilon}} \Bigg[K_M(t,x,x)-
K_{C_{q}}(t,x,x)+ K_{\mathbb{H}}(t,0)\\
&-\sum_{\substack{\kappa \in \Gamma_{\gamma} \backslash \Gamma \\
\kappa \ne id}} (K_{C_{q}}-K_{\mathbb{H}})(t,\kappa x,
\kappa x) - K_{\mathbb{H}}(t,0)\Bigg]d\mu (x)\\
=&\int_{ C_{q,\varepsilon}} (K_M-
K_{C_{q}})(t,x,x)d\mu(x)\\
&-\sum_{\substack{\kappa \in \Gamma_{\gamma} \backslash \Gamma \\
\kappa \ne id}}\int_{
C_{q,\varepsilon}}(K_{C_{q}}-K_{\mathbb{H}})(t,\kappa
x, \kappa x)d\mu(x).
\end{align*}
This allows us to write the equation (\ref{threeintegrals}) as
\begin{align}\label{fourintegrals}
\text{\rm HTr}K_M(t)+\text{\rm ETr}K_M(t) -\text{\rm DTr}K_M(t)=&\int_{M\backslash 
C_{q,\varepsilon}} (K_M-
K_{\mathbb{H}})(t,x,x)d\mu (x) \notag\\
&+\int_{ C_{q,\varepsilon}} (K_M -
K_{C_{q}})(t,x,x)d\mu (x) \notag\\
&-\int_{M\backslash C_{q,\varepsilon}} DK_M(t,x)d\mu (x)\\
&-\sum_{\substack{\kappa \in \Gamma_{\gamma} \backslash \Gamma \\
\kappa \ne id}}\int_{
C_{q,\varepsilon}}(K_{C_{q}}-K_{\mathbb{H}})(t,\kappa
x, \kappa x)d\mu(x).\notag
\end{align}
It remains to establish that the sum of the last two integrals in
the right hand side of (\ref{fourintegrals}) above equals to
\begin{align*}
\int_{C_{q}\backslash
C_{q,\varepsilon}}(K_{C_{q}}-K_{\mathbb{H}})(t,x,x)d\mu(x).
\end{align*}
In this direction, from the sum of the last two integrals of (\ref{fourintegrals}) we write
\begin{align*}
\int_{M\backslash  C_{q,\varepsilon}} &DK_M(t,x)d\mu (x)
+\sum_{\substack{\kappa \in \Gamma_{\gamma} \backslash \Gamma
\\ \kappa \ne id}}\int_{
C_{q,\varepsilon}}(K_{C_{q}}-K_{\mathbb{H}})(t,\kappa
x, \kappa x)d\mu(x) \\
=& \int_{M\backslash  C_{q,\varepsilon}} \Bigg[
(K_{C_{q}}-K_{\mathbb{H}})(t, x,x)+
\sum_{\substack{\kappa \in \Gamma_{\gamma} \backslash \Gamma \\
\kappa \ne id}} (K_{C_{q}}-K_{\mathbb{H}})(t,\kappa x,
\kappa x)\Bigg]d\mu(x) \\
&+\sum_{\substack{\kappa \in \Gamma_{\gamma} \backslash \Gamma \\
\kappa \ne id}}\int_{
C_{q,\varepsilon}}(K_{C_{q}}-K_{\mathbb{H}})(t,\kappa
x, \kappa x)d\mu(x).
\end{align*}
Collecting the two sums over $\kappa$ under one integral
allows us to write further
\begin{align*}
\int_{M\backslash  C_{q,\varepsilon}}
&(K_{C_{q}}-K_{\mathbb{H}})(t, x,x)d\mu(x)
+\sum_{\substack{\kappa \in \Gamma_{\gamma} \backslash \Gamma
\\ \kappa \ne id}}\int_{
M}(K_{C_{q}}-K_{\mathbb{H}})(t,\kappa
x, \kappa x)d\mu(x) \\
=&\int_{M\backslash  C_{q,\varepsilon}}
(K_{C_{q}}-K_{\mathbb{H}})(t, x,x)d\mu(x)\\
&+\sum_{\kappa \in \Gamma_{\gamma} \backslash \Gamma }\int_{
M}(K_{C_{q}}-K_{\mathbb{H}})(t,\kappa x, \kappa
x)d\mu(x) - \int_{M} (K_{C_{q}}-K_{\mathbb{H}})(t,
x,x)d\mu(x)\\
=& \sum_{\kappa \in \Gamma_{\gamma} \backslash \Gamma }\int_{ \Gamma
\backslash
\mathbb{H}}(K_{C_{q}}-K_{\mathbb{H}})(t,\kappa x,
\kappa x)d\mu(x) -\int_{C_{q,\varepsilon}}
(K_{C_{q}}-K_{\mathbb{H}})(t, x,x)d\mu(x)\\
=&\int_{C_{q}\backslash C_{q,\varepsilon}}
(K_{C_{q}}-K_{\mathbb{H}})(t, x,x)d\mu(x)
\end{align*}
since the sum of the integral in the fourth equality above unfolds
to $\Gamma_{\gamma}\backslash \mathbb{H}$ which represents the
infinite cone $C_{q}$. This proves the case for one degenerating elliptic element and
easily generalizes to several degenerating elements. This completes the proof of the
formal aspect of the lemma.

\hskip 0.2in We need to show finiteness for all $t>0$. The proof is
similar to Theorem 1.1 of \cite{JoLu 97b}. In particular this will
imply that the right hand side of the equation in the statement of
the lemma is independent of the choice of $\varepsilon$. That the
first and third integrals are bounded follows from 
Lemma \ref{constantboundlemma} and Proposition \ref{thmintegraloverdifferencecusps} respectively.

\hskip 0.2in 
It remains to establish the finiteness of the second
integral. We will prove this by two applications of the maximum
principle (\cite{Ch 84} page 180). Consider the
following function
\begin{equation*}
D(t,x,y) = K_M(t,x,y) - K_{C_{q}}(t,x,y).
\end{equation*}
Fix $y$  in some conical neighborhood $C_{q,\varepsilon}$ of the surface $M$.
Observe that $D(t,x,y)$ satisfies the heat equation with respect to
the $x$ and $t$.  Fix some $\varepsilon_0 > \varepsilon$, so that
all conical points have hyperbolic neighborhoods of area equal to
$\varepsilon_0$. Using the maximum principle, the function
$D(t,x,y)$ attains its maximum when $x$ lies on the boundary of the
conical neighborhood $C_{q, \varepsilon_0}$. Note that $y$ is contained $C_{q, \varepsilon}$ which lies inside the enveloping conical neighboorhood
$C_{q, \varepsilon_0}$. Using the maximum principle along with
the positivity of the heat kernels, we get the bound
\begin{align}\label{maxprinciple1}
- \sup_{\substack{z\in \partial C_{q,\varepsilon_0}\\
0 \le \tau \le t}} K_{C_{q}}(\tau,z,y) \le D(t,x,y) \le
\sup_{\substack{z\in \partial C_{q,\varepsilon_0}\\
0 \le \tau \le t}} K_M(\tau,z,y).
\end{align}
For each $z$, the terms in  (\ref{maxprinciple1})
satisfy the heat equation on $C_{q,\varepsilon_0/2}$ with zero
initial data. A second application of the maximum principle gives
the bound
\begin{align}\label{maxprinciple2}
- \sup_{\substack{z\in \partial C_{q,\varepsilon_0}\\ w\in
\partial C_{q,\varepsilon_0/2}\\0 \le \tau \le t}}
K_{C_{q}}(\tau,z,w) \le D(t,x,y) \le \sup_{\substack{z\in
\partial C_{q,\varepsilon_0}\\ w\in \partial C_{q,\varepsilon_0/2}
\\0 \le \tau \le t}}
K_M(\tau,z,w).
\end{align}
Standard bounds for the heat kernel (\cite{Ch 84} page
198) and equation (\ref{maxprinciple2}) provide upper and lower
bounds for the function $D(t,x,y)$. Thus the second integral in the statement of the lemma can be made
arbitrarily small since both the integrand as well as the domain of
integration can be made arbitrarily small.

\hskip 0.2in We remark here that the lower bounds in
(\ref{maxprinciple1}) and (\ref{maxprinciple2}) can be improved
trivially to zero combining (\ref{periodization}) with the
observation that the fundamental group of $C_{q}$ embeds into
the fundamental group of $M$. 
\end{proof}

\hskip 0.2in The following theorem is the principal result of this section as well as one of
the main tools used in this paper. For instance, this type
of regularized convergence will be used to show the convergence in
the context of elliptic degeneration of the spectral
weighted counting functions, the Selberg zeta function, the spectral
zeta function, as well as other functions such as the Poisson
kernel, the wave kernel, and the resolvent kernel. With these
remarks in mind, we state the following theorem.

\begin{thm}\label{pointwiseanduniformconvergence}
Let $M_{q}$ denote an elliptically degenerating family of
compact or non-compact hyperbolic Riemann surfaces of finite volume
converging to the non-compact hyperbolic surface $M_{\infty}.$ 
\begin{enumerate}[(a)]
\item (Pointwise) For fixed $z=t+is$ with $t>0$, we have
\begin{align*}
\lim_{q \to
\infty}[\text{\rm HTr}K_{M_{q}}(z)+\text{\rm ETr}K_{M_{q}}(z)-\text{\rm DTr}K_{M_{q}}(z)]
=\text{\rm HTr}K_{M_{\infty}}(z)+\text{\rm ETr}K_{M_{\infty}}(z).
\end{align*}
\item (Uniformity) For any $t>0$, there exists a constant $C$ such
that for all $s \in \mathbb{R}$ and all $q$, we have the bound
\begin{align*}
|\text{\rm HTr}K_{M_{q}}(z)+\text{\rm ETr}K_{M_{q}}(z)-\text{\rm DTr}K_{M_{q}}(z)| \le
C(1+|s|)^{3/2}.
\end{align*}
\end{enumerate}
\end{thm}

\begin{proof}
We will prove part (a) in the case $M_{q}$ is an elliptically
degenerating compact family. Then we will extend the result in the
non-compact setting.

\hskip 0.2in If $M_{q}$ is a family of elliptically degenerated compact
hyperbolic Riemann surfaces, the above Lemma \ref{threeintegralslemma} allows us to write for sufficiently small
$\varepsilon$ and $t>0$
\begin{align}
\tag{I}(\text{\rm HTr}K_{M_{q}} + \text{\rm ETr}K_{M_{q}} &- \text{\rm DTr}K_{M_{q}})(t+is) = \int_{M_{q}\backslash C_{q, \varepsilon}} (K_{M_{q}}-K_{\mathbb{H}})(t+is,x,x)d\mu(x)\\
\tag{II}&+ \int_{C_{q, \varepsilon}} (K_{M_{q}}-K_{C_{q}})(t+is,x,x)d\mu(x)\\
\tag{III}&- \int_{C_{q}\backslash C_{q, \varepsilon}}
(K_{C_{q}}-K_{\mathbb{H}})(t+is,x,x)d\mu(x).
\end{align}

\hskip 0.2in As $q$ goes to infinity, Theorem \ref{convergenceheatkernels}
part (a) implies that the integrand in (I) converges uniformly on
compact subsets of $M_{q}$ bounded away from the developing
cusps and the cones corresponding to elliptic elements. On such
compact sets, the metric converges uniformly. Since the domain of
integration is compact, we get that
\begin{align*}
\lim_{q \to \infty} (\textrm{I}) = \int_{M_{\infty}\backslash
C_{\infty,\varepsilon}}
(K_{M_{\infty}}-K_{\mathbb{H}})(z,x,x)d\mu(x).
\end{align*}

\hskip 0.2in Integral (II) corresponds to the integral over
$C_{\infty, \varepsilon}$. Using Corollary
\ref{uniformboundscorollary} part (d) for a small choice of
$\varepsilon$, we can make integral (II) arbitrarily small.

\hskip 0.2in It remains to consider integral (III). In this
direction, choose $\delta >\varepsilon$ and break down integral
(III) in two, namely
\begin{align}\label{III breakup}
\notag\int_{C_{q}\backslash C_{q,\varepsilon}}
[K_{C_{q}}-K_{\mathbb{H}}](z,x,x)d\mu(x) =&
\int_{C_{q,\delta}\backslash C_{q,\varepsilon}}
[K_{C_{q}}-K_{\mathbb{H}}](z,x,x)d\mu(x)\\
&+\int_{C_{q}\backslash C_{q,\delta}}
[K_{C_{q}}-K_{\mathbb{H}}](z,x,x)d\mu(x).
\end{align}
From Theorem \ref{convergenceheatkernels} part (a), we have that the
integrand of (III) converges uniformly in $q$. Since the domain
$C_{q,\delta}\backslash C_{q,\varepsilon}$ is compact, we
have that
\begin{align*}
\lim_{q \to \infty}\int_{C_{q}\backslash
C_{q,\varepsilon}}
[K_{C_{q}}-K_{\mathbb{H}}](z,x,x)d\mu(x) =
\int_{C_{\infty,\delta}\backslash C_{\infty,\varepsilon}}
[K_{C_{\infty}}-K_{\mathbb{H}}](z,x,x)d\mu(x),
\end{align*}
where the second integral in (\ref{III breakup}) can be made
arbitrarily small by choosing $\delta$ small enough as in part (b) of
Corollary \ref{uniformboundscorollary}. This completes the proof of
Theorem \ref{pointwiseanduniformconvergence} part (a) in the compact
case.

\hskip 0.2in 
The non-compact version of  Theorem \ref{pointwiseanduniformconvergence} part (a) can be argued as follows. Suppose that $M_{q} = M_{q,\infty}$ has $p$ cusps. Using Proposition \ref{Judge}, the surface $M_{q,\infty}$ can be realized as the limit of an elliptically degenerating family of compact hyperbolic surfaces $M_{q,p}$ having $m+p$ conical points, by degenerating $p$ cones into cusps. The compact case of Theorem \ref{pointwiseanduniformconvergence} part (a) applies: For every fixed $q$, we have
\begin{align*}
\lim_{p\to \infty} \mathrm{Htr}K_{M_{q,p}}(z) + \mathrm{Etr}K_{M_{q,p}}(z) - \mathrm{Dtr}K_{M_{q,p}}(z) = 
\mathrm{Htr}K_{M_{q,\infty}}(z) + \mathrm{Etr}K_{M_{q,\infty}}(z)\,\,.
\end{align*}
The compact case of Theorem \ref{pointwiseanduniformconvergence} also applies to the family $M_{q,p}$ when we let $q$ and $p$ go simultaneously to infinity:
\begin{align*}
\lim_{q,p\to \infty} \mathrm{Htr}K_{M_{q,p}}(z) + \mathrm{Etr}K_{M_{q,p}}(z) - \mathrm{Dtr}K_{M_{q,p}}(z) = 
\mathrm{Htr}K_{M_{\infty,\infty}}(z) + \mathrm{Etr}K_{M_{\infty,\infty}}(z)\,\,.
\end{align*}
Having showed that for each $q$ a limit as $p \to \infty$ exists and that a limit exists when $q,p \to \infty$ simultaneously, we conclude that 
\begin{align*}
\lim_{q\to \infty} \mathrm{Htr}K_{M_{q,\infty}}(z) + \mathrm{Etr}K_{M_{q,\infty}}(z) - \mathrm{Dtr}K_{M_{q,\infty}}(z) = 
\mathrm{Htr}K_{M_{\infty,\infty}}(z) + \mathrm{Etr}K_{M_{\infty,\infty}}(z)\,\,.
\end{align*}

\hskip 0.2in We will prove part (b) of Theorem
\ref{pointwiseanduniformconvergence}, by first treating the compact
case and then extend the proof to the non-compact case. From Lemma
\ref{constantboundlemma} we have that integral (I) is $O(1)$,
independent of $q$. Integrals (II) and (III) are both
$O(s^{3/2})$ independent of $q$ according to Corollary
\ref{uniformboundscorollary} part (c) and part (a) respectively. This
proves the compact case.

\hskip 0.2in To prove the non-compact case, we note that any
non-compact hyperbolic surface $M_{q}$ can be realized as the
limit of a compact family $M_{q,q'}$. The compact case of
Theorem \ref{pointwiseanduniformconvergence} part (b) applies and we
can write
\begin{align*}
|\text{\rm HTr}K_{M_{q,q'}}(z)+\text{\rm ETr}K_{M_{q,q'}}(z)-\text{\rm DTr}K_{M_{q,q'}}(z)|
\le C(1+|s|)^{3/2}
\end{align*}
where the bound is independent of $q$ and $q'$ and the
constant $C$ is independent of the limiting surface. By letting
$q'$ go to infinity and applying Theorem
\ref{pointwiseanduniformconvergence} part (a), we can write
\begin{align*}
|\text{\rm HTr}K_{M_{q,\infty}}(z)+\text{\rm ETr}K_{M_{q,\infty}}(z)-\text{\rm DTr}K_{M_{q,\infty}}(z)|
\le C(1+|s|)^{3/2}
\end{align*}
uniformly in $q$. This completes the proof.
\end{proof}

\begin{rmk}\label{uniformboundrmk}
\emph{
By following the steps of Theorem \ref{pointwiseanduniformconvergence} with slight modifications, we can derive the following bound for the difference of traces. For $0< t <1$, there is a positive constant $C$ such that 
\begin{align*}
|\text{\rm HTr}K_{M_{q}}(z)+\text{\rm ETr}K_{M_{q}}(z)-\text{\rm DTr}K_{M_{q}}(z)|
\le Ct^{-2}(1+|s|)^{3/2}\,\,,
\end{align*}
holds. To do so, we need to revisit integrals (I) through (III) from Theorem \ref{pointwiseanduniformconvergence}. 
For integral (I), Lemma \ref{constantboundlemma} gives the upper bound $Ct^{-3/2}$. Looking back to formulas (\ref{boundintegralA}) through (\ref{boundintegralC}) in the proof of Corollary \ref{uniformboundscorollary} part (c), we see that $Ct(1+|s|)^{3/2}$ provides an upper bound for integral (II). 
For integral (III), we start by splitting it as  
\begin{align*}
\tag{III.1}\int_{C_{q}\backslash C_{q, \varepsilon}}
(K_{C_{q}}-K_{\mathbb{H}})(t,x,x)d\mu(x)
=&\int_{C_{q}\backslash C_{q, \varepsilon_1}}
(K_{C_{q}}-K_{\mathbb{H}})(t,x,x)d\mu(x)\\
\tag{III.2}&+\int_{C_{q,\varepsilon_1}\backslash C_{q, \varepsilon}}
(K_{C_{q}}-K_{\mathbb{H}})(t,x,x)d\mu(x).
\end{align*} 
with $\varepsilon_1 > \textrm{max} \{ 2\pi, \varepsilon\}$. 
We recall Proposition \ref{thmintegraloverdifferencecusps}, so that for integral (III.1) we obtain the bound
\begin{align*}
|\textrm{III.1}| \le \frac{1}{\sqrt{|z|}}\bigg[\zeta (1+2\eta \gamma)+ \pi\bigg],
\end{align*}
with $\eta = t/(4(t^2+s^2))$ and $\gamma = \log(1 + (\varepsilon_1/2\pi)^2)$. If $s \ne 0$, then $2\eta \gamma \searrow 0 $ as $t \searrow 0$; consequently  
$\zeta(1+2\eta\gamma) \sim (2\eta\gamma)^{-1}$ and 
\begin{align*}
|\textrm{III.1}| &\le \frac{1}{\sqrt{t^2+s^2}}\bigg[ \frac{2(t^2+s^2)}{\gamma t} + c_{\gamma}\bigg]
= (t^2+s^2)^{3/4}\bigg[ \frac{c_{\gamma}}{t^2} + \frac{2}{\gamma t}\bigg] 
\le C_{\gamma} t^{-2}(1+|s|)^{3/2}.
\end{align*}
For integral (III.2) we use the Corollary \ref{uniformboundscorollary} part (c) and the inclusion of heat kernels as follows:
\begin{align*}
|\textrm{III.2}| \le \int_{C_{q,\varepsilon_1}}
(K_{M_{q}}-K_{\mathbb{H}})(t,x,x)d\mu(x) \le Ct(1+|s|)^{3/2}.
\end{align*}
}
\end{rmk}

\hskip 0.2in As a consequence of Theorem \ref{pointwiseanduniformconvergence} part (a), we have the following corollary, which described the small time asymptotic behavior for the regularized trace of the heat kernel.

\begin{cor}\label{smalltasymptotics}
Let $M_{q}$ denote an elliptically degenerated family of
compact or non-compact hyperbolic Riemann surfaces of finite volume
which converges to the non-compact hyperbolic surface $M_{\infty}$.
Then for any fixed $\delta>0$, there exists a positive constant $c$
such that for all $0<t<\delta$, we have
\begin{align*}
\text{\rm HTr}K_{M_{q}}(t)+ \text{\rm ETr}K_{M_{q}}(t) - \text{\rm DTr}K_{M_{q}}(t) =
O\left(t^{-3/2}\right)\,\,
\end{align*}
uniformly in $q$.
\end{cor}

\begin{proof}
As per Remark \ref{uniformboundrmk} above, integral (I) is $O(t^{-3/2})$ while integrals (II) and (III.2) are $O(1).$
The only change is with integral (III.1) for which the special case $s = 0$ allows to improve its upper bound. Namely,
$2\eta \gamma = \gamma /(2t ) \to \infty $ as $t \searrow 0$; consequently  
$\zeta(1+2\eta\gamma) \sim 1 + (2\eta\gamma)^{-1}$ and 
\begin{align*}
|\textrm{III.1}| &\le \frac{1}{\sqrt{t}}\bigg[ \frac{2t}{\gamma} + c_{\gamma}\bigg]
\le C_{\gamma} t^{-1/2}.
\end{align*}
\end{proof}

\begin{rmk}
\emph{
In the non-compact setting, aside from the $m$ degenerating conical points, each surface in the family has $p$ cusps. Consequently, Lemma \ref{threeintegralslemma} becomes a 5 integral lemma
\begin{align*}
(\text{\rm HTr}K_M+\text{\rm ETr}K_M - \text{\rm DTr}K_M)(t)= &  \int_{M\backslash (m \times C_{q,\varepsilon} \cup\,\, p \times C_{\infty,\varepsilon} )  } (K_M-K_{\mathbb{H}})(t,x,x)d\mu(x)\\
&+ \int_{m \times C_{q,\varepsilon}}(K_M-K_{C_{q}})(t,x,x)d\mu(x)\\
&- \int_{m \times (C_{q}\backslash C_{q,\varepsilon})}(K_{C_{q}}-K_{\mathbb{H}})(t,x,x)d\mu(x)\\
&+ \int_{p \times C_{\infty,\varepsilon}}(K_M-K_{C_{\infty}})(t,x,x)d\mu(x)\\
&- \int_{p \times (C_{\infty}\backslash C_{\infty,\varepsilon})}(K_{C_{\infty}}-K_{\mathbb{H}})(t,x,x)d\mu(x)\,\,.
\end{align*}
Integrals (IV) and (V) above are the counterparts of integrals (II) and (III) respectively. The boundedness of integral (IV) follows using similar arguments as in the case of in integral (II). Unlike integral (III),  integral (V) turns out to be bounded. Namely, we split integral (V) as
\begin{align*}
\int_{C_{\infty}\backslash C_{\infty, \varepsilon}}
(K_{C_{\infty}}-K_{\mathbb{H}})(t,x,x)d\mu(x)
=&\int_{C_{\infty}\backslash C_{\infty, \varepsilon_1}}
(K_{C_{\infty}}-K_{\mathbb{H}})(t,x,x)d\mu(x)\\
&+\int_{C_{\infty,\varepsilon_1}\backslash C_{\infty, \varepsilon}}
(K_{C_{\infty}}-K_{\mathbb{H}})(t,x,x)d\mu(x)
\end{align*}
with $\varepsilon_1 > \textrm{max} \{2\sqrt{2},\varepsilon\}.$ The second integral on the right-hand side above is over compact domain. The boundedness of the first integral follows from Theorem 3.1 of \cite{JoLu 97a}. 
}
\end{rmk}

\hskip 0.20in Theorem \ref{pointwiseanduniformconvergence} part (a) shows that the
hyperbolic heat trace plus the elliptic heat trace minus the
degenerating heat trace converges pointwise to the hyperbolic heat
trace plus the elliptic heat trace on the limiting surface. Through
elliptic degeneration, the angles that parametrize the degenerating
cones become arbitrarily small as these cones turn into cusps. Their
contribution to the volume of $M_{q}$ becomes arbitrarily
small. This means the the volume of $M_{q}$ converges to the
volume of the limiting surface $M_{\infty}$. It follows that the
regularized heat trace minus the degenerating heat trace on
$M_{q}$ converges pointwise to the regularized heat trace on
$M_{\infty}$. The corollary to Theorem \ref{pointwiseanduniformconvergence}
above shows uniformity of convergence near zero. We need to consider
the asymptotics for all positive $t$.

\section{Convergence of small eigenvalues and eigenfunctions}
\label{Convergence of small eigenfunctions}

\hskip 0.2in In this section we will show the convergence of small eigenvalues and corresponding
eigenfunctions through elliptic degeneration. This result is needed
in the proof of Theorem \ref{uniformlongtimeasymptotics}.  
As stated in the introduction, we will systematically develop in \cite{GJ 16}
applications of the results in the present article to determine asymptotic
behavior of various spectral functions through elliptic degeneration.  

\hskip 0.2in The following theorem is the main result of this
section.
\begin{thm}\label{specconvinverselaplace}
Let $f$ be any measurable function on $\mathbb{R}_{+}$ such that
there exist a vertical line $t>0$ where its Laplace transform
$\mathscr{L}(f)(t+is)$ is $L^1$ as a function of $s$. Let
$M_{q}$ be an elliptically degenerating family of hyperbolic
Riemann surfaces of finite volume which converge to the limiting
surface $M_{\infty}$. Let $x$ and $y$ be any two points which remain
bounded away from the developing cusps. For $z=t+is$ with $t>0$ and any $T>0$, we have
the limit
\begin{align*}
\lim_{q\to \infty} \frac{1}{2\pi i}
\int_{t-i\infty}^{t+i\infty}
K_{M_{q}}(z,x,y)\mathscr{L}(f)(z)e^{Tz} \frac{dz}{z} =
\frac{1}{2\pi i} \int_{t-i\infty}^{t+i\infty}
K_{M_{\infty}}(z,x,y)\mathscr{L}(f)(z)e^{Tz} \frac{dz}{z}.
\end{align*}
The convergence is uniform on compact subsets of $M_{\infty}\times
M_{\infty}$.
\end{thm}

\begin{proof}
For $t>0$ and any points $x$ and $y$ which are bounded away from the developing cusps, the heat kernel converges uniformly. The convergence of the heat kernel
(see Theorem \ref{convergenceheatkernels}) together with the $L^1$ assumption for the Laplace
transform of $f$ fulfill the hypotheses of the dominated convergence theorem which we apply to conclude
the proof.
\end{proof}

\begin{rmk}
\emph{For our purposes here, the function $f$ described in Theorem
\ref{specconvinverselaplace} can be any function that is
continuously differentiable on $\mathbb{R}_{+}$ which vanishes at
$t=0$ and whose first derivative is of bounded variation. }
\end{rmk}

\begin{lemma}\label{specconvlemma}
Let $M$ be a fixed hyperbolic Riemann surface of finite volume and
$f$ be a function as described in the previous remark.
\begin{enumerate}[(a)]
\item If $M$ is compact, then
\begin{align*}
\frac{1}{2\pi i} \int_{t-i\infty}^{t+i\infty}
K_{M}(z,x,y)\mathscr{L}(f)(z)e^{Tz} \frac{dz}{z} =
\sum_{\lambda_{M,n}<T} f(T-\lambda_{M,n})\phi_{M,n}(x)\phi_{M,n}(y).
\end{align*}
\item If $M$ is not compact, then
\begin{align*}
\frac{1}{2\pi i} \int_{t-i\infty}^{t+i\infty}
K_{M}(z,x,y)\mathscr{L}(f)(z)e^{Tz} \frac{dz}{z} &=
\sum_{\lambda_{M,n}<T}
f(T-\lambda_{M,n})\phi_{M,n}(x)\phi_{M,n}(y)\\
+ \frac{1}{2\pi} \sum_{P}\int_{0}^{\sqrt{T-1/4}} f(T-1/4-r^2)&
E_{\mathrm{par};M,P}(1/2+ir,x)\overline{E_{\mathrm{par};M,P}(1/2+ir,y)}dr,
\end{align*}
where the last integral is zero if $T \le 1/4$.
\end{enumerate}
\end{lemma}
\begin{proof}
The lemma follows from the spectral decomposition of the heat kernel (see (\ref{heatkernelexpansioncompact}) and (\ref{heatkernelexpansionnoncompact}))
and the basic properties of the inverse Laplace transform (see
\cite{Wi 41} pages 73 and 91).
\end{proof}

\begin{cor}
Let $M_{q}$ be a degenerating family of non-compact hyperbolic
Riemann surfaces with limiting surface $M_{\infty}$. Let $x$ and $y$
be two points bounded away from the developing cusps. Then, for
every fixed $T>0$, we have
\begin{align*}
\lim_{q\to \infty}\sum_{\lambda_{M_{q},n}<T}
(T-\lambda_{M_{q},n})\phi_{M_{q},n}(x)\phi_{M_{q},n}(y)
=&  \sum_{\lambda_{M_{\infty},n}<T}
(T-\lambda_{M_{\infty},n})\phi_{M_{\infty},n}(x)\phi_{M_{\infty},n}(y)\\
+\frac{1}{2\pi} \sum_{P}\int_{0}^{\sqrt{T-1/4}} (T-1/4-r^2)
E_{\mathrm{par};M_{\infty},P}&(1/2+ir,x)\overline{E_{\mathrm{par};M_{\infty},P}(1/2+ir,y)}dr
\end{align*}
where the integral is zero if $T<1/4$.
\end{cor}
\begin{proof}
This follows directly from Lemma \ref{specconvlemma} with $f(t)=t$.
\end{proof}

\hskip 0.2in We would like to apply Theorem
\ref{specconvinverselaplace} with $f(t)=1$. However,
$\mathscr{L}(f)(z) = 1/z$ is not $L^1$, so we cannot apply Theorem
\ref{specconvinverselaplace} directly. In order to make the
statement, we need to set some notation. So let $w > 0$ and define the following cumulative distribution for a
hyperbolic Riemann surface $M$
\begin{align*}
C_{M,w}(x,T) = \frac{1}{2\pi i} \int_{t-i\infty}^{t+i \infty}
K_M(z,x,x)e^{Tz}\frac{dz}{z^{w+1}}.
\end{align*}
Note that, using part (a) of Proposition \ref{kerneluniformbounds}, we easily obtain that the above integral converges whenever $w>0$.
Furthermore, if $M$ is compact, then
\begin{align*}
C_{M,w}(x,T) = \sum_{\lambda_{M,n}\le
T}(T-\lambda_{M,n})^{w}\phi_{M,n}(x)^2,
\end{align*}
and if $M$ is non-compact
\begin{align*}
C_{M,w}(x,T) =& \sum_{\lambda_{M,n}\le
T}(T-\lambda_{M,n})^{w}\phi_{M,n}(x)^2 \\
&+\frac{1}{2\pi} \sum_{P}\int_{0}^{\sqrt{T-1/4}} (T-1/4-r^2)^w
E_{\mathrm{par},M,P}(1/2+ir,x)^2dr,
\end{align*}
where the integral is zero if $T<1/4$.

\begin{thm}\label{specconvcountingdistribution}
Let $M_{q}$ be a degenerating family of hyperbolic Riemann
surfaces with limiting surface $M_{\infty}$. Let $x$ be a point
which is bounded away from the developing cusps. Then, if $0\le T <1/4$ is not
an eigenvalue of $M_{\infty}$, we have
\begin{align*}
\lim_{q \to \infty} C_{M_{q},0}(x,T) =
C_{M_{\infty},0}(x,T).
\end{align*}
\end{thm}
\begin{proof}
From Theorem \ref{specconvinverselaplace}, we have that for weight $w=1$, the following limit holds
\begin{align*}
\lim_{q \to \infty} C_{M_{q},1}(x,T) =
C_{M_{\infty},1}(x,T)\,\,.
\end{align*}
To show the above limit holds for weight $w=0$, we use the mean value theorem as in \cite{JoLu 97a}.
\end{proof}

\hskip 0.2in We end this section with the following corollary which
states the convergence of the small eigenfunctions.
\begin{cor}\label{smalleigenfunctions}
Let $M_{q}$ be a family of elliptically degenerating compact
hyperbolic Riemann surfaces which converges to $M_{\infty}$. Suppose
that $0\le T < 1/4 $ is not an eigenvalue of the limiting surface.
Then, for any point $x$ which is bounded away from the developing
cusps, we have
\begin{align*}
\lim_{q \to \infty} \sum_{\lambda_{M_{q},n}\le T}
\phi_{M_{q},n}(x)^2 = \sum_{\lambda_{M_{\infty},n}\le T}
\phi_{M_{\infty},n}(x)^2.
\end{align*}
In particular, if the eigenspace associated to the eigenvalue
$\lambda_{M_{\infty},n}$ is one-dimensional, then
\begin{align*}
\lim_{q \to \infty} \phi_{M_{q},n}(x) =
\phi_{M_{\infty},n}(x).
\end{align*}
The convergence is uniform on compact subsets of $M_{\infty}$.
\end{cor}
\begin{proof}
This follows directly from Theorem \ref{specconvcountingdistribution}. 
\end{proof}

\section{Uniform long time asymptotics}
\label{Uniform long time asymptotics}

\hskip 0.2in One of the results of the previous section, Corollary \ref{smalltasymptotics}, presents the behavior of the regularized trace for small values of the time variable. In this section, we continue the investigation of the regularized trace for large values of the time variable. For all surfaces under our consideration, the spectrum of the
Laplacian is discrete below 1/4. For such a hyperbolic surface $M$, we
denote the eigenvalues in this range by $\{\lambda_{M,n}\}$ and the
corresponding normalized eigenfunction by $\{\phi_{M,n}\}$. The main
result of this section is Theorem \ref{uniformlongtimeasymptotics}. Before we get to it, we need one definition and some ancillary lemmas.

\begin{defn}\label{alphatruncatedtrace}
Let $M_{q}$ be an elliptically degenerating family of compact
or non-compact hyperbolic Riemann surfaces of finite volume which
converge to the non-compact hyperbolic surface $M_{\infty}$. Let
$0\le \alpha  < 1/4$ be such that $\alpha$ is not an eigenvalue of
$M_{\infty}$. We defined the $\alpha$-truncated hyperbolic and elliptic trace by
\begin{align*}
\text{\rm HTr}K_{M_{q}}^{(\alpha)}(t) + \text{\rm ETr}K_{M_{q}}^{(\alpha)}(t) =
\text{\rm HTr}K_{M_{q}}(t) + \text{\rm ETr}K_{M_{q}}(t) -
\sum_{\lambda_{q,n}\le \alpha} e^{-\lambda_{q,n}t}.
\end{align*}
\end{defn}

\hskip 0.2in As in the course of the proof for the behavior of the trace for small values of the time parameter, the investigation of the long time asymptotic of the regularized trace is based on analyzing the three  three integrals as in Theorem \ref{pointwiseanduniformconvergence}. For integrals
(I) and (III) over $M_{q}$ and $C_{q}$ away from
developing cusps, we need the following lemma.

\begin{lemma}\label{truncatedkernellimitlemma}
Let $R_{q}$ denote either $M_{q}$ or $C_{q}$, that
is, either a degenerating hyperbolic surface of finite volume or a
degenerating hyperbolic cone of infinite volume. For $\alpha < 1/4$
and $c < \alpha$, the limit
\begin{align*}
\lim_{q \to \infty} e^{ct} K_{R_{q}}^{(\alpha)}(t,x,x)
= e^{ct} K_{R_{\infty}}^{(\alpha)}(t,x,x)
\end{align*}
is uniform for $x \in R_{q}\backslash C_{q,\varepsilon}$
and $t > 0$.
\end{lemma}
\begin{proof}
To prove this lemma, we first realize the heat kernel on $M_{q}$ as a Stieltjes integral of the periodized heat kernel in the upper-plane against geodesic counting functions. By using bounds on the heat kernel in the upper-half plane and on the counting functions together with the uniform convergence of the hyperbolic metrics away from the developing cusps as well as the convergence of the small eigenvalues, the result then follows. For more details, see the proof of Lemma 3.2 of \cite{JoLu 97b}.
\end{proof}

\begin{lemma}\label{truncatedkernelsupbound}
Let $M_{q}$ denote an elliptically degenerating family of
compact or non-compact hyperbolic Riemann surfaces of finite volume
which converge to the non-compact hyperbolic surface $M_{\infty}$.
Let $c < \alpha < 1/4$ and $\varepsilon < 1/2$. Then there is a
constant $C$ such that for all $t > 0$, we have
\begin{align*}
\sup_{\substack{ q \\ x \in \partial C_{q,\varepsilon}
\\ y \in \partial C_{q,\varepsilon} }}
|K_{M_{q}}^{(\alpha)}(t,x,y) -
K_{C_{q}}^{(\alpha)}(t,x,y)| \le C e^{-ct}.
\end{align*}
\begin{proof}
Using two applications of the maximum principle (as in \ref{maxprinciple1} and \ref{maxprinciple2}) gives us the following bound
\begin{align}\label{maxprinciple3}
\sup_{\substack{  x \in \partial C_{q,\varepsilon}
\\ y \in \partial C_{q,\varepsilon} }}
|K_{M_{q}}(t,x,y) -
K_{C_{q}}(t,x,y)| \le C .
\end{align}
\hskip 0.2in We also remark that 1/4 is a lower bound for the bottom of the spectrum for both the infinite volume cylinder $C_{\infty}$ and the infinite volume cone $C_{q}$. Consequently, since $\alpha < 1/4$, it follows that $K_{C_{\infty}}^{(\alpha)}(t,x,y) = K_{C_{\infty}}(t,x,y)$ and $K_{C_{q}}^{(\alpha)}(t,x,y) = K_{C_{q}}(t,x,y)$. The triangle inequality allows us to write
\begin{align*}
e^{ct}|K_{M_{q}}^{(\alpha)}(t,x,y) -
K_{C_{q}}^{(\alpha)}(t,x,y)| \le &
\,\, |e^{ct}K_{M_{q}}^{(\alpha)}(t,x,y) -
e^{ct}K_{M_{\infty}}^{(\alpha)}(t,x,y)|\\
&+|e^{ct}K_{C_{q}}^{(\alpha)}(t,x,y) -
e^{ct}K_{C_{\infty}}^{(\alpha)}(t,x,y)|\\
&+|e^{ct}K_{M_{\infty}}^{(\alpha)}(t,x,y) -
e^{ct}K_{C_{\infty}}^{(\alpha)}(t,x,y)|\\
\le &\,\, \varepsilon + |e^{ct}K_{M_{\infty}}^{(\alpha)}(t,x,y) -
e^{ct}K_{C_{\infty}}(t,x,y)|
\end{align*}
where the last inequality follows from Lemma \ref{truncatedkernellimitlemma} above.
Another application of the triangle inequality gives the further bound
\begin{align*}
\varepsilon + e^{ct}|K_{M_{\infty}}(t,x,y) -
K_{C_{\infty}}(t,x,y)| + e^{ct} \sum_{\lambda_n < \alpha} e^{-\lambda_n t},
\end{align*}
where the collection $\{\lambda_n\}$ is the finite set of eigenvalues of $M_{\infty}$ in the range $[0,\alpha)$
Now apply the sup over $x,y \in \partial C_{q,\varepsilon}$ and use the bound as in (\ref{maxprinciple3}) to complete the proof.
\end{proof}
\end{lemma}

\begin{lemma}\label{fl2norminequality}
Let $f(t,x)$ be a solution to the Dirichlet heat problem on the
finite cone $C_{q,\varepsilon}$. For fixed $t >0$ let
$\parallel f(t, \cdot)\parallel_{C_{q,\varepsilon},2}$ denote
the $L^2$-norm of $f(t,\cdot)$ as a function on
$C_{q,\varepsilon}$. Then for all $t_0,t > 0$, we have
\begin{align*}
\parallel f(t_0+t, \cdot)\parallel_{2;C_{q,\varepsilon}} \le
\parallel f(t_0, \cdot)\parallel_{2;C_{q,\varepsilon}}
e^{-t/4}.
\end{align*}
\end{lemma}
\begin{proof}
From the definitions, we have that
\begin{align*}
\partial_t \parallel f(t_0+t,
\cdot)\parallel_{2;C_{q,\varepsilon}}^2
=\int_{C_{q,\varepsilon}} 2ff_t d\mu =\int_{C_{q,\varepsilon}}
2f\Delta f d\mu = -2\int_{C_{q,\varepsilon}} \langle \text{grad} f, \text{grad}f\rangle d\mu
\end{align*}
where the last inequality follows from the Green's theorem as applied to functions that vanish on the boundary of the domain of integration.
Therefore,
\begin{align*}
\partial_t \parallel f(t_0+t,
\cdot)\parallel_{2;C_{q,\varepsilon}}^2
&=\frac{-2\int_{C_{q,\varepsilon}} \langle\text{grad} f, \text{grad}f\rangle d\mu}{\parallel
f(t_0+t, \cdot)\parallel_{2;C_{q,\varepsilon}}^2} \parallel
f(t_0+t, \cdot)\parallel_{2;C_{q,\varepsilon}}^2\\
& = -2 \lambda_{f} \parallel
f(t_0+t, \cdot)\parallel_{2;C_{q,\varepsilon}}^2\\
&\le -\frac{1}{2} \parallel f(t_0+t,
\cdot)\parallel_{2;C_{q,\varepsilon}}^2.
\end{align*}
The last inequality follows from the fact that $\lambda=1/4$ is a
lower bound for the bottom of the spectrum for $C_{q}$. The
result follows by integration.
\end{proof}

\begin{lemma}\label{poissonkernellongtimebound}
For any $\varepsilon < \delta$, there exists a constant $C$ such
that
\begin{align*}
\parallel P_{q,
\delta}(t,\zeta,\cdot)\parallel_{C_{q,\varepsilon},2} \le C
e^{-t/4}.
\end{align*}
\end{lemma}
\begin{proof}
Choose  $t_0$ small enough such that $0<t_0<t$ and set $t' = t-t_0$. Then we can write
\begin{align*}
\parallel P_{q,
\delta}(t,\zeta,\cdot)\parallel_{2;C_{q,\varepsilon}}
\le& \parallel P_{q,
\delta}(t,\zeta,\cdot)\parallel_{2;C_{q,\delta}} =
\parallel P_{q,
\delta}(t'+t_0,\zeta,\cdot)\parallel_{2;C_{q,\delta}}\\
\le& \parallel P_{q,
\delta}(t_0,\zeta,\cdot)\parallel_{2;C_{q,\delta}} e^{-t'/4}
\end{align*}
where in the last inequality we use the Lemma \ref{fl2norminequality}. Now using
Proposition \ref{poissonkernelproperties} we get a supremum bound uniform in
$q$, namely we can bound the preceding inequality above by
\begin{align*}
\sup_{\substack{\zeta \in \partial C_{q,\delta}\\ x \in
C_{q,\varepsilon} }} P_{q, \delta}(t_0,\zeta,x) \cdot \text{vol}(C_{q,\delta})^{1/2}
e^{-(t-t_0)/4} \le c(t_0) e^{-t/4}
\end{align*}
which completes the proof.
\end{proof}

\begin{thm}\label{uniformlongtimeasymptotics}
Let $M_{q}$ be an elliptically degenerating family of compact
or non-compact hyperbolic Riemann surfaces of finite volume which
converge to the non-compact hyperbolic surface $M_{\infty}$. Let $\alpha$ be given according to  the Definition \ref{alphatruncatedtrace} above.
Then for any $c < \alpha$, there exist a constant $C$ such that the
bound
\begin{align*}
|\text{\rm HTr}K_{M_{q}}^{(\alpha)}(t) + \text{\rm ETr}K_{M_{q}}^{(\alpha)}(t) -
\text{\rm DTr}K_{M_{q}}(t)| \le C e^{-ct}
\end{align*}
holds for all $t\ge 0$ and uniformly in $q$.
\end{thm}
\begin{proof}
Our proof
consists of analyzing the three integrals labeled (I), (II), and
(III) coming from Theorem \ref{pointwiseanduniformconvergence}. For the
uniformity of integral (I) we use Lemma
\ref{truncatedkernellimitlemma} whereas for the integral (III) we
use Proposition \ref{thmintegraloverdifferencecusps}. In both cases, we obtain $O(e^{-ct})$. It remains to
consider integral (II). In this direction, let
$\{\lambda_{n,q}\}$ be the eigenvalues on $M_{q}$ which
converge to the eigenvalues on $M_{\infty}$ which are less than 1/4.
Let $\{\phi_{n,q}\}$ denote the corresponding eigenfunctions.
Since the small eigenvalues and corresponding eigenfunctions converge through degeneration (see Section \ref{Convergence of small eigenfunctions}), we have
that the sum
\begin{align*}
\sum_{\lambda_{n,q}<1/4} e^{-t\lambda_{n,q}}
\phi_{n,q}(x) \phi_{n,q}(y)
\end{align*}
varies continuously in $q$ as well as on the limiting surface
$M_{\infty}$. Fix $\delta <1/2$ and let $0 < \varepsilon < \delta$.
For $x,y \in C_{q,\delta}$ and $t>0$ consider the decomposition
\begin{align*}
K_{M_{q}}(t,x,y) - K_{C_{q}}(t,x,y) = u(t,x,y) + v(t,x,y)
+ \sum_{n=1}^{N} e^{-t\lambda_{n,q}} \phi_{n,q}(x)
\phi_{n,q}(y)
\end{align*}
where $u$ and $v$ are solutions to the homogeneous heat equation in
both $(t,x)$ and $(t,y)$ such that $u$ vanishes on $\partial
C_{q,\delta}$ and has appropriate initial values, whereas $v$
vanishes at the initial values and has appropriate values on
$\partial C_{q,\delta}$.

\hskip 0.2in Let $\alpha < 1/4$ be such that $M_{\infty}$ has no
eigenvalue in the interval $(\alpha,1/4)$. From above, we can write
\begin{align*}
K_{M_{q}}^{(\alpha)}(t,x,y) - K_{C_{q}}(t,x,y) = u(t,x,y)
+ v(t,x,y)
\end{align*}
and denote the left hand side of the equation above by
$D_{q}^{(\alpha)}(t,x,y)$.

\hskip 0.2in We will analyze first the function $v(t,x,y)$. On the one hand, the function $v$ satisfies the heat equation subject to the Dirichlet condition in the variables $(t,x)$.  For $\zeta \in \partial C_{q,\delta}$, we have that
\begin{align*}
v(t,\zeta,y) = D_{q}^{(\alpha)}(t,\zeta,y) - u(t,\zeta,y) = D_{q}^{(\alpha)}(t,\zeta,y) = f_1(t,\zeta)
\end{align*}
since the $u$ vanishes for $\zeta \in \partial C_{q,\delta}.$ Recalling the Remark \ref{Poissonkernelremark}, we then realize $v$  as a Poisson kernel, namely
\begin{align*}
v(t,x,y) &= \int_{0}^{t}\int_{\partial
C_{q,\delta}}
P_{q,\delta}(t-\sigma,x,\zeta) f_1(\sigma,\zeta) d\rho(\zeta) d\sigma\\
&= \int_{0}^{t}\int_{\partial
C_{q,\delta}}
P_{q,\delta}(t-\sigma,x,\zeta)  D_{q}^{(\alpha)}(\sigma,\zeta,y)d\rho(\zeta) d\sigma.
\end{align*}
On the other hand, $v$ satisfies the heat equation in the variables $(t,y)$. Noting that for $\xi \in \partial C_{q,\delta}$, we have
\begin{align*}
v(t,x,\xi)
= \int_{0}^{t}\int_{\partial
C_{q,\delta}}
P_{q,\delta}(t-\sigma,x,\zeta)  D_{q}^{(\alpha)}(\sigma,\zeta,\xi)d\rho(\zeta) d\sigma = f_2(t,\xi),
\end{align*}
we use the Remark \ref{Poissonkernelremark} again to get
\begin{align*}
v(t,x,y) &= \int_{0}^{t}\int_{\partial
C_{q,\delta}}
P_{q,\delta}(t-\tau,\xi,y) f_2(\tau,\xi) d\rho(\xi) d\tau\\
&= \int_{0}^{t}\int_{\partial
C_{q,\delta}}\int_{0}^{\tau} \int_{\partial C_{q,\delta}}
P_{q,\delta}(t-\tau,\xi,y)P_{q,\delta}(\tau-\sigma,x,\zeta)
D_{q}^{(\alpha)}(\sigma,\zeta,\xi) d\rho(\zeta) d\sigma d\rho(\xi)   d\tau.
\end{align*}
Then we look at the following integral
\begin{align*}
\int_{C_{q,\varepsilon}}v(t,x,x) d\mu(x).
\end{align*}
Using the supremum norm on the $D_{q}^{\alpha}$ as in Lemma
\ref{truncatedkernelsupbound}, the $L^2$-norm of the Poisson kernel
as in Lemma \ref{poissonkernellongtimebound}, and the Cauchy-Schwarz
inequality, we obtain the bound
\begin{align*}
\int_{C_{q,\varepsilon}}v(t,x,y) d\mu(x) \le C
\int_{0}^{t}\int_{o}^{\tau}
\exp{(-(t-\tau)/4)}\exp{(-(\tau-\sigma)/4)}\exp{(-c\sigma)}d\sigma
d\tau
\end{align*}
which is clearly $O(e^{-ct})$.

\hskip 0.2in Next we will look at the function $u(t,x,y)$. We
will express  $u$ in terms of the Dirichlet heat kernel on the domain
$C_{q,\epsilon}$. For fixed $y$, the function $u$
satisfies the heat equation in $(t,x)$ with zero boundary condition
and initial data given by
\begin{align}\label{u initialdata}
g(x,y) =\sum_{\lambda_{n,q}<1/4} \phi_{n,q}(x)
\phi_{n,q}(y).
\end{align}
This gives the following integral representation
\begin{align}\label{u integralrepresentation}
u(t,x,y)=\int_{C_{q,\delta}}
K_{q,\delta}^{D}(t,z,x)g(z,y)d\mu(z).
\end{align}
For fixed $x$, the integral representation in (\ref{u
integralrepresentation}) is a solution to the heat equation on
$C_{q,\delta}$ in $(t,y)$.  Then consider the function $u(\tau+t,x,y)$ which satisfies the heat equation in $(\tau, y)$, vanishes for $y\in \partial C_{q,\delta}$, and has initial value $u(t,x,y)$.
As such, we can write,
\begin{align}\label{u doubleintegralrepresentation}
\notag u(\tau + t,x,y)&= \int_{C_{q,\delta}}
K_{q,\delta}^{D}(\tau,w,y)u(t,x,w)d\mu(w)\\
&= \int_{C_{q,\delta}} K_{q,\delta}^{D}(\tau,w,y)\Bigg(
\int_{C_{q,\delta}}
K_{q,\delta}^{D}(t,z,x)g(z,w)d\mu(z)\Bigg)d\mu(w).
\end{align}
Using (\ref{u initialdata}) and (\ref{u
doubleintegralrepresentation}), we can further write
\begin{align}\label{u sumintegrals}
\notag u(\tau + t,x,y) &= \sum_{\lambda_{n,q}<1/4}
\int_{C_{q,\delta}} \int_{C_{q,\delta}}
K_{q,\delta}^{D}(\tau,w,y) K_{q,\delta}^{D}(t,z,x)
\phi_{n,q}(z) \phi_{n,q}(w) d\mu(z)d\mu(w)\\
&= \sum_{\lambda_{n,q}<1/4} \int_{C_{q,\delta}}
K_{q,\delta}^{D}(\tau,w,y)\phi_{n,q}(w)d\mu(w)
\int_{C_{q,\delta}} K_{q,\delta}^{D}(t,z,x)
\phi_{n,q}(z) d\mu(z).
\end{align}
In (\ref{u sumintegrals}), we set $x=y$ and $t=\tau$, so that we can
write
\begin{align}\label{u sumintegralssquare}
u(2t,x,x) &= \sum_{\lambda_{n,q}<1/4}
\Bigg(\int_{C_{q,\delta}} K_{q,\delta}^{D}(t,z,x)
\phi_{n,q}(z) d\mu(z)\Bigg)^2
\end{align}
where we used one variable of integration as opposed to two.
The next step is to analyze the integral
\begin{align*}
\int_{C_{q,\delta}} u(t,x,x) d\mu(x).
\end{align*}

Consider an  complete orthonormal system of eigenfunctions $\{\psi_m(x)\}$ of the Dirichlet
problem on $C_{q,\delta}$. Then, the heat kernel $K_{q,\delta}^{D}$ has the following expression
\begin{align*}
K_{q,\delta}^{D}(t,z,x) = \sum_{m=0}^{\infty} e^{-\lambda_m t} \psi_m(z) \psi_m(x).
\end{align*}
This allows us to write
\begin{align*}
\int_{C_{q,\delta}} K_{q,\delta}^{D}(t,x,z)
\phi_{n,q}(z) d\mu(z)= \sum_{m=1}^{\infty} a_{n,m}e^{-\lambda_m t}\psi_m(x),
\end{align*}
where the coefficients $a_{n,m}$ are given by
\begin{align*}
a_{n,m} = \int_{C_{q,\delta}} \psi_m(z) \phi_{n,q}(z) d\mu(z).
\end{align*}

By the positivity of (\ref{u sumintegralssquare}), we can write the
inequality
\begin{align*}
0 \le F_{\varepsilon}(t) = \int_{C_{q,\varepsilon}}u(t,x,x)
d\mu(x) \le \int_{C_{q,\delta}}u(t,x,x) d\mu(x) \le
F_{\delta}(t).
\end{align*}
It suffices to show that $F_{\delta}(t) \le C \exp{(-t/4)}$ for some
constant $C$ independent of $q$. Notice that we have the
equality
\begin{align}\label{f delta}
F_{\delta}(t) = \int_{C_{q,\delta}}u(t,x,x) d\mu(x) =
\sum_{\lambda_{n,q}<1/4} \sum_{m=1}^{\infty} a_{n,m}^2
e^{-\lambda_m t}.
\end{align}
Clearly, the above function $F_{\delta}(t)$ is monotone decreasing
in $t$. Let $N$ denote the integer that bounds the number of
eigenvalues on $M_{q}$ which are less than 1/4; such a
universal choice is possible by Buser's theorem (see page 251 of
\cite{Ch 84}). From (\ref{u sumintegrals}), we deduce that
$F_{\delta}(0)\le N$. Since $\lambda_m \ge 1/4$ for all $m$, we have
from (\ref{f delta}) that $F_{\delta}'(t) \le (-1/4) F_{\delta}(t)$.
We then integrate and obtain
\begin{align*}
F_{\delta}(t) \le F_{\delta}(0) e^{-t/4}.
\end{align*}
This in turn gives
\begin{align*}
0 \le F_{\varepsilon}(t) \le F_{\delta}(t) \le F_{\delta}(0)
e^{-t/4} \le N e^{-t/4}.
\end{align*}
Thus we have that
\begin{align*}
\int_{C_{q,\delta}} [ u(t,x,x) + v(t,x,x) ]\,\, d\mu(x) = O(e^{-ct})
\end{align*}
which means that integral (II) has the above bound.
\end{proof}

\textbf{Acknowledgments.}
The first author (DG) acknowledges support from a PSC-CUNY
grant. The second author (JJ) acknowledges support from grants from the
NSF and PSC-CUNY. We would also like to  thank the referee for providing a detailed list of comments. They have helped us improve the presentation of paper as well as
address some of the finer points.

\addcontentsline{toc}{section}{Bibliography}
\renewcommand\refname{Bibliography}

\vspace{5mm}\noindent
Daniel Garbin \\
Department of Mathematics and Computer Science \\
Queensborough Community College\\
222-05 56th Avenue \\
Bayside, NY 11364 \\
U.S.A. \\
e-mail: dgarbin@qcc.cuny.edu

\vspace{5mm} \noindent
Jay Jorgenson \\
Department of Mathematics \\
The City College of New York \\
Convent Avenue at 138th Street \\
New York, NY 10031\\
U.S.A. \\
e-mail: jjorgenson@mindspring.com


\begin{thebibliography}{xxxxxxxx} 
\bibitem[AJS 09]{AJS 09}
AVDISPHAI\'C, ~M, JORGENSON, J., and SMAJLOVI\'C, L: \emph{Asymptotic behavior of the Selberg zeta
functions for degenerating families of hyperbolic manifolds.} Comm. Math. Phys. \textbf{310} (2012), no. 1, 217--236. 

\bibitem[CdV 83]{CdV 83}
COLIN de VERDIERE, ~Y.: \emph{Pseudo-laplaciens. II.} (French) [Pseudo-Laplacians. II] Ann. Inst. Fourier (Grenoble) \textbf{33} (1983), no. 2, 87--113.

\bibitem[Ch 84]{Ch 84}
CHAVEL, ~I.: \emph{Eigenvalues in Riemannian Geometry}. Including a chapter by Burton Randol. With an appendix by Jozef Dodziuk. Pure and Applied Mathematics, \textbf{115}. Academic Press, Inc., Orlando, FL, (1984), xiv+362 pp.
\bibitem[DJ 98]{DJ 98}
DODZIUK, ~J., and JORGENSON, J.: \emph{Spectral asymptotics on
degenerating hyperbolic 3-manifolds.} Mem. Amer. Math. Soc. \textbf{135} (1998), no. 643, viii+75 pp.
\bibitem[Fr 09]{Fr 09}
FREIXAS I MONTPLET, ~G: \emph{An arithmetic Riemann-Roch theorem for
pointed stable curves}. Ann. Sci. \'Ec. Norm. Sup\'er (4) \textbf{42}
(2009), no. 2, 335--369.
\bibitem[FvP 11]{FvP 11}
FREIXAS I MONTPLET, ~G., and VON PIPPICH, A.:  In preparation.

\bibitem[GJ 16]{GJ 16}
GARBIN, ~D., and JORGENSON, J.: \emph{Spectral asymptotics
on sequences of elliptically degenerating Riemann surfaces.} In preparation.
\bibitem[GJM 08]{GJM 08}
GARBIN, ~D., JORGENSON, J., and MUNN, M.: \emph{On the appearance of
Eisenstein series through degeneration}. Comment. Math. Helv. \textbf{83}
(2008), no. 4, 701--721.

\bibitem[GvP 09]{GvP 09}
GARBIN, ~D., and VON PIPPICH, A.: \emph{On behavior of Eisenstein series
through elliptic degeneration}. Commun. Math. Physics \textbf{292} (2009), no. 2, 511--528.

\bibitem[GT 83]{GT 83} 
GILBARG, ~D. A. and TRUDINGER, N. S.: \emph{Elliptic partial differential equations of second order.} Second edition. Grundlehren der Mathematischen Wissenschaften [Fundamental Principles of Mathematical Sciences], \textbf{224}. Springer-Verlag, Berlin, (1983), xiii+513 pp.

\bibitem[Ha 09]{Ha 09}
HAHN, ~T.: \emph{An arithmetic Riemann-Roch theorem for metrics with cusps.} Ph.D. dissertation,
Humboldt University, (2009).

\bibitem[He 76]{He 76}
HEJHAL, ~D. A.: \emph{The Selberg Trace Formula for PSL(2,R), Vol. 1.} Lecture Notes in Mathematics, \textbf{548}. Springer-Verlag, Berlin-New York, (1976), vi+516 pp.

\bibitem[He 83]{He 83}
HEJHAL, ~D. A.: \emph{The Selberg Trace Formula for PSL(2;R), Vol. 2.} Lecture Notes in Mathematics, \textbf{1001}. Springer-Verlag, Berlin, (1983), viii+806 pp.

\bibitem[He 90]{He 90}
HEJHAL, ~D. A.: \emph{Regular b-groups, degenerating Riemann
surfaces, and spectral theory}.  Mem. Amer. Math. Soc. 88 (1990), no. \textbf{437}, iv+138 pp.

\bibitem[He 92]{He 92}
HEJHAL, ~D. A.: \emph{Eigenvalues of the Laplacian for Hecke
Triangle Groups}. Mem. Amer. Math. Soc. 97 (1992), no. \textbf{469}, vi+165 pp.

\bibitem[HJL 95]{HJL 95}
HUNTLEY, ~J., JORGENSON, J. and LUNDELIUS, R.: \emph{Continuity of
small eigenfunctions on degenerating Riemann surfaces with
hyperbolic cusps}. Bol. Soc. Math. Mexicana (3) \textbf{1} (1995), no. 2,
119--125.

\bibitem[HJL 97]{HJL 97}
HUNTLEY, ~J., JORGENSON, J. and LUNDELIUS, R.: \emph{On the
asymptotic behavior of counting functions associated to degenerating
hyperbolic Riemann surfaces.} J. Func. Analysis \textbf{149} (1997), no. 1,
58--82.

\bibitem[Ji 94]{Ji 94}
JI, ~L.: \emph{Degeneration of pseudo-Laplace operators for hyperbolic Riemann surfaces.}
Proc. Amer. Math. Soc. \textbf{121} (1994), no. 1, 283--293. 

\bibitem[JoLu 95]{JoLu 95}
JORGENSON, ~J. and LUNDELIUS, R.: \emph{Convergence of the heat
kernel and the resolvent kernel on degenerating hyperbolic Riemann
surfaces of finite volume.} Quaestiones Mathematicae \textbf{18}
(1995), no. 4, 345--363.

\bibitem[JoLu 96]{JoLu 96}
JORGENSON, ~J. and LUNDELIUS, R.: \emph{Continuity of relative
hyperbolic spectral theory through metric degeneration.} Duke Math.
J. \textbf{84} (1996), no. 1, 47--81.

\bibitem[JoLu 97a]{JoLu 97a}
JORGENSON, ~J. and LUNDELIUS, ~R.: \emph{Convergence of the
normalized spectral counting functions on degenerating hyperbolic
Riemann surfaces of finite volume.} J. Func Analysis \textbf{149}
(1997), no. 1, 28--57.

\bibitem[JoLu 97b]{JoLu 97b}
JORGENSON, ~J., and LUNDELIUS, R.: \emph{A regularized heat trace
for hyperbolic Riemann surfaces of finite volume.} Comment. Math.
Helv. \textbf{72} (1997), no. 4, 636--659.

\bibitem[Ju 95]{Ju 95}
JUDGE, ~C.:  \emph{On the existence of Maass cusp forms on hyperbolic surfaces with cone points.} J. Amer. Math. Soc. \textbf{8} (1995), no. 3, 715--759.

\bibitem[Ju 98]{Ju 98}
JUDGE, ~C.: \emph{Conformally coverting cusps to cones.} Conf. Geom.
Dyn. \textbf{2} (1998), 107--113 (electronic).

\bibitem[Ku 73]{Ku 73}
KUBOTA, ~T: \emph{Elementary theory of Eisenstein series.} Kodansha
Ltd., Tokyo; Halsted Press [John Wiley \& Sons], New
York-London-Sydney, (1973), xi+110 pp.

\bibitem[LP 76]{LP 76}
LAX, ~P. and PHILLIPS, R.:  \emph{Scattering theory for automorphic functions.} Annals of Mathematics Studies, No. \textbf{87}. Princeton Univ. Press, Princeton, N.J., (1976), x+300 pp.


\bibitem[Lu 93]{Lu 93}
LUNDELIUS, ~R.: \emph{Asymptotics of the determinant of the
Laplacian on hyperbolic surfaces of finite volume.} Duke Math. J.
\textbf{71} (1993), no.1, 212--242.

\bibitem[McK 72]{McK 72}
McKEAN, ~H.: \emph{Selberg's trace formula as applied to a compact
Riemann surface.} Comm. Pure and Appl. Math. \textbf{25} (1972),
225--246.

\bibitem[McO 88]{McO 88}
McOWEN, ~R.: \emph{Point singularities and conformal metrics on Riemann surfaces.} 
Proc. Amer. Math. Soc. \textbf{103} (1988), no. 1, 222--224.

\bibitem[Mu 83]{Mu 83}
M\"ULLER, ~W.: \emph{Spectral theory for Riemannian manifolds with
cusps and a relted trace formula.} Math. Nachr. \textbf{111} (1983),
197--288.
\bibitem[Ru 74]{Ru 74}
RUDIN, ~W.: \emph{Real and complex analysis.} Second edition. McGraw-Hill Series in Higher Mathematics. McGraw-Hill Book Co., New York-D\"usseldorf-Johannesburg, (1974), xii+452 pp.

\bibitem[Sch 06]{Sch 06}
SCHULZE, M.: \emph{On the resolvent of the Laplacian on functions for degenerating surfaces
of finite geometry} J. Funct. Anal. \textbf{236} (2006), no. 1,  120--160.

\bibitem[Se 56]{Se 56}
SELBERG, ~A: \emph{Harmonic Analysis and discontinuous groups in
weakly symmetric Riemannian spaces with applications to Dirichlet
series.} J. Indian Math. Soc. B. \textbf{20} (1956), 47--87.

\bibitem[Sh 87]{Sh 87}
SHUBIN, ~M.: \emph{Pseudodifferential Operators and Spectral
Theory.} Translated from the Russian by Stig I. Andersson. Springer Series in Soviet Mathematics. Springer-Verlag, Berlin, (1987), x+278 pp.

\bibitem[Ts 97]{Ts 97}
TSUZUKI, ~M.: \emph{Elliptic factors of Selberg zeta functions.}
Duke Math. J. \textbf{88} (1997), no. 1, 29--75.

\bibitem[vP 10]{vP 10}
VON PIPPICH, ~A.: \emph{The arithmetic of elliptic Eisenstein series.}
Ph.D. dissertation, Humboldt University, (2010).

\bibitem[Wi 41]{Wi 41}
WIDDER, ~D.: \emph{The Laplace Transform.}  Princeton Mathematical Series, v. 6. Princeton University Press, Princeton, N. J., (1941), x+406 pp.

\bibitem[Wo 87]{Wo 87}
WOLPERT, ~S. A.: \emph{Asymptotics of the spectrum and the Selberg
zeta function on the space of Riemann surfaces.} Comm. Math. Phys.
\textbf{112} (1987), no. 2, 283--315.
\end{thebibliography}
\end{document}